\documentclass[reqno,11pt]{amsart}
\usepackage[utf8]{inputenc}
\DeclareUnicodeCharacter{FF0C}{,}
\usepackage{mathrsfs}
\usepackage{}

\def\subjclass#1{{\renewcommand{\thefootnote}{}
\footnote{\emph{Mathematics Subject Classification (2020):} #1}}}
\usepackage{amsfonts}
\usepackage{amssymb}

\date{\today}
\hoffset=- 2cm \voffset=- 0cm 
 \theoremstyle{plain}
\newtheorem{Thm}{Theorem}

\newtheorem{Prop}[Thm]{Proposition}
\newtheorem{Rem}[Thm]{Remark}
\newtheorem{Lem}[Thm]{Lemma}

\usepackage{amssymb}

\usepackage{color}

\def\0{\mathbf 0}

\def\R{\mathbb R}

\errorcontextlines=0 \numberwithin{equation}{section}
\numberwithin{Thm}{section}
\allowdisplaybreaks

\begin{document}
\large

\title[Damped wave-type magnetohydrodynamic equations]
{Stabilizing effect of a background magnetic field on the 2D damped wave-type MHD equations}


\author[]{Zhi Chen}
\address{Zhi Chen: School of Mathematics and Statistics, Anhui Normal University, Wuhu 241002, People's Republic of China }
\email{zhichenmath@ahnu.edu.cn}%

\author[]{Mingwen Fei}
\address{Mingwen Fei: School of Mathematics and Statistics, Anhui Normal University, Wuhu 241002, People's Republic of China}
\email{mwfei@ahnu.edu.cn}%

\author[]{Hongxia Lin }
\address{Hongxia Lin: School of Mathematical Sciences，Chengdu University of Technology, Chengdu, 610059, People's Republic of China}
\email{linhongxia13@cdut.edu.cn}

\author[]{Jiahong Wu }
\address{Jiahong Wu: Department of Mathematics, University of Notre Dame, Notre Dame, IN 46556, USA}
\email{jwu29@nd.edu}

\author[]{Qian Zu}
\address{Qian Zu: School of Mathematics and Statistics, Anhui Normal University, Wuhu 241002, People's Republic of China}
\email{qianzu@ahnu.edu.cn}



\begin{abstract}
	The stabilizing effect of a background magnetic field on electrically conducting fluids has been rigorously established for the standard MHD equations. This paper extends this theory to the more physically accurate damped wave-type MHD equations, where the induction equation is hyperbolic-parabolic and the velocity field has only vertical damping with no dissipation. These two features make the stability analysis harder than in the standard MHD setting. To overcome these difficulties,
	we design an energy functional exploiting the anisotropic structure, and
	discover a remarkable cancellation between the two most dangerous nonlinear terms by
	exploiting the full algebraic structure of the coupled system. As a consequence, we
	prove that any small perturbation near the background magnetic field is globally stable
	and establish optimal decay rates consistent with the 2D heat equation. To the
	best of our knowledge, this is the first rigorous stability result for the damped
	wave-type MHD equations near a background magnetic field.
\end{abstract}

\maketitle

\section{Introduction}

The classical incompressible magnetohydrodynamic (MHD) equations govern the motion of
electrically conducting fluids and are widely used in plasma physics, astrophysics, and
geophysics. These equations consist of the incompressible Navier-Stokes equations coupled
with the induction equation for the magnetic field $B$:
\begin{align}\label{MHD}
	\left\{
	\begin{array}{lll}
		\partial_t u + (u\cdot\nabla)u + \nabla p = (B\cdot\nabla)B + \nu\,\Delta u, \\[5pt]
		\partial_t B - \eta\,\Delta B + (u\cdot\nabla)B = (B\cdot\nabla)u, \\[5pt]
		\nabla\cdot u = \nabla\cdot B = 0.
	\end{array}
	\right.
\end{align}
However, the induction equation in \eqref{MHD} is derived under the quasi-static
approximation (that is, Amp\`{e}re's original law without Maxwell's correction), which
neglects the displacement current $\varepsilon_0\mu_0\partial_t E$ in Amp\`{e}re's law.
Retaining this term and combining all four of Maxwell's equations with Ohm's law for a
conducting fluid leads to a more accurate induction equation for the magnetic field $B$:
\begin{align}\label{induction-full}
	\gamma\,\partial_{tt} B + \partial_t B - \eta\,\Delta B
	+ (u\cdot\nabla)B = (B\cdot\nabla)u,
\end{align}
where the coefficient $\gamma = \varepsilon_0\mu_0 = 1/c^2$ is the product of the
permittivity and permeability of the medium, and $c$ is the speed of light in the medium.
In the Newtonian (non-relativistic) framework, $c$ is enormous relative to the fluid
velocity, so $\gamma = 1/c^2$ is an extremely small positive parameter. The coupled system consisting of a Navier-Stokes-type momentum equation and the induction equation
\eqref{induction-full} is called the damped wave-type MHD equations, or simply the
MHD-wave system,
\begin{align}\label{MHD-wave}
	\left\{
	\begin{array}{lll}
		\partial_t u + (u\cdot\nabla)u + \nabla p = (B\cdot\nabla)B + \nu\,\Delta u, \\[5pt]
		\gamma\,\partial_{tt} B + \partial_t B - \eta\,\Delta B + (u\cdot\nabla)B = (B\cdot\nabla)u, \\[5pt]
		\nabla\cdot u = \nabla\cdot B = 0.
	\end{array}
	\right.
\end{align}
We refer the reader to \cite{Comissiong,JWX,MNO, XY} for a detailed derivation from Maxwell's equations.
When $\gamma \to 0$, the term $\gamma\,\partial_{tt}B$ vanishes and \eqref{MHD-wave}
formally reduces to the classical MHD system \eqref{MHD}. Thus the MHD-wave system is a
physically more accurate model, with the standard MHD equations arising as its formal
limit in the Newtonian regime.

\vskip .1in
Despite its physical relevance, the presence of $\gamma\,\partial_{tt}B$ makes the
mathematical analysis of the MHD-wave system substantially more difficult than that of the standard MHD equations. In the classical MHD system the induction equation is parabolic,
so the magnetic field enjoys instantaneous smoothing. In contrast, the MHD-wave induction
equation \eqref{induction-full} is hyperbolic-parabolic in nature, and the term
$\gamma\,\partial_{tt}B$ acts as a ``bad" term in energy estimates: it contributes a
second-order time derivative that must be controlled without access to additional
dissipation. For instance, even in the two-dimensional case and for general initial data
(without any smallness assumption), it is unknown at present whether the $L^2$-norm of
the solution remains bounded for all time or may blow up in finite time. As a consequence,
many fundamental mathematical questions such as global well-posedness, long-time behavior, and
stability remain open for the MHD-wave system, even in the two-dimensional setting.

\vskip .1in
In this paper we focus on the following two-dimensional MHD-wave system with only
	vertical velocity damping,
\begin{align}\label{Equ}
	\left \{
	\begin{array}{lll}
		\partial_t u + \mu(0,u_2)^T + (u\cdot\nabla)u + \nabla p = (B\cdot\nabla)B, \\[5pt]
		\gamma\,\partial_{tt}B + \partial_t B - \eta\,\Delta B + (u\cdot\nabla)B = (B\cdot\nabla)u, \\[5pt]
		\nabla\cdot u = \nabla\cdot B = 0,
	\end{array}
	\right.
\end{align}
where $u=(u_1,u_2)^T$ and $B=(B_1,B_2)^T$ denote the velocity and magnetic fields,
$p$ is the scalar pressure, $\mu>0$ is the damping coefficient, $\eta>0$ is the magnetic
diffusivity, and $\gamma>0$ is the small parameter described above. The velocity equation
in \eqref{Equ} contains only vertical damping $\mu(0,u_2)^T$, meaning that horizontal
velocity is entirely undamped. This anisotropic structure arises naturally in physical
situations such as the dynamics of a conducting fluid in a strong external magnetic field
aligned with the vertical direction, where the Lorentz force predominantly damps motion
perpendicular to the field lines while leaving horizontal degrees of freedom essentially
free. Such anisotropic damping mechanisms appear in studies of magnetoconvection,
liquid-metal MHD, and the dynamics of the Earth's outer core; see, e.g.,
\cite{AMS,A,GBM,GB}.

\vskip .1in
For any constant $\alpha > 0$, the state $u^{(0)}\equiv \mathbf{0}$, $B^{(0)}=(0,\alpha)$ is an
exact steady solution of \eqref{Equ}, representing a uniform vertical background magnetic
field. A fundamental physical observation, confirmed by numerous experiments and numerical
simulations \cite{AMS,A,AL,GBM,GB}, is that a background magnetic field exerts a
stabilizing effect on electrically conducting fluids: perturbations tend to be suppressed
and the perturbed flow decays back toward the equilibrium. Setting $b = B - B^{(0)}$, the
perturbation $(u, b)$ around the background state satisfies
\begin{align}\label{type-wave-1}
	\left \{
	\begin{array}{lll}
		\partial_t u + \mu(0,u_2)^T + (u\cdot\nabla)u + \nabla p
		= (b\cdot\nabla)b + \alpha\,\partial_2 b, \\[5pt]
		\gamma\,\partial_{tt}b + \partial_t b - \eta\,\Delta b + (u\cdot\nabla)b
		= (b\cdot\nabla)u + \alpha\,\partial_2 u, \\[5pt]
		\nabla\cdot u = \nabla\cdot b = 0,
	\end{array}
	\right.
\end{align}
supplemented with initial data
\[
u(x,0) = u_0(x),\quad b(x,0) = b_0(x),\quad (\partial_t b)(x,0) = a_0(x).
\]

\vskip .1in
The stabilizing effect of a background magnetic field has been studied extensively for the
standard MHD equations ($\gamma=0$). In the work of Boardman, Lin and Wu \cite{BLW},
which is the most closely related MHD counterpart to the present paper, it was proved
rigorously that for the 2D inviscid and resistive MHD equations with partial velocity damping,
any small perturbation near a background vertical magnetic field decays to zero. A related
stability result for the 2D inviscid resistive system near a horizontal background field
was obtained by Ji and Wu \cite{JiWu}. Subsequent works have extended this stabilization
theory in several directions: Feng-Hafeez-Wu \cite{FHW} treated 2D MHD with only
vertical velocity dissipation and magnetic damping; Li-Wu-Xu \cite{LiWuXu} studied 2D
MHD with vertical velocity dissipation and horizontal magnetic diffusion;
Lai-Wu-Zhang \cite{LaWZ1} considered 2D MHD with mixed partial damping; and the 3D
setting with various partial dissipation configurations was investigated by
Abidi-Zhang \cite{AbidiZhang}, He-Xu-Yu \cite{HeXuYu}, Wu-Zhu \cite{WZ1},
Lin-Wu-Zhu \cite{LWZ1} and Lai-Wu-Zhang-Zhao \cite{LWZZ}, among others.
Global small solutions and decay estimates near a background magnetic field for 2D MHD
with only velocity dissipation were obtained by Lin-Xu-Zhang \cite{LXZ} and
Ren-Wu-Xiang-Zhang \cite{RWXZ}, and for 3D non-resistive MHD on periodic domains
by Pan-Zhou-Zhu \cite{PZZ}. The 2D MHD system with a velocity damping term was
studied by Wu-Wu-Xu \cite{WWX}. We refer the reader also to
\cite{CW,CZZ,LaWZ,LJWL,LWZ, WZ, ZZ} and the references therein.

\vskip .1in
There is comparatively little literature on the MHD-wave system itself. When the damping
term $(0,u_2)^T$ is replaced by full Laplacian dissipation $\Delta u$, Matsui-Nakasato-Ogawa
\cite{MNO} proved small-data global well-posedness and the singular limit $\gamma\to 0$
in Fourier-Sobolev spaces. Ji-Wu-Xu \cite{JWX} obtained global well-posedness of the
2D MHD-wave equations in a critical Sobolev setting when $\gamma$ and the initial
data satisfy a suitable joint smallness condition, and also established the singular limit
$\gamma \to 0$ recovering the standard MHD system. Xie-Yu \cite{XY} established
large-time behavior by spectral analysis, and Sun-Wang \cite{SW} proved global existence
and uniqueness under full velocity dissipation. Since the MHD-wave system is a physically
more accurate model than the standard MHD equations, it is natural to ask whether the
stabilizing effect of the background magnetic field persists in this more accurate setting.
This is the primary motivation of the present work.

\vskip .1in
Compared with the standard MHD case, the stability problem for \eqref{type-wave-1} is
considerably more delicate, mainly for two reasons. First, the term
$\gamma\,\partial_{tt}b$ changes the induction equation from parabolic to
hyperbolic-parabolic, giving rise to time derivatives of $b$ in energy estimates
that cannot be directly absorbed by the available dissipation. This forces the introduction
of a substantially more elaborate energy functional than what is used in the standard MHD
theory. Second, and more fundamentally, consider the linearization of \eqref{type-wave-1}
around the background state:
\begin{align}\label{linearized system}
	\left\{
	\begin{array}{lll}
		\partial_t u +\mu (0,u_2)^T = \alpha\partial_2 b, \\[5pt]
		\gamma \partial_{tt} b + \partial_t b -\eta \Delta b = \alpha\partial_2 u.
	\end{array}
	\right.
\end{align}
Differentiating \eqref{linearized system} in time and making suitable substitutions,
one can convert \eqref{linearized system} into a system in which $u$ and $b$ satisfy
decoupled third-order-in-time equations:
\begin{align}\label{linearized system1}
	\left\{
	\begin{array}{lll}
		\gamma\partial_{ttt} u + \partial_{tt} u + \mu\gamma\partial_{tt}(0,u_2)^T
		- \eta\Delta\partial_t u + \mu\partial_t(0,u_2)^T
		- \mu\eta\Delta(0,u_2)^T - \alpha^2\partial_{22}u = 0,\\[5pt]
		\gamma\partial_{ttt} b + \partial_{tt} b + \mu\gamma\partial_{tt}(0,b_2)^T
		- \eta\Delta\partial_t b + \mu\partial_t(0,b_2)^T
		- \mu\eta\Delta(0,b_2)^T - \alpha^2\partial_{22}b = 0.
	\end{array}
	\right.
\end{align}
The first components $u_1$ and $b_1$ share the same characteristic polynomial
\[
\gamma\lambda_1^3 + \lambda_1^2 + \eta|\xi|^2\lambda_1 + \alpha^2\xi_2^2 = 0,
\]
while the second components $u_2$ and $b_2$ are governed by a structurally different cubic,
\[
\gamma\lambda_2^3 + (1+\mu\gamma)\lambda_2^2
+ (\mu + \eta|\xi|^2)\lambda_2 + \mu\eta|\xi|^2 + \alpha^2\xi_2^2 = 0,
\]
in which the additional terms $\mu\gamma\lambda_2^2$, $\mu\lambda_2$, and $\mu\eta|\xi|^2$
reflect the partial damping acting on the second velocity component. For the standard MHD
equations the corresponding polynomials are quadratic, whose roots are readily analyzed.
For the cubics here, characterizing the behavior of the roots across different frequency
regimes is substantially harder, and the Fourier-side approach via Duhamel's principle
employed in \cite{BLW} does not extend to the present setting.

\vskip .1in
To overcome these difficulties, our strategy proceeds as follows. We first observe that the vertical damping $\mu(0,u_2)^T$ in the velocity equation provides a weak, anisotropic dissipation mechanism. To see this precisely, we apply the Leray-Helmholtz projection operator $\mathbb{P}:= I - \nabla\Delta^{-1}\nabla\cdot$ to the velocity equation in
\eqref{type-wave-1}. Using the identity
\begin{align*}
	\mathbb{P}(0,u_2)^T
	= (0,u_2)^T - \nabla\Delta^{-1}\nabla\cdot(0,u_2)^T
	= \Delta^{-1}\partial_1^2 u
	= -\mathcal{R}_1^2 u,
\end{align*}
where $\mathcal{R}_1 = \partial_1(-\Delta)^{-1/2}$ is the first Riesz transform, the projected
velocity equation becomes
\begin{align*}
\partial_t u + \mu\Delta^{-1}\partial_1^2 u = \alpha\partial_2 b + \mathbb{P}\bigl(-(u\cdot\nabla)u + (b\cdot\nabla)b\bigr).
\end{align*}
The term $\mu\,\Delta^{-1}\partial_1^2 u$ acts as a weak dissipation in the
$x_1$-direction (one order weaker than a full Laplacian), but provides no direct control
in the $x_2$-direction. However, by carefully examining the wave structure in the
linearized system \eqref{linearized system}, we observe that the background magnetic
field $\alpha\partial_2 b$ in the coupling term actually generates an effective weak
dissipation in the $x_2$-direction, as reflected in the term $\alpha^2\partial_{22}u$
appearing in \eqref{linearized system1}. This dissipation is one order lower than what
standard diffusion would provide.
These considerations motivate the construction of an energy functional
$\mathcal{E}_0(t)$ consisting of two parts,
\begin{align}\label{E-0}
	\mathcal{E}_0(t) = \mathcal{E}_{01}(t) + \mathcal{E}_{02}(t),
\end{align}
where
\begin{align}
	\mathcal{E}_{01}(t)
	={}& \sup_{0\leq\tau\leq t}\Bigl(\|(u,b)(\tau)\|_{H^3}^2
	+ 2\gamma^2\|\partial_\tau b(\tau)\|_{H^3}^2
	+ 2\gamma\eta\|\nabla b(\tau)\|_{H^3}^2\Bigr)\nonumber\\
	&+ 2\int_{0}^t\Bigl(\mu\|u_2(\tau)\|_{H^3}^2
	+ \eta\|\nabla b(\tau)\|_{H^3}^2
	+ \gamma\|\partial_\tau b(\tau)\|_{H^3}^2\Bigr)\,d\tau,\label{E00}\\
	\mathcal{E}_{02}(t)
	={}& \alpha\int_{0}^t\|\partial_2 u(\tau)\|_{H^2}^2\,d\tau.\label{E01}
\end{align}
The first part $\mathcal{E}_{01}(t)$ captures the basic $H^3$ energy of $(u,b)$ together
with the hyperbolic contributions $\gamma^2\|\partial_t b\|_{H^3}^2$ and
$\gamma\eta\|\nabla b\|_{H^3}^2$ required to handle the term $\gamma\partial_{tt}b$.
The second part $\mathcal{E}_{02}(t)$ encodes the weak $x_2$-directional dissipation
generated by the background field coupling; its inclusion is essential for bounding the
triple products arising from the nonlinearity $(u\cdot\nabla)u$ in the $H^3$ energy
estimates. Our main efforts are then devoted to establishing the energy inequality
\begin{align}\label{E0}
	\mathcal{E}_0(t) \leq C\mathcal{E}_0(0) + C\mathcal{E}_0^{3/2}(0)
	+ C\mathcal{E}_0^{3/2}(t) + C\mathcal{E}_0^2(t).
\end{align}
Applying a standard bootstrapping argument then yields the desired stability result
stated in Theorem~\ref{Thm1} below.
\begin{Thm}\label{Thm1}
	Assume that the initial data $(u_0,b_0,a_0)\in H^3(\mathbb{R}^2)\times H^4(\mathbb{R}^2)\times H^3(\mathbb{R}^2)$ satisfies $\nabla\cdot u_0=\nabla\cdot b_0=\nabla\cdot a_0=0$. Then there exist sufficiently small $\epsilon(\gamma,\mu,\eta),\alpha_0(\gamma,\mu,\eta)>0$ such that, if
	\begin{align}\label{intia}
		\|u_0\|_{H^3(\mathbb{R}^2)}+\|b_0\|_{H^4(\mathbb{R}^2)}+\|a_0\|_{H^3(\mathbb{R}^2)}\leq \epsilon,\quad 0<\alpha\leq\alpha_0,
	\end{align}
	then \eqref{type-wave-1} has a unique global solution $(u,b)$ satisfying
	\begin{align}\label{stability}
		&\|(u,b)\|^2_{H^3(\mathbb{R}^2)}+2\gamma^2\|\partial_tb\|^2_{H^3(\mathbb{R}^2)}
		+2\gamma\eta\|\nabla b\|^2_{H^3(\mathbb{R}^2)}\nonumber\\
		&+\int_0^t\left(\mu\|u_2(\tau)\|^2_{H^3(\mathbb{R}^2)}+\eta\|\nabla b(\tau)\|^2_{H^3(\mathbb{R}^2)}
		+\gamma\|\partial_{\tau}b(\tau)\|^2_{H^3(\mathbb{R}^2)}+\alpha\|\partial_2u(\tau)\|_{H^2(\mathbb{R}^2)}^2\right)d\tau
		\leq C\epsilon^2
	\end{align}
	for any $t>0$ and some universal constant $C>0$.
\end{Thm}

To obtain precise decay rates for solutions in the whole space $\mathbb{R}^2$, it is
standard to impose additional assumptions on the initial data in a space of negative
Sobolev index or in a Lebesgue space with index less than $2$. We recall the definition
of the fractional Laplacian: for any real number $\beta$,
$$\widehat{\Lambda^{\beta}f}(\xi)=|\xi|^{\beta}\widehat{f}(\xi), \qquad \xi=(\xi_1,\xi_2).$$
The large-time behavior of the solutions depends on the eigenvalues of
\eqref{linearized system1}. As discussed above, characterizing the behavior of the
Fourier transform of $(u,b)$ across different frequency regimes is difficult, and direct
spectral analysis fails to yield the desired decay rates. To circumvent this difficulty,
we adopt a time-weighted energy method. Specifically, we first introduce the
negative-index Sobolev energy functional
\begin{align}\label{E-1}
	\mathcal{E}_1(t)=\mathcal{E}_{11}(t)+\mathcal{E}_{12}(t),
\end{align}
where
\begin{align*}
	&\mathcal{E}_{11}(t)=\sup_{0\leq\tau\leq t}\left(\|\Lambda^{-1} u(\tau)\|_{L^2}^2
	+\|\Lambda^{-1} b(\tau)\|_{L^2}^2+2\gamma^2\|\partial_\tau \Lambda^{-1} b(\tau)\|_{L^2}^2
+2\gamma\eta\|b(\tau)\|^2_{L^2}\right)\nonumber\\
	&\qquad \ \ \ \ \ \ +\int_{0}^t\left(\mu\|\Lambda^{-1} u_2(\tau)\|_{L^2}^2+\eta\| b(\tau)\|_{L^2}^2
	+\gamma\|\partial_\tau\Lambda^{-1} b(\tau)\|_{L^2}^2
	\right)d\tau,\\
    &\mathcal{E}_{12}(t)=\alpha\int_{0}^t
	\|\Lambda^{-1}\partial_2 u(\tau)\|_{L^2}^2d\tau.
\end{align*}
We note that $\mathcal{E}_1(t)$ captures the low-frequency behavior of the solution. We then introduce four
time-weighted energy functionals: for $k=0,1,2$,
\begin{align*}
	E_k(t)=&\sup_{0\leq\tau\leq t}(1+\tau)^{k+1}\left(\|\nabla^{k}u(\tau)\|_{H^{3-k}}^2
	+\|\nabla^{k}b(\tau)\|_{H^{4-k}}^2
	+2\gamma^2\|\partial_\tau\nabla^{k} b(\tau)\|_{H^{3-k}}^2\right)\\
	&+\int_{0}^t(1+\tau)^{k+1}\left(\mu\|\nabla^{k} u_2(\tau)\|_{H^{3-k}}^2
	+\eta\|\nabla^{k+1} b(\tau)\|_{H^{3-k}}^2\right.\\
	&+\left.\gamma\|\partial_\tau \nabla^{k}b(\tau)\|_{H^{3-k}}^2
	+\alpha\|\nabla^{k} \partial_2u(\tau)\|_{H^{2-k}}^2\right)d\tau,
\end{align*}
and
\begin{align*}
	E_3(t)=&\sup_{0\leq\tau\leq t}(1+\tau)^4\left(\|\nabla^3 u_2(\tau)\|_{L^2}^2
	+\|\nabla^3 b_2(\tau)\|_{H^1}^2
	+2\gamma^2\|\partial_\tau \nabla^3 b_2(\tau)\|_{L^2}^2\right)\\
	&+\int_{0}^t(1+\tau)^4\left(\mu\|\nabla^3 u_2(\tau)\|_{L^2}^2
	+\eta\|\nabla^4 b_2(\tau)\|_{L^2}^2
	+\gamma\|\partial_\tau\nabla^3 b_2(\tau)\|_{L^2}^2\right)d\tau,
\end{align*}
which encode the algebraic decay rates at each derivative level. The functional $E_3$
captures the faster decay rate of the second components $u_2$ and $b_2$ at the highest
derivative level. By establishing the bounds for $E_k(k=0,1,2,3)$, the precise decay rates will be obtained and stated  in Theorem~\ref{Thm2} below.
\begin{Thm}\label{Thm2}
	Under the assumptions of Theorem~\ref{Thm1}, suppose in addition that
	\begin{align*}
		\Lambda^{-1} u_0,\Lambda^{-1} b_0,\Lambda^{-1}a_0
\in L^2(\R^2).
	\end{align*}
	Then for $k=0,1,2$ and any $t>0$, the global solution $(u,b)$ of \eqref{type-wave-1}
	obeys the decay estimates
	\begin{align*}
		\|\nabla^{k} u(t)\|_{H^{3-k}(\mathbb{R}^2)}+\|\nabla^{k} b(t)\|_{H^{4-k}(\mathbb{R}^2)}
		+\|\partial_t \nabla^{k} b(t)\|_{H^{3-k}(\mathbb{R}^2)}
		\leq C(1+t)^{-\frac{1+k}{2}},
	\end{align*}
	and the second components satisfy the faster decay rate
	\begin{align*}
		\|\nabla^3 u_2(t)\|_{L^2(\mathbb{R}^2)}+\|\nabla^3 b_2(t)\|_{H^1(\mathbb{R}^2)}
		+\|\partial_t \nabla^3 b_2(t)\|_{L^2(\mathbb{R}^2)}\leq C(1+t)^{-2},
	\end{align*}
	where $C>0$ is a universal constant independent of $\epsilon$ and $t$.
\end{Thm}

\begin{Rem}
	The decay rates in Theorem~\ref{Thm2} are consistent with those of the 2D heat equation
	and are optimal. Moreover, the faster rate $(1+t)^{-2}$ for the second components
	$u_2$ and $b_2$ at the highest derivative level reflects the additional regularization
	provided by the vertical damping and the background magnetic field coupling. To the best
	of our knowledge, Theorems~\ref{Thm1} and~\ref{Thm2} together constitute the first
	rigorous result establishing the stabilizing effect of a background magnetic field for
	the MHD-wave equations.
\end{Rem}

For Theorem \ref{Thm1}, we now describe the main ideas in the proofs of the energy inequalities \eqref{E0}. Some of the nonlinear terms arising in the proof of \eqref{E0} cannot be bounded
directly in terms of $\mathcal{E}_{01}(t)$ and $\mathcal{E}_{02}(t)$. To handle
these, we exploit the full algebraic structure of \eqref{type-wave-1}. The most
challenging terms are
\begin{align}
	A_1 = 2\gamma\sum_{i=1}^2\int_{\mathbb{R}^2}
	(b\cdot\nabla)\partial_i^3 u\cdot\partial_t\partial_i^3 b\,dx,
	\quad
	A_{2} = 2\gamma\alpha\sum_{i=1}^2\int_{\mathbb{R}^2}
	\partial_i^3\partial_2 u\cdot\partial_t\partial_i^3 b\,dx.\label{def A12}
\end{align}
The central difficulty is the presence of the fourth-order term $\nabla\partial_i^3 u$
in both integrals, whose order exceeds the $H^3$ energy level. Such a term does not
arise in the standard energy estimates for the MHD equations, and the absence of
velocity dissipation makes it considerably harder to handle than in the MHD setting.
Although the Leray projection argument yields partial dissipation in the $x_1$-direction,
namely $\int_0^t\|\partial_1 u(\tau)\|_{H^2}^2\,d\tau$, this is far from sufficient to
control $\nabla\partial_i^3 u$. A naive integration by parts merely transfers the
fourth-order derivative onto $\partial_t b$, which is equally problematic. To resolve
this derivative-loss issue, we uncover new cancellations within the nonlinear structure
of \eqref{type-wave-1}.
First, integrating $A_1$ by parts gives rise to the following troublesome term
\begin{align*}
	2\gamma\sum_{i=1}^2\int_{\mathbb{R}^2}
	(b\cdot\nabla)\partial_i^3 b\cdot\partial_t \partial_i^3u\,dx.
\end{align*}
To address this term, we substitute $\partial_t u$ using the velocity equation of
\eqref{type-wave-1}, and then
 the difficulty reduces to
\begin{align*}
	-2\gamma\sum_{i=1}^2\int_{\mathbb{R}^2}
  \partial_i^3((b\cdot\nabla) u)\cdot (u\cdot\nabla\partial_i^3 b)\,dx,
\end{align*}
which still cannot be controlled directly. We then replace $(b\cdot\nabla)u$ using the
magnetic field equation of \eqref{type-wave-1}.
At this stage, a further integration by parts introduces the most difficult term
\begin{align*}
	2\gamma\alpha\sum_{i=1}^2\int_{\mathbb{R}^2}
	\partial_i^3((u\cdot\nabla) u)\cdot\partial_i^3\partial_2 b\,dx.
\end{align*}
We substitute $(u\cdot\nabla)u$ via the velocity equation once more,
and find that the resulting expression produces a term coming from $\partial_tu$ that cancels exactly with $A_2$
(see Proposition~\ref{A1b}). In this way, the two most dangerous terms $A_1$ and $A_2$
are resolved simultaneously through a chain of structural cancellations. Further details
can be found in the proofs of Lemmas~\ref{2.3} and~\ref{2.4}.

\vskip .1in
Establishing the decay estimates is a long and technically involved process. We decompose the
proof into the following five steps. The first step establishes the bound for negative Sobolev energy functional $\mathcal{E}_1(t)$. Next, to compute the decay rates for $\|\nabla^{k} u(t)\|_{H^{3-k}}$, $\|\nabla^{k} b(t)\|_{H^{4-k}}$ and $\|\partial_t \nabla^{k} b(t)\|_{H^{3-k}}$ with $k=0,1,2$, we divide the proof into three steps. In the last step, we establish higher decay rates for the horizontal derivatives $\|\nabla^{3} u_2(t)\|_{L^{2}}$, $\|\nabla^{3} b_2(t)\|_{H^{1}}$ and  $\|\nabla^{3}\partial_t b_2(t)\|_{L^{2}}$. It is noteworthy that we prove the decay estimates by establishing the boundedness of the energy functionals $E_0(t)$ to $E_3(t)$, which rely on the precise algebraic decay rates encoded in $E_0(t)$ through $E_3(t)$. The details will be presented in Section \ref{Dec}.

\vskip .1in
The rest of this paper is organized as follows. Section~\ref{Pre} establishes the
key nonlinear estimate, Proposition~\ref{A1b}, which controls the two most dangerous
terms $A_1$ and $A_2$ arising in the energy estimates. Through a chain of structural cancellations, this proposition provides a pointwise-in-time estimate \eqref{AAA1} used in the
stability proof and in the decay proof. Section~\ref{Sta} is
devoted to the proof of Theorem~\ref{Thm1}: by invoking \eqref{AAA1} at the critical
step of the $H^3$ energy estimate and applying a bootstrapping argument, we establish
the energy inequality \eqref{E0} and thereby prove \eqref{stability}.
Section~\ref{Dec} establishes Theorem~\ref{Thm2} by proving the upper bounds for $E_k(t)(k=0,1,2)$ and $E_3(t)$, in which the anisotropic Sobolev inequalities and the stability result \eqref{stability} are invoked. This section is divided into five subsections. In the first subsection, we establish the bound for negative Sobolev energy functional $\mathcal{E}_1(t)$. The following three subsections are devoted to proving the decay rates for $\|\nabla^{k} u(t)\|_{H^{3-k}}$, $\|\nabla^{k} b(t)\|_{H^{4-k}}$ and $\|\partial_t \nabla^{k} b(t)\|_{H^{3-k}}$ with $k=0,1,2$, respectively. In the last subsection, we prove the decay rates for $\|\nabla^{3} u_2(t)\|_{L^{2}}$, $\|\nabla^{3} b_2(t)\|_{H^{1}}$ and  $\|\nabla^{3}\partial_t b_2(t)\|_{L^{2}}$.

In this paper, $C$ denotes a positive constant which may depend on $\mu,\alpha,\gamma$
and $\eta$, and may change from line to line.

\vskip.2in
\section{A key nonlinear estimate}
\label{Pre}

The main result of this section is the following proposition, which provides the
crucial estimates for $A_1$ used in the proofs of Theorem \ref{Thm1} and Theorem \ref{Thm2}.

\begin{Prop}\label{A1b}
	Let $(u,b)$ be a smooth solution to \eqref{type-wave-1} on $[0,T)$ with
	$\mathcal{E}_0(t) < \infty$. Then the following pointwise-in-time estimate holds,	
\begin{align}\label{AAA1}
		A_1 \leq&
		-2\gamma\sum_{i=1}^2\frac{d}{dt}\int_{\mathbb{R}^2}
		\Bigl((b\cdot\nabla)\partial_i^3 b\cdot\partial_i^3 u
		+\alpha\partial_i^3 u\cdot\partial_i^3\partial_2 b
		+\gamma\partial_t\partial_i^3 b\cdot (u\cdot\nabla)\partial_i^3 b\Bigr)dx
		\notag\\
        &+C\left(\|(u,b,\partial_tb)\|_{H^{3}}+\|(u,b,\partial_tb)\|_{H^{3}}^2\right)\left(\|\partial_t\nabla^3b\|_{L^2}^2
        +\|\nabla^2\partial_1u\|_{L^2}^2+\|\nabla^2\partial_2u\|_{L^2}^2+\|\nabla^3 b\|_{H^1}^2\right)\notag\\
		&+\frac{2\mu}{3}\sum_{i=1}^2\|\partial_i^3 u_2\|_{L^2}^2
		+\Bigl(2\gamma\alpha^2+\frac{3\gamma^2\alpha^2\mu}{2}\Bigr)
		\sum_{i=1}^2\|\nabla\partial_i^3 b\|_{L^2}^2 -A_2
	\end{align}
for some constant $C>0$.
\end{Prop}

The proof relies on two anisotropic Sobolev inequalities, which are stated first for convenience. The first inequality is proved in \cite[Lemma~2.1]{LaWZ1} and the second one in
\cite[Lemma~2.2]{LWZ}.

\begin{Lem}\label{2.1}
	Assume that $f,\partial_1f,g,\partial_2g$ are all in $L^2(\mathbb{R}^2)$. Then
	\begin{align}\label{AN}
		 \|fg\|_{L^2} \leq C\|f\|_{L^2}^{\frac{1}{2}}
		\|\partial_1f\|_{L^2}^{\frac{1}{2}}\|g\|_{L^2}^{\frac{1}{2}}
		\|\partial_2g\|_{L^2}^{\frac{1}{2}}.
	\end{align}
	In particular, the following $L^\infty$ bound holds:
	\begin{align*}
		\|f\|_{L^{\infty}} &\leq C\|f\|_{L^2}^{\frac{1}{4}}
		\|\partial_1f\|_{L^2}^{\frac{1}{4}}\|\partial_2f\|_{L^2}^{\frac{1}{4}}
		\|\partial_{12}f\|_{L^2}^{\frac{1}{4}},
	\end{align*}
	and consequently
	\begin{align*}
		\|f\|_{L^{\infty}} &\leq C\|f\|_{H^1}^{\frac{1}{2}}\|\partial_1f\|_{H^1}^{\frac{1}{2}},\\
		\|f\|_{L^{\infty}} &\leq C\|f\|_{H^1}^{\frac{1}{2}}\|\partial_2f\|_{H^1}^{\frac{1}{2}}.
\end{align*}
\end{Lem}

 The proof of Proposition~\ref{A1b} will be divided into the following two lemmas. Lemma~\ref{2.3}
performs the first reduction, expressing $A_1$ in terms of a residual integral
that still contains fourth-order derivatives. Lemma~\ref{2.4} then handles this
residual by exploiting the magnetic field equation and uncovering the cancellation
with $A_2$.

\vskip .1in
\begin{Lem}\label{2.3}
Assume that $(u,b)$ is a smooth solution to \eqref{type-wave-1}, then we have
\begin{align}\label{A}
A_1\leq&-2\gamma\sum_{i=1}^2\frac{d}{dt}\int_{\R^2}(b\cdot\nabla)\partial_i^3 b\cdot\partial_i^3 udx-2\gamma\sum_{i,j,k=1}^2\int_{\R^2}b_j\partial_i^3\partial_ju\cdot u_k\partial_i^3\partial_kbdx\nonumber\\
&+C(\|(u,b,\partial_tb)\|_{H^3}+\|(u,b,\partial_tb)\|_{H^3}^2)\left(\|\nabla^2\partial_1u\|_{L^{2}}^2+\|\nabla^2\partial_2u\|_{L^{2}}^2+\|\nabla^3 b\|_{H^{1}}^2\right)
\end{align}
for some constant $C>0$.
\end{Lem}
\begin{proof}
Integrating by parts and using $\nabla\cdot b=0$, we deduce
 \begin{align}\label{A--1}
A_{1}&=-2\gamma\sum_{i=1}^2\int_{\R^2}(b\cdot\nabla)\partial_t\partial_i^3 b\cdot\partial_i^3 udx\nonumber\\
&=-2\gamma\sum_{i=1}^2\frac{d}{dt}\int_{\R^2}(b\cdot\nabla)\partial_i^3 b\cdot\partial_i^3 udx+2\gamma\sum_{i=1}^2\int_{\R^2}(\partial_t b\cdot\nabla)\partial_i^3 b\cdot\partial_i^3 udx\nonumber\\
&\quad+2\gamma\sum_{i=1}^2\int_{\R^2}( b\cdot\nabla)\partial_i^3 b\cdot\partial_i^3\partial_t udx\nonumber\\
&:=-2\gamma\sum_{i=1}^2\frac{d}{dt}\int_{\R^2}(b\cdot\nabla)\partial_i^3 b\cdot\partial_i^3 udx+2\gamma\sum_{i=1}^2\int_{\R^2}(\partial_t b\cdot\nabla)\partial_i^3 b\cdot\partial_i^3 udx+A_{3}.
\end{align}

In view of  H\"{o}lder's inequality, we obtain
\begin{align}\label{A1}
2\gamma\sum_{i=1}^2\int_{\R^2}(\partial_t b\cdot\nabla)\partial_i^3 b\cdot\partial_i^3 udx&\leq
C\|\partial_t b\|_{L^{\infty}}\|\nabla^3u\|_{L^2}\|\nabla^4b\|_{L^2}\nonumber\\
&\leq C\|\partial_{t}b\|_{H^{3}}(\|\nabla^3u\|_{L^2}^2+\|\nabla^3b\|_{H^1}^2).
\end{align}

 For the  term $A_{3}$, we use the velocity equation \eqref{type-wave-1}$_1$ to substitute for $\partial_t u$ and  rewrite it as
 \begin{align}\label{A3-8}
A_{3}&=2\gamma\sum_{i=1}^2\int_{\R^2}(b\cdot\nabla)\partial_i^3 b\cdot\partial_i^3\big(-\mu (0,u_2)^T-(u \cdot \nabla) u -\nabla p+(b \cdot \nabla) b+\alpha\partial_2 b\big)dx\nonumber\\
&=-2\gamma\mu\sum_{i=1}^2 \int_{\R^2}(b\cdot\nabla)\partial_i^3 b\cdot\partial_i^3(0,u_2)^Tdx -2\gamma\sum_{i=1}^2\int_{\R^2}(b\cdot\nabla)\partial_i^3 b\cdot\partial_i^3((u \cdot \nabla) u) dx\nonumber\\
&\quad-2\gamma\sum_{i=1}^2\int_{\R^2}(b\cdot\nabla)\partial_i^3 b\cdot\partial_i^3\nabla pdx+2\gamma\sum_{i=1}^2\int_{\R^2}(b\cdot\nabla)\partial_i^3 b\cdot\partial_i^3((b \cdot \nabla) b)dx\nonumber\\
&\quad+2\gamma\alpha\sum_{i=1}^2\int_{\R^2}(b\cdot\nabla)\partial_i^3 b\cdot\partial_i^3\partial_2 bdx\nonumber\\
&:=A_{31}+A_{32}+A_{33}+A_{34}+A_{35}.
\end{align}
Due to  H\"{o}lder's inequality and Sobolev's  inequality, one has
\begin{align}\label{A2}
A_{31}+A_{35}\leq&
C\| b\|_{L^{\infty}}\|\nabla^3u_2\|_{L^2}\|\nabla^3b\|_{H^1}+C\| b\|_{L^{\infty}}\|\nabla^3b\|_{H^1}^2\nonumber\\
\leq& C\|b\|_{H^3}(\|\nabla^3 b\|_{H^1}^2+\|\nabla^3u_2\|_{L^2}^2)
\end{align}
and
\begin{align}\label{A3}
A_{34}=&2\gamma\sum_{i=1}^2\int_{\R^2}(b\cdot\nabla)\partial_i^3 b\cdot(\partial_i^3b \cdot \nabla b+3\partial_i^2b \cdot \nabla\partial_i b+3\partial_i b \cdot\nabla \partial_i^2 b+(b\cdot\nabla)\partial_i^3b)dx\nonumber\\
\leq&
C\| b\|_{L^{\infty}}\| \nabla b\|_{L^{\infty}}\|\nabla^3 b\|_{L^2}\|\nabla^4b\|_{L^2}+C\| b\|_{L^{\infty}}^2\|\nabla^4 b\|_{L^2}^2\nonumber\\
&+C\| b\|_{L^{\infty}}\|\nabla^4b\|_{L^2}\| \nabla^2 b\|_{L^{2}}^{\frac{1}{2}}\| \nabla^2\partial_1 b\|_{L^{2}}^{\frac{1}{2}}\| \nabla^2 b\|_{L^{2}}^{\frac{1}{2}}\| \nabla^2\partial_2 b\|_{L^{2}}^{\frac{1}{2}}\nonumber\\
\leq&
C\| b\|_{H^{3}}^2\|\nabla^3 b\|_{H^1}^2.
\end{align}

To estimate $A_{33}$, we need to deal with the pressure term. Taking the divergence $\nabla\cdot$ to \eqref{type-wave-1}$_1$ yields
\begin{align}\label{p}
\Delta p=-\sum_{i,j=1}^2\partial_i\partial_j(u_i u_j-b_i b_j)-\mu\partial_2u_2
\end{align}
which then  implies
\begin{align}\label{SP}
&\|\nabla^3\nabla p\|_{L^2}
\leq C\big(\|\nabla^2(\nabla\cdot(u\cdot\nabla u))\|_{L^2}+\|\nabla^2(\nabla\cdot(b\cdot\nabla b))\|_{L^2}+\|\nabla^2\partial_2u_2\|_{L^2}\big).\end{align}
In view of the anisotropic inequality \eqref{AN}, we get
\begin{align}\label{SP-11}
\|\nabla^2\nabla\cdot(u\cdot\nabla u)\|_{L^2}&\leq C\|\nabla u\|_{L^{\infty}}\|\nabla^3u\|_{L^{2}}+C\|\nabla^2u\|_{L^{2}}^{\frac{1}{2}}\|\nabla^2\partial_{1}u\|_{L^{2}}^{\frac{1}{2}}
\|\nabla^2u\|_{L^{2}}^{\frac{1}{2}}\|\nabla^2\partial_{2}u\|_{L^{2}}^{\frac{1}{2}}\nonumber\\
&\leq C\|u\|_{H^{3}}(\|\nabla^2\partial_1u\|_{L^{2}}+\|\nabla^2\partial_2u\|_{L^{2}}).
\end{align}
Based on the analogous estimate, we can get
\begin{align}\label{SP1}
\|\nabla^2\nabla\cdot(b\cdot\nabla b)\|_{L^2}&\leq C\|\nabla b\|_{L^{\infty}}\|\nabla^3b\|_{L^{2}}
+C\|\nabla^2b\|_{L^{2}}^{\frac{1}{2}}\|\nabla^2\partial_{1}b\|_{L^{2}}^{\frac{1}{2}}
\|\nabla^2b\|_{L^{2}}^{\frac{1}{2}}\|\nabla^2\partial_{2}b\|_{L^{2}}^{\frac{1}{2}}\nonumber\\
&\leq C\| b\|_{H^{3}}\|\nabla^3b\|_{L^{2}}.
\end{align}
Consequently, it follows from \eqref{SP}, \eqref{SP-11} and \eqref{SP1} that
\begin{align}\label{A4}
A_{33}&\leq
C\| b\|_{L^{\infty}}\|\nabla^3 b\|_{H^1}\|\nabla^3\nabla p\|_{L^2}\nonumber\\
&\leq
C(\|(u,b)\|_{H^{3}}+\|(u,b)\|_{H^{3}}^2)\left(\|\nabla^2\partial_1u\|_{L^{2}}^2+\|\nabla^2\partial_2u\|_{L^{2}}^2+\|\nabla^3 b\|_{H^{1}}^2\right).
\end{align}

To estimate $A_{32}$, integrating by parts and using $\nabla\cdot u=\nabla\cdot b=0$, we have
\begin{align*}
A_{32}&=A_{32,1}-2\gamma\sum_{i=1}^2\int_{\R^2}(b\cdot\nabla)\partial_i^3 b\cdot (u \cdot \nabla)\partial_i^3 u dx\\
&=A_{32,1}+A_{32,2}-2\gamma\sum_{i,j,k=1}^2\int_{\R^2}b_j\partial_i^3\partial_ju\cdot u_k\partial_i^3\partial_kbdx,
\end{align*}
where
\begin{align*}
A_{32,1}&=-2\gamma\sum_{i=1}^2\int_{\R^2}(b\cdot\nabla)\partial_i^3 b\cdot(\partial_i^3u \cdot \nabla) u dx-6\gamma\sum_{i=1}^2\int_{\R^2}(b\cdot\nabla)\partial_i^3 b\cdot(\partial_i^2u \cdot \nabla)\partial_i u dx\\
&\quad-6\gamma\sum_{i=1}^2\int_{\R^2}(b\cdot\nabla)\partial_i^3 b\cdot(\partial_i u \cdot \nabla)\partial_i^2 u dx
\end{align*}
and
\begin{align*}
A_{32,2}=2\gamma\sum_{i,j,k=1}^2\int_{\R^2}u_k\partial_kb_j\partial_j\partial_i^3b\cdot \partial_i^3udx-2\gamma\sum_{i,j,k=1}^2\int_{\R^2}b_j \partial_ju_k\partial_i^3\partial_kb\cdot\partial_i^3udx.
\end{align*}
Applying Lemma \ref{2.1} together with H\"{o}lder's inequality and Sobolev's inequality yields
\begin{align}\label{A5}
A_{32,1}
&\leq
C\| b\|_{L^{\infty}}\|\nabla u\|_{L^{\infty}}\|\nabla^3u\|_{L^2}\|\nabla^3b\|_{H^1}\nonumber\\&\quad+C\| b\|_{L^{\infty}}\|\nabla ^2u\|_{L^{2}}^{\frac{1}{2}}\|\nabla ^2\partial_1u\|_{L^{2}}^{\frac{1}{2}}\|\nabla ^2u\|_{L^{2}}^{\frac{1}{2}}\|\nabla ^2\partial_2u\|_{L^{2}}^{\frac{1}{2}}\|\nabla^3b\|_{H^1}\nonumber\\
&\leq
C\| (u,b)\|_{H^{3}}^2(\|\nabla^3u\|_{L^2}^2+\|\nabla^3b\|_{H^1}^2).
\end{align}
Similarly, we have
 \begin{align}\label{A6}
A_{32,2}
&\leq C\| (u,b)\|_{H^{3}}^2(\|\nabla^3u\|_{L^2}^2+\|\nabla^3b\|_{H^1}^2).
\end{align}

Finally, by collecting \eqref{A--1}-\eqref{A3}, \eqref{A4}, \eqref{A5} and \eqref{A6}, we conclude that \eqref{A} holds.
\end{proof}

\vskip.1in
The proof of Lemma~\ref{2.3} reduces the problem to control the term
\[
-2\gamma\sum_{i,j,k=1}^2\int_{\mathbb{R}^2}
b_j\partial_i^3\partial_j u\cdot u_k\partial_i^3\partial_k b\,dx,
\]
which still involves fourth-order derivatives of $u$. The following lemma
shows how to handle this term by substituting $(b\cdot\nabla)u$ from the
magnetic field equation.

\begin{Lem}\label{2.4}
Assume that $(u,b)$ is a smooth solution to \eqref{type-wave-1}, then
\begin{align}\label{AA2}
&-2\gamma\sum_{i,j,k=1}^2\int_{\R^2}b_j\partial_i^3\partial_ju\cdot u_k\partial_i^3\partial_kbdx\notag\\
\leq&
-2\gamma\sum_{i=1}^2\frac{d}{dt}\int_{\R^2}(\alpha\partial_i^3u\cdot\partial_i^3\partial_2b+\gamma\partial_{t}\partial_i^3b\cdot (u\cdot\nabla)\partial_i^3 b)dx\notag\\
&+C\left(\|(u,b,\partial_tb)\|_{H^{3}}+\|(u,b,\partial_tb)\|_{H^{3}}^2\right)\left(\|\partial_t\nabla^3b\|_{L^2}^2+\|\nabla^2\partial_1u\|_{L^2}^2+\|\nabla^2\partial_2u\|_{L^2}^2+\|\nabla^3 b\|_{H^1}^2\right)\notag\\
&-{A_{2}}+\frac{2\mu}{3}\sum_{i=1}^2\| \partial_i^3u_2\|_{L^2}^2+(2\gamma\alpha^2+\frac{3\gamma^2\alpha^2\mu}{2})\sum_{i=1}^2\|\nabla \partial_i^3b\|_{L^2}^2
\end{align}
for some constant $C>0$.

\end{Lem}

\begin{proof}
\par Firstly, we have
 \begin{align}\label{BBBB}
&-2\gamma\sum_{i,j,k=1}^2\int_{\mathbb{R}^2}
b_j\partial_i^3\partial_j u\cdot u_k\partial_i^3\partial_k b\,dx=B_{1}+B_{2},
\end{align}
where
 \begin{align*}
&B_1=2\gamma\sum_{i=1}^2\int_{\R^2}(\partial_i^3 b\cdot\nabla) u\cdot (u\cdot \nabla)\partial_i^3bdx+6\gamma\sum_{i=1}^2\int_{\R^2}(\partial_i^2 b\cdot\nabla)\partial_i u\cdot (u\cdot \nabla)\partial_i^3b dx\nonumber\\&\qquad+6\gamma\sum_{i=1}^2\int_{\R^2}(\partial_i b\cdot\nabla)\partial_i^2 u\cdot (u\cdot \nabla)\partial_i^3b dx
\end{align*}
and
 \begin{align*}
B_2=-2\gamma\sum_{i=1}^2\int_{\R^2}\partial_i^3\big((b\cdot\nabla) u\big)\cdot (u\cdot\nabla)\partial_i^3 bdx.
\end{align*}

By H\"{o}lder's inequality, Lemma \ref{2.1} and Sobolev's inequality, there holds
\begin{align}\label{A7}
B_{1}
&\leq
C\|u\|_{L^{\infty}}\|\nabla u\|_{L^{\infty}}\|\nabla^3b\|_{H^1}^2+C\|u\|_{L^{\infty}}\|\nabla ^2u\|_{L^{2}}^{\frac{1}{2}}\|\nabla ^2\partial_1u\|_{L^{2}}^{\frac{1}{2}}\|\nabla ^2b\|_{L^{2}}^{\frac{1}{2}}\|\nabla ^2\partial_2b\|_{L^{2}}^{\frac{1}{2}}\|\nabla^3b\|_{H^1}\nonumber\\
&\quad+C\|u\|_{L^{\infty}}\|\nabla b\|_{L^{\infty}}\|\nabla^3u\|_{L^2}\|\nabla^3b\|_{H^1}\nonumber\\
&\leq
C\|(u,b)\|_{H^{3}}^2(\|\nabla^3b\|_{H^1}^2+\|\nabla^3u\|_{L^2}^2).
\end{align}

For the  term  $B_{2}$, the trick is to replace $(b\cdot\nabla)u$ by the other terms in the equations of the magnetic field \eqref{type-wave-1} and then
\begin{align}\label{A21}
B_{2}&=-2\gamma\sum_{i=1}^2\int_{\R^2}\partial_i^3(\gamma\partial_{tt}b+\partial_{t}b+(u\cdot\nabla) b-\eta\Delta b-\alpha\partial_2u)\cdot u\cdot\nabla\partial_i^3 bdx\nonumber\\
&=-2\gamma^2\sum_{i=1}^2\int_{\R^2}\partial_{tt}\partial_i^3b\cdot (u\cdot\nabla)\partial_i^3 bdx-2\gamma\sum_{i=1}^2\int_{\R^2}\partial_{t}\partial_i^3b\cdot (u\cdot\nabla)\partial_i^3 bdx\nonumber\\
&\quad-2\gamma\sum_{i=1}^2\int_{\R^2}\partial_i^3((u\cdot\nabla) b)\cdot (u\cdot\nabla)\partial_i^3 bdx+2\gamma\eta\sum_{i=1}^2\int_{\R^2}\partial_i^3(\Delta b)\cdot (u\cdot\nabla)\partial_i^3 bdx\nonumber\\
&\quad+2\gamma\alpha\sum_{i=1}^2\int_{\R^2}\partial_i^3\partial_2u\cdot (u\cdot\nabla)\partial_i^3 bdx\nonumber\\
&:=B_{3}+B_{4}+B_{5}+B_{6}+B_{7}.
\end{align}

Next, we need to bound $B_{3},B_{4},\cdots$ and $B_{7}$ one by one. First, by integration by parts and using $\nabla\cdot u=0$, we have
\begin{align}\label{BBB}
B_{3}
&=-2\gamma^2\sum_{i=1}^2\frac{d}{dt}\int_{\R^2}\partial_{t}\partial_i^3b\cdot (u\cdot\nabla)\partial_i^3 bdx
+2\gamma^2\sum_{i=1}^2\int_{\R^2}\partial_{t}\partial_i^3b\cdot (\partial_{t}u\cdot\nabla)\partial_i^3 bdx.
\end{align}

To deal with the second term on the right-hand side of \eqref{BBB}, the strategy here is to replace $\partial_t u$ by using the velocity field equation  in \eqref{type-wave-1}
and then
\begin{align}\label{K3112}
&2\gamma^2\sum_{i=1}^2\int_{\R^2}\partial_{t}\partial_i^3b\cdot (\partial_{t}u\cdot\nabla)\partial_i^3 bdx\nonumber\\
=&2\gamma^2\sum_{i=1}^2\int_{\R^2}\partial_{t}\partial_i^3b\cdot \left(-(u\cdot\nabla) u-\nabla p-\mu (0,u_2)^{T}+(b\cdot\nabla) b+\alpha\partial_2b
\right)\cdot\nabla\partial_i^3 bdx\nonumber\\
\leq&C\|-(u\cdot\nabla)u-\nabla p-u_2+(b\cdot\nabla )b+\partial_2b\|_{L^{\infty}}\left(\|\nabla^3 b\|_{H^1}^2+\|\partial_t\nabla^3 b\|_{L^2}^2\right).
\end{align}
To proceed we need to deal with $\|\nabla p\|_{L^{\infty}}$. By Lemma \ref{2.1}, we have
\begin{align}\label{PP11}
&\|\nabla p\|_{L^{\infty}}\leq C\|\nabla p\|_{L^2}^{\frac{1}{4}}\|\nabla\partial_1 p\|_{L^2}^{\frac{1}{4}}\|\nabla \partial_2 p\|_{L^2}^{\frac{1}{4}}\|\nabla \partial_{12}p\|_{L^2}^{\frac{1}{4}}.
\end{align}
Noting that
\begin{align*}
p&=\sum_{i,j=1}^2(-\Delta)^{-1}\partial_i\partial_j(u_i u_j-b_i b_j)+\mu(-\Delta)^{-1}\partial_2u_2\nonumber\\
&= \sum_{i,j=1}^2\mathcal{R}_i\mathcal{R}_j(u_i u_j-b_i b_j)+\mu(-\Delta)^{-1}\partial_2u_2,
\end{align*}
where $\mathcal{R}_i=\partial_i(-\Delta)^{-\frac 12}$ denotes the $i$-th Riesz transform.
Based on the boundedness of the operator $\mathcal{R}_i$ on $L^p(\mathbb{R}^2)$ for $1<p<+\infty$ and Sobolev's inequality, we have
\begin{align}\label{P}
\|\nabla p\|_{L^2}&=\sum_{i,j=1}^2\|\nabla\left(\mathcal{R}_i\mathcal{R}_j(u_i u_j-b_i b_j)+\mu{(-\Delta)}^{-1}\partial_2u_2\right)\|_{L^2}\nonumber\\
&\leq2\|u\|_{L^\infty}\|\nabla u\|_{L^{2}}+2\|b\|_{L^\infty}\|\nabla b\|_{L^{2}}+C\|u_2\|_{L^2}\nonumber\\
&\leq C\|u\|_{H^2}\|\nabla u\|_{L^{2}}+C\|b\|_{H^2}\|\nabla b\|_{L^{2}}+C\|u_2\|_{L^2}.
\end{align}
 Similar arguments to \eqref{SP-11} and \eqref{P} lead to
\begin{align}\label{PP12}
\|\nabla\partial_1 p\|_{L^2}
&\leq C\|u\|_{H^3}\|\nabla \partial_1u\|_{L^{2}}+C\|b\|_{H^3}\|\nabla^2 b\|_{L^{2}}+C\|\partial_1u_2\|_{L^2},
\end{align}
\begin{align}\label{PP13}
\|\nabla\partial_2 p\|_{L^2}
&\leq C\|u\|_{H^3}\|\nabla\partial_2u\|_{L^{2}}+C\|b\|_{H^3}\|\nabla^2 b\|_{L^{2}}+C\|\partial_2u_2\|_{L^2},
\end{align}
and
\begin{align}\label{PP14}
\|\nabla\partial_{12} p\|_{L^2}
\leq&C\|u\|_{H^3}\left(\|\partial_1u\|_{H^{2}}+\|\partial_2u\|_{H^{2}}\right)+C\|b\|_{H^3}\|\nabla b\|_{H^{2}}+C\|\partial_{12}u_2\|_{L^2}.
\end{align}
Putting \eqref{P}-\eqref{PP14} into \eqref{PP11}, we obtain
\begin{align}\label{PPPP}
&\|\nabla p\|_{L^{\infty}}\leq C\left(\|(u,b)\|_{H^{3}}+\|(u,b)\|_{H^{3}}^2\right).
\end{align}

Combining \eqref{K3112} and \eqref{PPPP}, then we have
\begin{align}\label{A8-8}
&2\gamma^2\sum_{i=1}^2\int_{\R^2}\partial_{t}\partial_i^3b\cdot (\partial_{t}u\cdot\nabla)\partial_i^3 bdx\nonumber\\
&\leq C\left(\|(u,b)\|_{H^{3}}+\|(u,b)\|_{H^{3}}^2\right)
\left(\|\nabla^3 b\|_{H^1}^2+\|\partial_t\nabla^3 b\|_{L^2}^2\right).
\end{align}
Thereby, there holds
\begin{align}\label{A8}
B_{3}
&\leq-2\gamma^2\sum_{i=1}^2\frac{d}{dt}\int_{\R^2}\partial_{t}\partial_i^3b\cdot (u\cdot\nabla)\partial_i^3 bdx\nonumber\\
&\quad+C\left(\|(u,b)\|_{H^{3}}+\|(u,b)\|_{H^{3}}^2\right)
\left(\|\nabla^3 b\|_{H^1}^2+\|\partial_t\nabla^3 b\|_{L^2}^2\right).
\end{align}

Making use of H\"{o}lder's inequality and Sobolev's inequality, we obtain
\begin{align}\label{A9}
B_{4}
\leq&
C\|u\|_{L^{\infty}}\|\partial_t\nabla^3b\|_{L^2}\|\nabla^4b\|_{L^2}
\leq C\|u\|_{H^3}\left(\|\nabla^3 b\|_{H^1}^2+\|\partial_t\nabla^3 b\|_{L^2}^2\right),
\end{align}
 and
 \begin{align}\label{A11}
B_{5}&=-2\gamma\sum_{i=1}^2\int_{\R^2}(\partial_i^3u\cdot\nabla) b\cdot (u\cdot\nabla)\partial_i^3 bdx-6\gamma\sum_{i=1}^2\int_{\R^2}(\partial_i^2u\cdot\nabla)\partial_i b\cdot (u\cdot\nabla)\partial_i^3 bdx\nonumber\\
&\quad-6\gamma\sum_{i=1}^2\int_{\R^2}(\partial_i u\cdot\nabla)\partial_i^2 b\cdot (u\cdot\nabla)\partial_i^3 bdx-2\gamma\sum_{i=1}^2\int_{\R^2}(u\cdot\nabla)\partial_i^3 b\cdot (u\cdot\nabla)\partial_i^3 bdx\nonumber\\
&\leq C\|u\|_{L^\infty}\|\nabla^3u\|_{L^2}\|\nabla b\|_{L^\infty}\|\nabla^3 b\|_{H^1}+C\|u\|_{L^\infty}\|\nabla^3b\|_{L^2}\|\nabla u\|_{L^\infty}\|\nabla^3 b\|_{H^1}\nonumber\\
&\quad+C\|u\|_{L^\infty}\|\nabla^2u\|_{L^2}^{\frac{1}{2}}\|\nabla^2\partial_1u\|_{L^2}^{\frac{1}{2}}\|\nabla^2 b\|^{\frac{1}{2}}\|\nabla^2\partial_2b\|_{L^2}^{\frac{1}{2}}\|\nabla^3 b\|_{H^1}+C\|u\|_{L^\infty}^2\|\nabla^3 b\|_{H^1}^2\nonumber\\
&\leq C\|(u,b)\|_{H^3}^2(\|\nabla^3 b\|_{H^1}^2+\|\nabla^2\partial_1 u\|_{L^2}^2+\|\nabla^2\partial_2 u\|_{L^2}^2),
\end{align}
and
\begin{align}\label{A10}
B_{6}
&=-2\gamma\eta\sum_{i=1}^2\int_{\R^2}(\nabla u\cdot\nabla)\partial_i^3b\cdot\partial_i^3\nabla bdx\nonumber\\
&\leq
C\|\nabla u\|_{L^{\infty}}\|\nabla^4 b\|_{L^2}^2
\leq C\|u\|_{H^3}\|\nabla^4 b\|_{L^2}^2.
\end{align}

Next, we estimate the last term $B_7$ on the right-hand side of \eqref{A21}. By integration by parts, we obtain
\begin{align}\label{A112}
B_{7}
&=-2\gamma\alpha\sum_{i=1}^2\int_{\R^2}(u\cdot\nabla)\partial_i^3\partial_2u\cdot \partial_i^3 bdx\nonumber\\
&=2\gamma\alpha\sum_{i=1}^2\int_{\R^2}(\partial_2u\cdot\nabla)\partial_i^3u\cdot\partial_i^3 bdx+2\gamma\alpha\sum_{i=1}^2\int_{\R^2}(u\cdot\nabla)\partial_i^3u\cdot \partial_i^3\partial_2 bdx\nonumber\\
&:=B_{71}+B_{72}.
\end{align}

For $B_{71}$, integration by parts  gives
\begin{align}\label{A12}
B_{71}&=-2\gamma\alpha\sum_{i=1}^2\int_{\R^2}(\partial_2u\cdot\nabla)\partial_i^3 b\cdot \partial_i^3udx
\leq C\|u\|_{H^{3}}
(\|\nabla^4 b\|_{L^2}^2+\|\nabla^3u\|_{L^2}^2).
\end{align}

The second term $B_{72}$ needs more subtle work. Firstly, we divide it into two parts:
\begin{align}
B_{72}
&=B_{72,1}+B_{72,2}, \label{B72-de}
\end{align}
where
\begin{align*}
B_{72,1}&=-2\gamma\alpha\sum_{i=1}^2\int_{\R^2}(\partial_i^3u\cdot\nabla) u\cdot\partial_i^3\partial_2 bdx-6\gamma\alpha\sum_{i=1}^2\int_{\R^2}(\partial_i^2u\cdot\nabla)\partial_i u\cdot\partial_i^3\partial_2 bdx\\
&\quad-6\gamma\alpha\sum_{i=1}^2\int_{\R^2}(\partial_i u\cdot\nabla)\partial_i^2u\cdot\partial_i^3\partial_2 bdx
\end{align*}
and
\begin{align*}
B_{72,2}&=2\gamma\alpha\sum_{i=1}^2\int_{\R^2}\partial_i^3((u\cdot\nabla) u)\cdot\partial_i^3\partial_2 bdx.
\end{align*}

With the help of  Lemma \ref{2.1}  we get
\begin{align}\label{A13}
B_{72,1}
&\leq
C\|\nabla u\|_{L^{\infty}}\|\nabla^3u\|_{L^2}\|\nabla^3\partial_2 b\|_{L^2}+C\|\nabla^2 u\|_{L^{2}}^{\frac{1}{2}}\|\nabla^2\partial_1u\|_{L^{2}}^{\frac{1}{2}}\|\nabla^2 u\|_{L^{2}}^{\frac{1}{2}}\| \nabla^2\partial_2u\|_{L^{2}}^{\frac{1}{2}}\|\nabla^3\partial_2b\|_{L^2}\nonumber\\
&\leq C\| u\|_{H^{3}}(\|\nabla^2\partial_1u\|_{L^2}^2+\|\nabla^2\partial_2u\|_{L^2}^2+\|\nabla^4 b\|_{L^2}^2).
\end{align}

For $B_{72,2}$, the strategy here is to replace $(u\cdot\nabla) u$ by using the velocity field equation \eqref{type-wave-1}, then
\begin{align}\label{2152}
B_{72,2}
&=2\gamma\alpha\sum_{i=1}^2\int_{\R^2}\partial_i^3(-\partial_tu+(b\cdot\nabla) b-\nabla p+\alpha\partial_2b -\mu (0,u_2)^{T})\cdot\partial_i^3\partial_2 bdx\notag\\
&=-2\gamma\alpha\sum_{i=1}^2\int_{\R^2}\partial_t\partial_i^3u\cdot\partial_i^3\partial_2 bdx+2\gamma\alpha\sum_{i=1}^2\int_{\R^2}\partial_i^3((b\cdot\nabla) b )\cdot\partial_i^3\partial_2 bdx\notag\\
&\quad+2\gamma\alpha^2\sum_{i=1}^2\|\partial_i^3\partial_2 b\|_{L^2}^2-2\gamma\mu\alpha\sum_{i=1}^2\int_{\R^2}\partial_i^3 (0,u_2)^{T}\cdot\partial_i^3\partial_2 bdx,
\end{align}
where we have used the following fact $$2\gamma\alpha\sum_{i=1}^2\int_{\R^2}\partial_i^3\nabla p\cdot\partial_i^3\partial_2 bdx=0.$$
\par
Now we further bound the terms in (\ref{2152}).  By integration by parts, the first term can be rewritten as
\begin{align*}
&-2\gamma\alpha\sum_{i=1}^2\int_{\R^2}\partial_t\partial_i^3u\cdot\partial_i^3\partial_2 bdx\nonumber\\
&=-2\gamma\alpha\sum_{i=1}^2\frac{d}{dt}\int_{\R^2}\partial_i^3u\cdot\partial_i^3 \partial_2bdx-2\gamma\alpha\sum_{i=1}^2\int_{\R^2}\partial_i^3\partial_2u\cdot\partial_t \partial_i^3bdx\nonumber\\
&=-2\gamma\alpha\sum_{i=1}^2\frac{d}{dt}\int_{\R^2}\partial_i^3u\cdot\partial_i^3 \partial_2bdx-A_{2}.
\end{align*}

For the second term in (\ref{2152}), combining H\"{o}lder's inequality, Lemma \ref{2.1} and Sobolev's inequality, one obtains
 \begin{align*}
&2\gamma\alpha\sum_{i=1}^2\int_{\R^2}\partial_i^3((b\cdot\nabla) b )\cdot\partial_i^3\partial_2 bdx\nonumber\\
&=2\gamma\alpha\sum_{i=1}^2\int_{\R^2}\left(\partial_i^3b\cdot\nabla b+3\partial_i^2b\cdot\nabla\partial_i b +3\partial_i b\cdot\nabla^2\partial_i b+(b\cdot\nabla)\partial_i^3 b \right)\cdot\nabla^3\partial_2 bdx\nonumber\\
&\leq C\|\nabla b\|_{L^{\infty}}\|\nabla^3 b\|_{L^2}\|\nabla^3\partial_2 b\|_{L^2}+C\|\nabla^2 b\|_{L^{2}}^{\frac{1}{2}}\|\nabla^2\partial_1b\|_{L^{2}}^{\frac{1}{2}}\|\nabla^2b\|_{L^{2}}^{\frac{1}{2}}\| \nabla^2\partial_2b\|_{L^{2}}^{\frac{1}{2}}\|\nabla^3b\|_{H^{1}}\nonumber\\
&\quad+ C\|b\|_{L^{\infty}}\|\nabla^3 b\|_{H^1}^2\nonumber\\
&\leq  C\|b\|_{H^{3}}\|\nabla^3 b\|_{H^1}^2.
\end{align*}

Also, for the last term in (\ref{2152}), applying H\"{o}lder's inequality and Young's inequality, we have
 \begin{align*}
-2\gamma\mu\alpha\sum_{i=1}^2\int_{\R^2}\partial_i^3 (0,u_2)^{T}\cdot\partial_i^3\partial_2 bdx
\leq \frac{2\mu}{3}\sum_{i=1}^2\| \partial_i^3u_2\|_{L^2}^2+\frac{3\gamma^2\alpha^2\mu}{2}\sum_{i=1}^2\|\nabla \partial_i^3b\|_{L^2}^2.
\end{align*}

Consequently, one has
\begin{align}\label{B13-222}
B_{72,2}
&\leq-2\gamma\alpha\sum_{i=1}^2\frac{d}{dt}\int_{\R^2}\partial_i^3u\cdot\partial_i^3 \partial_2bdx-A_2+ \frac{2\mu}{3}\sum_{i=1}^2\| \partial_i^3u_2\|_{L^2}^2\nonumber\\
&\quad+(2\gamma\alpha^2+\frac{3\gamma^2\alpha^2\mu}{2})\sum_{i=1}^2\|\nabla \partial_i^3b\|_{L^2}^2+C\|b\|_{H^{3}}\|\nabla^3 b\|_{H^1}^2.
\end{align}

Thanks to \eqref{A112}-\eqref{B72-de}, \eqref{A13} and \eqref{B13-222} one further arrives at
\begin{align}\label{B7-222}
B_{7}
&\leq-2\gamma\alpha\sum_{i=1}^2\frac{d}{dt}\int_{\R^2}\partial_i^3u\cdot\partial_i^3 \partial_2bdx-A_2+(2\gamma\alpha^2+\frac{3\gamma^2\alpha^2\mu}{2})\sum_{i=1}^2\|\nabla\partial_i^3 b\|_{L^2}^2\nonumber\\
&\quad+ \frac{2\mu}{3}\sum_{i=1}^2\| \partial_i^3u_2\|_{L^2}^2+C\| (u,b)\|_{H^{3}}(\|\nabla^2\partial_1u\|_{L^2}^2+\|\nabla^2\partial_2u\|_{L^2}^2+\|\nabla^3 b\|_{H^1}^2).
\end{align}

It follows from \eqref{A8}, \eqref{A9}, \eqref{A11}, \eqref{A10} and \eqref{B7-222}  that
\begin{align}\label{B2-222}
B_2
\leq&
-2\gamma\sum_{i=1}^2\frac{d}{dt}\int_{\R^2}(\alpha\partial_i^3u\cdot\partial_i^3\partial_2b+\gamma\partial_{t}\partial_i^3b\cdot (u\cdot\nabla)\partial_i^3 b)dx-{A_{2}}\notag\\
&+C\left(\|(u,b)\|_{H^{3}}+\|(u,b)\|_{H^{3}}^2\right)\left(\|\partial_t\nabla^3b\|_{L^2}^2+\|\nabla^2\partial_1u\|_{L^2}^2+\|\nabla^2\partial_2u\|_{L^2}^2+\|\nabla^3 b\|_{H^1}^2\right)\notag\\
&+\frac{2\mu}{3}\sum_{i=1}^2\| \partial_i^3u_2\|_{L^2}^2+(2\gamma\alpha^2+\frac{3\gamma^2\alpha^2\mu}{2})\sum_{i=1}^2\|\nabla\partial_i^3 b\|_{L^2}^2.
\end{align}

Therefore, combining \eqref{BBBB}, \eqref{A7} with \eqref{B2-222}, we can get \eqref{AA2}.
\end{proof}

\vskip.1in
\textbf{Completeness of the proof of Proposition \ref{A1b}}: Proposition \ref{A1b} then follows from Lemma \ref{2.3} and \ref{2.4} immediately.
\hfill$\square$

\vskip .3in
\section{Proof of Theorem \ref{Thm1}}
\label{Sta}

\par In this section, we are devoted to proving Theorem \ref{Thm1}. Since the local existence for \eqref{type-wave-1} can be established by the standard approach (see \cite{MB}), we focus here on the global $a$ $priori$ estimate. Then the local solution can be extended to a global solution and  the stability result in Theorem \ref{Thm1} can be obtained by the bootstrapping argument (see \cite[p.21]{Tao}).

\subsection{A Priori Estimate on $\mathcal{E}_{01}(t)$ }
\label{Set1}
 To begin with, we deal with the natural energy functional $\mathcal{E}_{01}(t)$. More precisely, we prove next proposition.
\begin{Prop}\label{Pro11}
Let $\mathcal{E}_{01}(t)$ and $\mathcal{E}_{02}(t)$ be defined as \eqref{E00} and \eqref{E01}, respectively. Then there exists a positive constant $C$, such that
\begin{align}\label{in3.1}
\mathcal{E}_{01}(t)\leq& C\left(\mathcal{E}_{01}(0)+\mathcal{E}_{01}^{\frac{3}{2}}(0)+\mathcal{E}_{01}^{\frac{3}{2}}(t)+\mathcal{E}_{02}^{\frac{3}{2}}(t)+\mathcal{E}_{01}^{2}(t)+\mathcal{E}_{02}^{2}(t)\right).
\end{align}
\end{Prop}
\begin{proof}
Due to the equivalence of the norm $\|(u,b)\|_{H^3}\sim \|(u,b)\|_{L^2}+\|(u,b)\|_{\dot{H}^3}$, it suffices to estimate $\|(u,b)\|_{L^2}$ and $\|(u,b)\|_{\dot{H}^3}$.
\vskip.1in
\par \textbf{Step I. Estimate of ${L^2}$-norm}
\par By the standard energy estimate, we get
\begin{align}\label{3.1}
&\frac{1}{2}\frac{d}{dt}\big(\|u\|^2_{L^2}+\|b\|^2_{L^2}+2\gamma\int_{\R^2}\partial_tb\cdot bdx\big)+\mu\|u_2\|^2_{L^2}+\eta\|\nabla b\|^2_{L^2}=\gamma\|\partial_tb\|^2_{L^2}.
\end{align}

Taking the ${L^2}$ inner product of \eqref{type-wave-1}$_2$ with $\partial_tb$ and integration by parts yields
\begin{align}\label{3.2}
\frac{1}{2}\frac{d}{dt}\big(\gamma\|\partial_tb\|^2_{L^2}+\eta\|\nabla b\|^2_{L^2}+2\int_{\R^2}(b\cdot\nabla) b\cdot udx\big)+\|\partial_tb\|^2_{L^2}
=D_1+D_2,
\end{align}
where
\begin{align*}
D_1=\int_{\R^2}(\partial_tb\cdot\nabla) b\cdot udx
-\int_{\R^2}(u\cdot\nabla) b\cdot\partial_tbdx
\end{align*}
and
\begin{align}\label{D_2}
D_2=\int_{\R^2}(b\cdot\nabla) b\cdot \partial_t udx
+\alpha\int_{\R^2}\partial_2u\cdot\partial_tbdx.
\end{align}

By multiplying \eqref{3.2} by $2\gamma$ and adding the resultant to \eqref{3.1}, we obtain
\begin{align}\label{mhd33}
&\frac{1}{2}\frac{d}{dt}\Big(\|u\|^2_{L^2}+\|b\|^2_{L^2}+2\gamma^2\|\partial_tb\|^2_{L^2}
+2\gamma\eta\|\nabla b\|^2_{L^2}+4\gamma\int_{\R^2}(b\cdot\nabla) b\cdot u dx+2\gamma\int_{\R^2}\partial_tb\cdot bdx
\Big)\nonumber\\
&+\mu\|u_2\|^2_{L^2}+\eta\|\nabla b\|^2_{L^2}+\gamma\|\partial_tb\|^2_{L^2}
=2\gamma(D_1+D_2).
\end{align}

By virtue of H\"{o}lder's inequality and Sobolev's inequality, we obtain
\begin{align}\label{D-1}
D_{1}
\leq&
C\|u\|_{L^\infty}\|\partial_t b\|_{L^2}\|\nabla b\|_{L^2}
\leq C\|u\|_{H^3}\left(\|\nabla b\|_{L^2}^2+\|\partial_t b\|_{L^2}^2\right).
\end{align}

 For $D_{2}$, we divide into three parts to proceed.

 {\it {Part I.}} In this part we  use \eqref{type-wave-1}$_1$ to replace $\partial_tu$, and integration by parts and then get
\begin{align}\label{diff decom-1}
&\int_{\R^2}(b\cdot \nabla)b\cdot \partial_tudx\nonumber\\
&=\int_{\R^2}(b\cdot\nabla) b\cdot(-\mu (0,u_2)^{T}  +(b \cdot \nabla) b+\alpha\partial_2 b)dx
\nonumber\\
&\quad-\int_{\R^2}(b\cdot\nabla) b\cdot\nabla pdx-\int_{\R^2}(b\cdot\nabla) b\cdot (u\cdot\nabla) udx
\nonumber\\
&=\int_{\R^2}(b\cdot\nabla) b\cdot(-\mu (0,u_2)^{T} +(b \cdot \nabla) b+\alpha\partial_2 b)dx-\int_{R^2}(b\cdot\nabla) b\cdot\nabla pdx
\nonumber\\
&\quad+\int_{\R^2}((u\cdot \nabla b)\cdot \nabla)b\cdot udx-\int_{\R^2}((b\cdot \nabla u)\cdot \nabla)b\cdot u dx-\int_{\R^2}(b\cdot \nabla u)\cdot(u\cdot \nabla b)dx.
\end{align}

We can note that
\begin{align}\label{D21}
&\int_{\R^2}(b\cdot\nabla) b\cdot(-\mu (0,u_2)^{T}  +(b \cdot \nabla) b+\alpha\partial_2 b)dx\nonumber\\&
\leq
C\|b\|_{L^\infty}\|\nabla b\|_{L^2}\|u_2\|_{L^2}+C\|b\|_{L^\infty}^2\|\nabla b\|_{L^2}^2+C\|b\|_{L^\infty}\|\nabla b\|_{L^2}\|\partial_2b\|_{L^2}\nonumber\\
&\leq C\left(\|b\|_{H^3}+\|b\|_{H^3}^2\right)\left(\|\nabla b\|_{L^2}^2+\|u_2\|_{L^2}^2\right).
\end{align}
And in view of \eqref{P} we obtain
\begin{align}\label{D22}
&-\int_{\R^2}(b\cdot\nabla) b\cdot\nabla pdx
\nonumber\\&\leq
C\|b\|_{L^\infty}\|\nabla b\|_{L^2}\|\nabla p\|_{L^2}\nonumber\\
&\leq C\|b\|_{H^3}\|\nabla b\|_{L^2}\Big(\|u\|_{H^2}\|\nabla u\|_{L^{2}}+\|b\|_{H^2}\|\nabla b\|_{L^{2}}+\|u_2\|_{L^2}\Big)\nonumber\\
&\leq C\left(\|(u,b)\|_{H^3}+\|(u,b)\|_{H^3}^2\right)\left(\|\nabla u\|_{L^{2}}^2+\|u_2\|_{L^2}^2+\|\nabla b\|_{L^2}^2\right).
\end{align}
Moreover, applying H\"{o}lder's inequality and Sobolev's inequality, we obtain
\begin{align}\label{D23-1}
&\int_{\R^2}((u\cdot \nabla b)\cdot \nabla)b\cdot udx-\int_{\R^2}((b\cdot \nabla u)\cdot \nabla)b\cdot u dx
\nonumber\\ &\leq
C\|u\|_{L^\infty}^2\|\nabla b\|_{L^2}^2+C\|u\|_{L^\infty}\|b\|_{L^\infty}\|\nabla u\|_{L^2}\|\nabla b\|_{L^2}\nonumber\\
&\leq C\|u\|_{H^2}^2\|\nabla b\|_{L^2}^2+C\|u\|_{H^3}\|b\|_{H^3}\|\nabla u\|_{L^2}\|\nabla b\|_{L^2}\nonumber\\
&\leq C\|(u,b)\|_{H^3}^2\left(\|\nabla b\|_{L^2}^2+\|\nabla u\|_{L^2}^2\right).
\end{align}

Finally, submitting \eqref{D21}-\eqref{D23-1} into \eqref{diff decom-1} leads to
\begin{align}\label{diff decom-1-1}
\int_{\R^2}(b\cdot \nabla)b\cdot \partial_tudx&\leq-\int_{\R^2}(b\cdot \nabla u)\cdot(u\cdot \nabla b)dx\nonumber\\
& \quad+C\left(\|(u,b)\|_{H^3}+\|(u,b)\|_{H^3}^2\right)\left(\|\nabla u\|_{L^{2}}^2+\|u_2\|_{L^2}^2+\|\nabla b\|_{L^2}^2\right).
\end{align}

 {\it {Part II.}} In this part we  use \eqref{type-wave-1}$_2$ to replace $(b\cdot \nabla) u$ in \eqref{diff decom-1-1}, and integration by parts twice and then have
\begin{align}
&-\int_{\R^2}(b\cdot \nabla u)\cdot(u\cdot \nabla b)dx\nonumber\\
&=-\gamma\int_{\R^2}\partial_{tt}b\cdot (u\cdot\nabla )bdx-\int_{\R^2}\partial_{t}b\cdot (u\cdot\nabla )bdx-\int_{\R^2}(u\cdot\nabla) b\cdot (u\cdot\nabla) bdx\nonumber\\
&\quad+\eta\int_{\R^2}\Delta b\cdot (u\cdot\nabla) bdx+\alpha\int_{\R^2}\partial_2u\cdot (u\cdot\nabla b)dx\nonumber\\
&=-\gamma\frac{d}{dt}\int_{\R^2}\partial_{t}b\cdot (u\cdot\nabla) bdx+\gamma\int_{\R^2}\partial_{t}b\cdot (\partial_{t}u\cdot\nabla) bdx\nonumber\\
&\quad-\int_{\R^2}\partial_{t}b\cdot (u\cdot\nabla) bdx-\int_{\R^2}(u\cdot\nabla) b\cdot (u\cdot\nabla) bdx+\eta\int_{\R^2}\Delta b\cdot (u\cdot\nabla) bdx\nonumber\\
&\quad+\alpha\int_{\R^2}(\partial_2u\cdot \nabla) u\cdot bdx+\alpha\int_{\R^2}(u\cdot \nabla )u\cdot\partial_2 bdx.
\label{diff decom-2}
\end{align}

Next in order to deal with the second term in \eqref{diff decom-2}, we take advantage of $\eqref{type-wave-1}_1$and \eqref{P}, and then arrive at
\begin{align}\label{DD}
&\gamma\int_{\R^2}\partial_{t}b\cdot (\partial_{t}u\cdot\nabla) bdx\nonumber\\
&=\gamma\int_{\R^2}\partial_{t}b\cdot (-\mu(0,u_2)^{T} - (u \cdot \nabla) u -\nabla p+(b \cdot \nabla) b+\alpha\partial_2 b)\cdot\nabla bdx\nonumber\\
&\leq  CV(t)\left(\|\nabla u\|_{L^{2}}^2+\|u_2\|_{L^2}^2+\|\nabla b\|_{L^2}^2+\|\partial_tb\|_{L^{2}}^2\right),
\end{align}
where
\begin{align}\label{VT}
V(t)=\|(u,b,\partial_tb)\|_{H^{3}}+\|(u,b,\partial_tb)\|_{H^{3}}^2.
\end{align}

Furthermore, the estimates hold
\begin{align}\label{DDD}
&-\int_{\R^2}\partial_{t}b\cdot (u\cdot\nabla) bdx-\int_{\R^2}(u\cdot\nabla) b\cdot (u\cdot\nabla) bdx+\eta\int_{\R^2}\Delta b\cdot (u\cdot\nabla) bdx\nonumber\\
&\leq
C\|u\|_{L^\infty}\|\nabla b\|_{L^2}\|\partial_tb\|_{L^2}+C\|u\|_{L^\infty}^2\|\nabla b\|_{L^2}^2+C\|u\|_{L^\infty}\|\Delta b\|_{L^2}\|\nabla b\|_{L^2}\nonumber\\
&\leq C(\|u\|_{H^3}+\|u\|_{H^3}^2)\left(\|\nabla b\|_{H^1}^2+\|\partial_tb\|_{L^2}^2\right)
\end{align}
and
\begin{align}\label{DDDD}
&\alpha\int_{\R^2}(\partial_2u\cdot \nabla) u\cdot b dx
\leq C\|b\|_{H^{3}}\|\nabla u\|_{L^{2}}^2.
\end{align}
Inserting \eqref{DD}-\eqref{DDDD} into \eqref{diff decom-2}, it follows
\begin{align}\label{diff decom-2-2}
-\int_{\R^2}(b\cdot \nabla u)\cdot(u\cdot \nabla b)dx
&\leq-\gamma\frac{d}{dt}\int_{\R^2}\partial_{t}b\cdot (u\cdot\nabla) bdx+\alpha\int_{\R^2}(u\cdot \nabla) u\cdot\partial_2 bdx\nonumber\\
&\quad +CV(t)\left(\|\nabla u\|_{L^{2}}^2+\|u_2\|_{L^2}^2+\|\nabla b\|_{H^1}^2+\|\partial_tb\|_{L^{2}}^2\right).
\end{align}

{\it {Part III.}} In this part, we  handle the second term on the right-hand side of \eqref{diff decom-2-2}, the strategy here is to replace $(u\cdot \nabla) u$ by using the velocity field equation \eqref{type-wave-1}$_1$, and integration by parts and then have \begin{align*}
&\alpha\int_{\R^2}(u\cdot \nabla)u\cdot\partial_2 bdx\nonumber\\
&=-\alpha\int_{\R^2}\partial_tu\cdot\partial_2 bdx+\alpha\int_{\R^2}(b\cdot\nabla )b\cdot\partial_2 bdx+\alpha^2\|\partial_2b\|_{L^2}^2\nonumber\\
&\quad-\alpha\mu\int_{\R^2}(0,u_2)^{T}\cdot\partial_2 bdx-\alpha\int_{\R^2}\nabla p\cdot\partial_2 bdx\nonumber\\
&=-\alpha\frac{d}{dt}\int_{\R^2}u\cdot\partial_2 bdx-\alpha\int_{\R^2}\partial_2u\cdot\partial_t bdx+\alpha\int_{\R^2}(b\cdot\nabla) b\cdot\partial_2 bdx\nonumber\\
&\quad+\alpha^2\|\partial_2b\|_{L^2}^2-\alpha\mu\int_{\R^2}(0,u_2)^{T}\cdot\partial_2 bdx.
\end{align*}

It is easy to find that
\begin{align*}
&\alpha\int_{\R^2}(b\cdot\nabla) b\cdot\partial_2 bdx\leq C\|b\|_{L^{\infty}}\|\nabla b\|_{L^{2}}^2\leq C\| b\|_{H^{3}}\|\nabla b\|_{L^{2}}^2
\end{align*}
and
\begin{align*}
&-\alpha\mu\int_{\R^2}(0,u_2)^{T}\cdot\partial_2 bdx\leq\mu\alpha\|u_2\|_{L^{2}}\|\nabla b\|_{L^{2}}\leq \frac{\mu}{3\gamma}\| u_2\|_{L^{2}}^{2}+\frac{3\gamma\alpha^2\mu}{4}\|\nabla b\|_{L^{2}}^2.
\end{align*}
In summary, we have
 \begin{align}\label{L-1}
\alpha\int_{\R^2}(u\cdot \nabla) u\cdot\partial_2 bdx
&\leq-\alpha\frac{d}{dt}\int_{\R^2}u\cdot\partial_2 bdx-\alpha\int_{\R^2}\partial_2u\cdot\partial_t bdx+\alpha^2\|\partial_2b\|_{L^2}^2\nonumber\\
&\quad+\frac{\mu}{3\gamma}\| u_2\|_{L^{2}}^{2}+\frac{3\gamma\alpha^2\mu}{4}\|\nabla b\|_{L^{2}}^2+ C\|b\|_{H^3}\|\nabla b\|_{L^2}^2.
\end{align}

Combining \eqref{diff decom-1-1},  \eqref{diff decom-2-2} and \eqref{L-1}, we have
\begin{align}\label{DDDDDD}
\int_{\R^2}(b\cdot \nabla)b\cdot \partial_tudx&\leq-\gamma\frac{d}{dt}\int_{\R^2}\partial_{t}b\cdot (u\cdot\nabla) bdx-\alpha\frac{d}{dt}\int_{\R^2}u\cdot\partial_2 bdx-\alpha\int_{\R^2}\partial_2u\cdot\partial_t bdx\nonumber\\
&\quad+\alpha^2\|\partial_2b\|_{L^2}^2+\frac{\mu}{3\gamma}\| u_2\|_{L^{2}}^{2}+\frac{3\gamma\alpha^2\mu}{4}\|\nabla b\|_{L^{2}}^2\nonumber\\
&\quad+CV(t)\left(\|\nabla u\|_{L^{2}}^2+\|u_2\|_{L^2}^2+\|\nabla b\|_{H^1}^2+\|\partial_tb\|_{L^{2}}^2\right).
\end{align}

 Thus, from \eqref{D_2} and \eqref{DDDDDD}, we have

\begin{align}\label{DDDDDD1}
D_2&\leq-\gamma\frac{d}{dt}\int_{\R^2}\partial_{t}b\cdot (u\cdot\nabla) bdx-\alpha\frac{d}{dt}\int_{\R^2}u\cdot\partial_2 bdx\nonumber\\
&\quad+\alpha^2\|\partial_2b\|_{L^2}^2+\frac{\mu}{3\gamma}\| u_2\|_{L^{2}}^{2}+\frac{3\gamma\alpha^2\mu}{4}\|\nabla b\|_{L^{2}}^2\nonumber\\
&\quad+CV(t)\left(\|\nabla u\|_{L^{2}}^2+\|u_2\|_{L^2}^2+\|\nabla b\|_{H^1}^2+\|\partial_tb\|_{L^{2}}^2\right).
\end{align}

 Substituting \eqref{D-1} and \eqref{DDDDDD1} into \eqref{mhd33}, we find
\begin{align}\label{mhd333}
&\frac{1}{2}\frac{d}{dt}\Big(\|u\|^2_{L^2}+\|b\|^2_{L^2}+2\gamma^2\|\partial_tb\|^2_{L^2}+2\gamma\eta\|\nabla b\|^2_{L^2}+2\gamma\int_{\R^2}\partial_tb\cdot bdx\nonumber\\&\qquad+4\gamma\int_{\R^2}(b\cdot\nabla) b\cdot u dx+4\gamma^2\int_{\R^2}\partial_{t}b\cdot (u\cdot\nabla) bdx+4\gamma\alpha\int_{\R^2}u\cdot\partial_2 bdx
\Big)\nonumber\\
&\qquad+\frac{\mu}{3}\|u_2\|^2_{L^2}+(\eta-2\gamma\alpha^2-\frac{3\gamma^2\alpha^2\mu}{2})\|\nabla b\|^2_{L^2}+\gamma\|\partial_tb\|^2_{L^2}\nonumber\\
&\qquad\leq CV(t)\left(\|\nabla u\|_{L^{2}}^2+\|u_2\|_{L^2}^2+\|\nabla b\|_{H^1}^2+\|\partial_tb\|_{L^{2}}^2\right).
\end{align}

Integrating \eqref{mhd333} over $[0,t](t>0)$ yields
\begin{align*}
&\|u\|^2_{L^2}+\|b\|^2_{L^2}+2\gamma^2\|\partial_tb\|^2_{L^2}
+2\gamma\eta\|\nabla b\|^2_{L^2}+2\gamma\int_{\R^2}\partial_tb\cdot bdx
\nonumber\\
&+4\gamma\int_{\R^2}(b\cdot\nabla) b\cdot udx+4\gamma^2\int_{\R^2}\partial_{t}b\cdot (u\cdot\nabla )bdx+
4\gamma\alpha\int_{\R^2}u\cdot\partial_2 bdx\nonumber\\
&+\int_{0}^t\bigg(\frac{2\mu}{3}\|u_2(\tau)\|^2_{L^2}+(2\eta-4\gamma\alpha^2-3\gamma^2\alpha^2\mu)\|\nabla b(\tau)\|^2_{L^2}+2\gamma\|\partial_{\tau}b(\tau)\|^2_{L^2}\bigg)d\tau\nonumber\\
&\leq\|u_0\|^2_{L^2}+\|b_0\|^2_{L^2}+\gamma^2\|a_0\|^2_{L^2}
+2\gamma\eta\|\nabla b_0\|^2_{L^2}+2\gamma\int_{\R^2}a_0\cdot b_0dx\nonumber \\
&\quad+4\gamma\int_{\R^2}(b_0\cdot\nabla) b_0\cdot u_0dx+4\gamma^2\int_{\R^2}a_0\cdot (u_0\cdot\nabla) b_0dx+
4\gamma\alpha\int_{\R^2}u_0\cdot\partial_2 b_0dx\nonumber\\
&\quad+C\sup_{0\leq\tau\leq t}V(\tau)\int_0^t\left(\|\nabla u(\tau)\|_{L^{2}}^2+\|u_2(\tau)\|_{L^2}^2+\|\nabla b(\tau)\|_{H^1}^2+\|\partial_{\tau}b(\tau)\|_{L^{2}}^2\right)d\tau.
\end{align*}
With the help of the following inequalities:
\begin{align*}
2\gamma\int_{\R^2}\partial_tb\cdot bdx
\leq & \frac{2}{3}\|b\|_{L^2}^2+\frac{3\gamma^2}{2}\|\partial_t b\|_{L^2}^2, \\
4\gamma\alpha\int_{\R^2}u\cdot \partial_2bdx
\leq & \frac{2}{3}\|u\|_{L^2}^2+6\gamma^2\alpha^2\|\nabla b\|_{L^2}^2,\\
2\gamma\int_{\R^2}a_0\cdot b_0dx
\leq & \frac{1}{\sqrt{2}}\|b_0\|_{L^2}^2+\sqrt{2}\gamma^2\|a_0\|_{L^2}^2,\\
4\gamma\alpha\int_{\R^2}u_0\cdot\partial_2 b_0dx
\leq & \frac{1}{\sqrt{2}}\|u_0\|_{L^2}^2+4\sqrt{2}\gamma^2\alpha^2\|\partial_2 b_0\|_{L^2}^2,
\end{align*}
one further has
\begin{align}\label{energy1}
&\|u\|^2_{L^2}+\|b\|^2_{L^2}+\frac{3\gamma^2}{2}\|\partial_tb\|^2_{L^2}
+6\gamma(\eta-3\gamma\alpha^2)\|\nabla b\|^2_{L^2}\nonumber\\
&
+\int_{0}^t\left(2\mu\|u_2(\tau)\|^2_{L^2}+(6\eta-12\gamma\alpha^2-9\gamma^2\alpha^2\mu)\|\nabla b(\tau)\|^2_{L^2}+6\gamma\|\partial_{\tau}b(\tau)\|^2_{L^2}\right)d\tau\nonumber\\
&+12\gamma\int_{\R^2}(b\cdot\nabla) b\cdot udx+12\gamma^2\int_{\R^2}\partial_{t}b\cdot (u\cdot\nabla) bdx\nonumber\\
&\leq C\left(\|u_0\|^2_{L^2}+\|b_0\|^2_{L^2}+\gamma^2\|\partial_tb_0\|^2_{L^2}
+2\gamma\eta\|\nabla b_0\|^2_{L^2}\right)\nonumber\\
&\quad+12\gamma\int_{\R^2}(b_0\cdot\nabla) b_0\cdot u_0dx+12\gamma^2\int_{\R^2}a_0\cdot( u_0\cdot\nabla) b_0dx\nonumber\\
&\quad+C\sup_{0\leq\tau\leq t}V(\tau)\int_0^t\left(\|\nabla u(\tau)\|_{L^{2}}^2+\|u_2(\tau)\|_{L^2}^2+\|\nabla b(\tau)\|_{H^1}^2+\|\partial_{\tau}b(\tau)\|_{L^{2}}^2\right)d\tau.
\end{align}

\vskip.1in
\textbf{Step II. Estimate of ${\dot{H}^3}$-norm}
\par We apply the operator $\partial_i^3(i=1,2)$ to \eqref{type-wave-1}$_1$ and \eqref{type-wave-1}$_2$, multiply the resulting equations by $\partial_i^3u$ and $\partial_i^3b$ respectively, and then sum them up to obtain
\begin{align}\label{3.18}
&\frac{1}{2}\frac{d}{dt}\sum_{i=1}^2\big(\|\partial_i^3u\|^2_{L^2}+\|\partial_i^3b\|^2_{L^2}+2\gamma\int_{\R^2}\partial_t\partial_i^3b\cdot\partial_i^3bdx\big)+\mu\sum_{i=1}^2\|\partial_i^{3}u_2\|^2_{L^2}\nonumber\\
&+\eta\sum_{i=1}^2\|\nabla\partial_i^{3} b\|^2_{L^2}-\gamma\sum_{i=1}^2\|\partial_t\partial_i^3b\|^2_{L^2}\nonumber\\
&=-\sum_{i=1}^2\int_{\R^2}\partial_i^3((u\cdot\nabla) u)\cdot\partial_i^3udx+\sum_{i=1}^{2}\int_{\R^2}\big(\partial_{i}^3(b\cdot\nabla b)\cdot\partial_{i}^3u+\partial_{i}^3((b\cdot\nabla) u)\cdot\partial_{i}^3b\big)dx\nonumber\\
&\quad-\sum_{i=1}^2\int_{\R^2}\partial_i^3((u\cdot \nabla) b)\cdot\partial_i^3bdx\nonumber\\
&:=I_1+I_2+I_3.
\end{align}

Moreover, applying the operator $\partial_i^3$ to \eqref{type-wave-1}$_2$ and multiplying the resulting equation by $\partial_t\partial_i^3b$ give
\begin{align}\label{3.19}
&\frac{1}{2}\frac{d}{dt}\sum_{i=1}^2\big(\gamma\|\partial_t\partial_i^3b\|^2_{L^2}+\eta\|\nabla\partial_i^3 b\|^2_{L^2}\big)+\sum_{i=1}^2\|\partial_t\partial_i^3b\|^2_{L^2}\nonumber\\
&=\sum_{i=1}^2\int_{\R^2}\partial_i^3((b\cdot\nabla) u)\cdot\partial_t\partial_i^3bdx-\sum_{i=1}^2\int_{\R^2}\partial_i^3((u\cdot\nabla) b)\cdot\partial_t\partial_i^3bdx+\alpha\sum_{i=1}^2\int_{\R^2}\partial_i^3\partial_2u\cdot\partial_t\partial_i^3bdx\nonumber\\
&:=I_4+I_5+\frac{A_2}{2\gamma},
\end{align}
where $A_2$ is defined in \eqref{def A12}.

By multiplying \eqref{3.19} by $2\gamma$ and adding the resultant to \eqref{3.18}, we get
\begin{align}\label{mhd3}
&\frac{1}{2}\frac{d}{dt}\sum_{i=1}^2\left(\|\partial_i^3 u\|^2_{L^2}+\|\partial_i^3 b\|^2_{L^2}+2\gamma^2\|\partial_t\partial_i^3b\|^2_{L^2}
+2\gamma\eta\|\nabla\partial_i^3 b\|^2_{L^2}+2\gamma\int_{\R^2}\partial_t\partial_i^3b\cdot\partial_i^3bdx
\right)\nonumber\\
&+\mu\sum_{i=1}^2\|\partial_i^3u_2\|^2_{L^2}+\eta\sum_{i=1}^2\|\nabla \partial_i^3 b\|^2_{L^2}+\gamma\sum_{i=1}^2\|\partial_t\partial_i^3b\|^2_{L^2}\nonumber\\
&=I_1+I_2+I_3+2\gamma I_4+2\gamma  I_5+A_2.
\end{align}

\par For $I_1$, it follows from $\nabla\cdot u=0$,
H\"{o}lder's inequality and Sobolev's inequality that
\begin{align}
I_{1}&=-\sum_{i=1}^2\int_{\R^2}\big((\partial_{i}^3u\cdot\nabla) u+3(\partial_{i}^2u\cdot\nabla)\partial_{i} u+3(\partial_{i}u\cdot\nabla)\partial_{i}^2 u\big)\cdot\partial_{i}^3udx\nonumber\\
&\leq C\|\partial_{i}^2u\|_{L^4}\|\nabla\partial_{i}u\|_{L^4}\|\partial_{i}^3u\|_{L^2}+C\|\nabla u\|_{L^{\infty}}\|\nabla\partial_{i}^2u\|_{L^2}\|\partial_{i}^3u\|_{L^2}\nonumber\\
&\leq C\|u\|_{H^{3}}\|\nabla u\|_{H^2}^2.\label{control I1}
\end{align}

To bound $I_2$, we have
\begin{align}
I_{2}&=
\sum_{i=1}^{2}\int_{\R^2}\Big(3(\partial_{i}b\cdot\nabla)\partial_{i}^2b\cdot\partial_{i}^3u+
3(\partial_{i}^2b\cdot\nabla)\partial_{i}b\cdot\partial_{i}^3u
+(\partial_{i}^3b\cdot\nabla) b\cdot\partial_{i}^3u\Big)dx\nonumber\\&\quad+\sum_{i=1}^{2}
\int_{\R^2}\Big(3(\partial_{i}b\cdot\nabla)\partial_{i}^2u\cdot\partial_{i}^3b+
3(\partial_{i}^2b\cdot\nabla)\partial_{i}u\cdot\partial_{i}^3b
+(\partial_{i}^3b\cdot\nabla) u\cdot\partial_{i}^3b\Big)dx\nonumber\\
&\leq C\left(\|\nabla b\|_{L^\infty}+\|\nabla^2 b\|_{L^\infty}\right)\|\nabla b\|_{H^3}\|u\|_{H^3}+C\|\nabla u\|_{L^\infty}\|\nabla^3b\|_{L^2}^2\nonumber\\
&\leq C\|u\|_{H^3}\|\nabla b\|_{H^3}^2,\label{control I2}
\end{align}
where we have used
\begin{align*}
\int_{\R^2}(b\cdot\nabla)\partial_{i}^3 b \cdot\partial_{i}^3udx+\int_{\R^2}(b\cdot\nabla)\partial_{i}^3 u \cdot\partial_{i}^3bdx=0.
\end{align*}
Similarly, there hold
\begin{align}
I_{3}
\leq &C\|u\|_{H^3}\|\nabla b\|_{H^3}^2\label{control I3}
\end{align}
and
\begin{align}
2\gamma I_5\leq C\|(u,b)\|_{H^3}\left(\|\partial_t b\|_{H^3}^2+\|\nabla b\|_{H^3}^2\right).\label{control I5}
\end{align}

\par For $2\gamma I_4$, we get
 \begin{align}
2\gamma I_4&=
A_{1}+2\gamma\sum_{i=1}^2\int_{\R^2}(\partial_i^3b\cdot\nabla) u\cdot\partial_t\partial_i^3 bdx\nonumber\\
&\quad+6\gamma\sum_{i=1}^2\int_{\R^2}(\partial_i^2b\cdot\nabla)\partial_i u\cdot\partial_t\partial_i^3 bdx+6\gamma\sum_{i=1}^2\int_{\R^2}(\partial_i b\cdot\nabla)\partial_i^2 u\cdot\partial_t\partial_i^3 bdx\nonumber\\ &\leq
A_1+C\sum_{i=1}^2
\|\nabla u\|_{L^{\infty}}\|\partial_i^3 b\|_{L^2}\|\partial_t\partial_i^3b\|_{L^2}\nonumber\\
&\quad
+C\sum_{i=1}^2(\|\partial_i^2 b\|_{L^{\infty}}\|\nabla\partial_iu\|_{L^2}\|\partial_t\partial_i^3b\|_{L^2}+\|\partial_i b\|_{L^{\infty}}\|\nabla\partial_i^2u\|_{L^2}\|\partial_t\partial_i^3b\|_{L^2})\nonumber\\
&\leq A_1+C\|(u,b)\|_{H^3}\big(\|\nabla b\|_{H^3}^2+\|\partial_{t}b\|_{H^3}^2\big).\label{control I4}
\end{align}
\par

Submitting \eqref{control I1}-\eqref{control I4} into \eqref{mhd3} and invoking \eqref{AAA1}, we have
\begin{align}\label{mhd32}
&\frac{1}{2}\frac{d}{dt}\sum_{i=1}^2\left(\|\partial_i^3 u\|^2_{L^2}+\|\partial_i^3 b\|^2_{L^2}+2\gamma^2\|\partial_t\partial_i^3b\|^2_{L^2}
+2\gamma\eta\|\nabla\partial_i^3 b\|^2_{L^2}+2\gamma\int_{\R^2}\partial_t\partial_i^3b\cdot\partial_i^3bdx\right.\nonumber\\
&\left.\qquad\qquad+4\gamma\int_{\R^2}\left((b\cdot\nabla)\partial_i^3 b\cdot\partial_i^3 u+\gamma\partial_{t}\partial_i^3b\cdot (u\cdot\nabla)\partial_i^3 b+\alpha\partial_i^3u\cdot\partial_i^3\partial_2 b\right)dx\right)\nonumber\\
&\quad+\sum_{i=1}^2\bigg(\frac{\mu}{3}\|\partial_i^3u_2\|^2_{L^2}+(\eta-2\gamma\alpha^2-\frac{3\gamma^2\alpha^2\mu}{2})\|\nabla \partial_i^3b\|^2_{L^2}+\gamma\|\partial_{t}\partial_i^3b\|^2_{L^2}\bigg)\nonumber\\
&\leq CV(t)\left(\|\partial_{t} b(t)\|_{H^3}^2+\|\nabla b(t)\|_{H^3}^2+\|u_2(t)\|_{H^3}^2+\|\partial_{2}u(t)\|_{H^2}^2\right),
\end{align}
where we have used \begin{align*}\|\nabla u\|_{H^2}^2=\|\partial_1 u\|_{H^2}^2+\|\partial_2 u\|_{H^2}^2=\| \nabla u_2\|_{H^2}^2+\|\partial_2 u\|_{H^2}^2.\end{align*}

 Integrating \eqref{mhd32} over $[0,t]$, we derive
\begin{align}\label{s1}
&\sum_{i=1}^2\Big(\|\partial_i^3 u\|^2_{L^2}+\|\partial_i^3 b\|^2_{L^2}+2\gamma^2\|\partial_t\partial_i^3b\|^2_{L^2}
+2\gamma\eta\|\nabla\partial_i^3 b\|^2_{L^2}+2\gamma\int_{\R^2}\partial_t\partial_i^3b\cdot\partial_i^3bdx
\nonumber\\
&\quad\quad+4\gamma\int_{\R^2}\big((b\cdot\nabla)\partial_i^3 b\cdot\partial_i^3 u+\gamma\partial_{t}\partial_i^3b\cdot (u\cdot\nabla)\partial_i^3 b+\alpha\partial_i^3u\cdot\partial_i^3\partial_2 b\big)dx\Big)\nonumber\\
&\quad+\sum_{i=1}^2\int_{0}^t\left(\frac{2\mu}{3}\|\partial_i^3u_2(\tau)\|^2_{L^2}+(2\eta-4\gamma\alpha^2-3\gamma^2\alpha^2\mu)\|\nabla\partial_i^3b(\tau)\|^2_{L^2}+2\gamma\|\partial_{\tau}\partial_i^3b(\tau)\|^2_{L^2}\right)d\tau\nonumber\\
&\leq \sum_{i=1}^2\Big(\|\partial_i^3 u_0\|^2_{L^2}+\|\partial_i^3 b_0\|^2_{L^2}+2\gamma^2\|\partial_i^3a_0\|^2_{L^2}
+2\gamma\eta\|\nabla\partial_i^3 b_0\|^2_{L^2}+2\gamma\int_{\R^2}\partial_i^3b_0\cdot\partial_i^3a_0dx
\nonumber\\
&\quad\qquad+4\gamma\int_{\R^2}\big((b_0\cdot\nabla)\partial_i^3 b_0\cdot\partial_i^3 u_0+\gamma\partial_i^3a_0\cdot (u_0\cdot\nabla)\partial_i^3 b_0+\alpha\partial_i^3u_0\cdot\partial_i^3\partial_2 b_0\big)dx\Big)\nonumber\\
&\quad+C\sup_{0\leq\tau\leq t}V(\tau)\int_{0}^t\left(\|\partial_{\tau} b(\tau)\|_{H^3}^2+\|\nabla b(\tau)\|_{H^3}^2+\|u_2(\tau)\|_{H^3}^2+\|\partial_{2}u(\tau)\|_{H^2}^2\right)d\tau.
\end{align}

 Observe that
\begin{align}
&2\gamma\sum_{i=1}^2\int_{\R^2}\partial_t\partial_i^3b\cdot\partial_i^3bdx
\leq \frac{2}{3}\|\nabla^3b\|_{L^2}^2+\frac{3\gamma^2}{2}\|\partial_t\nabla^3 b\|_{L^2}^2,\label{Y1}\\
&2\gamma\sum_{i=1}^2\int_{\R^2}\partial_i^3a_0\cdot\partial_i^3b_0dx
\leq \frac{1}{\sqrt{2}}\|\nabla^3b_0\|_{L^2}^2+\sqrt{2}\gamma^2\|\nabla^3a_0\|_{L^2}^2,\label{Y2}\\
&4\gamma\alpha\sum_{i=1}^2\int_{\R^2}\partial_i^3u\cdot\partial_i^3\partial_2bdx
\leq \frac{2}{3}\|\nabla^3u\|_{L^2}^2+6\gamma^2\alpha^2\|\nabla^4 b\|_{L^2}^2,\label{Y3}\\
&4\gamma\alpha\sum_{i=1}^2\int_{\R^2}\partial_i^3u_0\cdot\partial_i^3\partial_2b_0dx
\leq \frac{1}{\sqrt{2}}\|\nabla^3u_0\|_{L^2}^2+4\sqrt{2}\gamma^2\alpha^2\|\nabla^4 b_0\|_{L^2}^2.\label{Y4}
\end{align}

Inserting \eqref{Y1}-\eqref{Y4} into \eqref{s1}, then we have
\begin{align}\label{energy2}
&\sum_{i=1}^2\left(\|\partial_i^3 u\|^2_{L^2}+\|\partial_i^3 b\|^2_{L^2}+\frac{3\gamma^2}{2}\|\partial_t\partial_i^3b\|^2_{L^2}
+6\gamma(\eta-3\gamma\alpha^2)\|\nabla\partial_i^3 b\|^2_{L^2}
\right)\nonumber\\
&\quad+\sum_{i=1}^2\int_{0}^t\left(2\mu\|\partial_i^3u_2(\tau)\|^2_{L^2}+\left(6\eta-12\gamma\alpha^2-9\gamma^2\alpha^2\mu\right)\|\nabla\partial_i^3 b(\tau)\|^2_{L^2}+6\gamma\|\partial_{\tau}\partial_i^3b(\tau)\|^2_{L^2}\right)d\tau\nonumber\\
&\quad+12\gamma\sum_{i=1}^2\int_{\R^2}\big((b\cdot\nabla)\partial_i^3 b\cdot\partial_i^3 u+\gamma\partial_{t}\partial_i^3b\cdot (u\cdot\nabla)\partial_i^3 b\big)dx\nonumber\\
&\leq C \sum_{i=1}^2\left(\|\partial_i^3 u_0\|^2_{L^2}+\|\partial_i^3 b_0\|^2_{L^2}+2\gamma^2\|\partial_i^3a_0\|^2_{L^2}
+2\gamma\eta\|\partial_i^3\nabla b_0\|^2_{L^2}\right)\nonumber\\
&\quad+12\gamma\sum_{i=1}^2\int_{\R^2}\big((b_0\cdot\nabla)\partial_i^3 b_0\cdot\partial_i^3 u_0+\gamma\partial_i^3a_0\cdot (u_0\cdot\nabla)\partial_i^3 b_0\big)dx\nonumber\\
&\quad+C\sup_{0\leq\tau\leq t}V(\tau)\int_{0}^t\left(\|\partial_{\tau} b(\tau)\|_{H^3}^2+\|\nabla b(\tau)\|_{H^3}^2+\|u_2(\tau)\|_{H^3}^2+\|\partial_{2}u(\tau)\|_{H^2}^2\right)d\tau.
\end{align}
\vskip.1in
 \textbf{Step III. Estimate of $\mathcal{E}_{01}(t)$}
\par If $\alpha$ is taken sufficiently small, then combining \eqref{energy1} with \eqref{energy2} and applying Sobolev's inequality imply
\begin{align*}
&\|(u,b)\|^2_{H^3}+\frac{3\gamma^2}{2}\|\partial_tb\|^2_{H^3}
+6\gamma(\eta-3\gamma\alpha^2)\|\nabla b\|^2_{H^3}\nonumber\\
&\quad+\int_0^t\left(2\mu\|u_2(\tau)\|^2_{H^3}+(6\eta-12\gamma\alpha^2-9\gamma^2\alpha^2\mu)\|\nabla b(\tau)\|^2_{H^3}+6\gamma\|\partial_{\tau}b(\tau)\|^2_{H^3}\right)d\tau\nonumber\\
&\leq C\left(\|(u_0,b_0)\|^2_{H^3}+2\gamma\eta\|\nabla b_0\|^2_{H^3}+2\gamma^2\|a_0\|^2_{H^3}\right)+C\|(u,b)\|_{H^{2}}\left(\|u\|_{H^{3}}^2+\|\nabla b\|_{H^{3}}^2+\|\partial_t b\|_{H^{3}}^2\right)\nonumber\\
&\quad+C\|(u_0,b_0)\|_{H^{2}}\left(\|u_0\|_{H^{3}}^2+\|\nabla b_0\|_{H^{3}}^2+\|a_0\|_{H^{3}}^2\right)\nonumber\\
&\quad+C\sup_{0\leq\tau\leq t}V(\tau)\int_{0}^t\left(\|\partial_{\tau} b(\tau)\|_{H^3}^2+\|\nabla b(\tau)\|_{H^3}^2+\|u_2(\tau)\|_{H^3}^2+\|\partial_{2}u(\tau)\|_{H^2}^2\right)d\tau,
\end{align*}
which immediately leads to the desired inequality \eqref{in3.1}. This completes the proof of Proposition \ref{Pro11}.
\end{proof}

\vskip.1in
\subsection{A Priori estimate on $\mathcal{E}_{02}(t)$}
In this subsection, we establish the dissipation of $u$ in $x_2$ direction and prove the estimate of $\mathcal{E}_{02}(t)$ defined in \eqref{E01}.
\begin{Prop}\label{Pro22}
Let $\mathcal{E}_{01}(t)$ and $\mathcal{E}_{02}(t)$ be defined as \eqref{E00} and \eqref{E01}, respectively. Then we have, for some positive constant $C$
\begin{align}\label{in3.2}
\mathcal{E}_{02}(t)\leq C\left(\mathcal{E}_{01}(0)+\mathcal{E}_{01}(t)+\mathcal{E}_{01}^{\frac{3}{2}}(t)+\mathcal{E}_{02}^{\frac{3}{2}}(t)\right).
\end{align}
\end{Prop}
\begin{proof}
 We will divide into two steps by using the special structure of the magnetic field equation \eqref{type-wave-1}$_2$ to control the following two terms
$$\int_0^t\|\partial_2u(\tau)\|^2_{L^2}d\tau  \ \text{and} \ \int_0^t\|\nabla^2\partial_2u(\tau)\|^2_{L^2}d\tau,$$   and then get $\mathcal{E}_{02}(t)$ with the help of  $\|\partial_2u\|_{H^2}\sim\|\partial_2u\|_{L^2}+\|\nabla^2\partial_2u\|_{L^2}$.

\par\textbf{Step I. Estimate of $\|\partial_2u\|_{L^2}$}

Multiplying \eqref{type-wave-1}$_2$ with $\partial_2 u$ in $L^2$ and integrating it over $\mathbb{R}^2$ yields
\begin{align*}
\alpha\|\partial_2u\|^2_{L^2}&=\gamma\frac{d}{dt}\int_{\R^2}\partial_2u\cdot\partial_t bdx+\frac{d}{dt}\int_{\R^2}\partial_2u\cdot bdx-\gamma\int_{\R^2}\partial_t \partial_2u\cdot \partial_tbdx-\int_{\R^2}\partial_t\partial_2 u\cdot bdx\nonumber\\
&\quad+\int_{\R^2}(u\cdot\nabla) b\cdot\partial_2u dx-\eta\int_{\R^2}\Delta b\cdot\partial_2u dx-\int_{\R^2}(b\cdot\nabla) u\cdot\partial_2u dx.
\end{align*}

Based on the velocity field equation \eqref{type-wave-1}$_1$, integration  by parts and Lemma \ref{2.1} we have
\begin{align}
-\gamma\int_{\R^2}\partial_t\partial_2 u\cdot \partial_tbdx
&=-\gamma\int_{\R^2}\partial_2((b\cdot\nabla) b+\alpha\partial_2b-(u\cdot\nabla) u-\mu (0,u_2)^{T})\cdot\partial_t bdx\nonumber\\
&\leq \frac{\gamma\alpha}{2}\|\nabla b\|_{L^2}^2+\frac{\gamma\alpha+\gamma\mu}{2}\|\partial_t b\|_{H^1}^2+\frac{\gamma\mu}{2}\|u_2\|_{L^2}^2\nonumber\\
&\quad+ C\| (u,b)\|_{H^{2}}\left(\|\nabla b\|_{H^{1}}^2+\|\nabla u\|_{H^{1}}^2+\|\partial_t b\|_{H^1}^2\right).\label{J1}
\end{align}
Similarly, one has
\begin{align*}
-\int_{\R^2}\partial_t\partial_2 u\cdot bdx&\leq C\left(\|\nabla b\|_{L^2}^2+\|u_2\|_{L^2}^2\right)+ C\| (u,b)\|_{H^{2}}\left(\|\nabla b\|_{H^{1}}^2+\|\nabla u\|_{H^{1}}^2\right).
\end{align*}

Thanks to  Lemma \ref{2.1} again, one gets
\begin{align*}
&\int_{\R^2}(u\cdot\nabla) b\cdot\partial_2u dx-\int_{\R^2}(b\cdot\nabla) u\cdot\partial_2u dx\nonumber \\&\leq C\|\partial_2 u\|_{L^{2}}\|u\|_{L^{2}}^{\frac{1}{2}}\|\partial_1 u\|_{L^{2}}^{\frac{1}{2}}\|\nabla b\|_{L^{2}}^{\frac{1}{2}}\|\nabla\partial_2 b\|_{L^{2}}^{\frac{1}{2}}\nonumber\\
&\quad+ C\|\partial_2 u\|_{L^{2}}\|b\|_{L^{2}}^{\frac{1}{2}}\|\partial_1 b\|_{L^{2}}^{\frac{1}{2}}\|\nabla u\|_{L^{2}}^{\frac{1}{2}}\|\nabla\partial_2 u\|_{L^{2}}^{\frac{1}{2}}\nonumber\\
&\leq C\| (u,b)\|_{H^{2}}\left(\|\nabla b\|_{H^{1}}^2+\|\nabla u\|_{H^{1}}^2\right).
\end{align*}

In summary, we obtain
\begin{align}\label{L2}
\alpha\|\partial_2u\|^2_{L^2}&\leq2\gamma\frac{d}{dt}\int_{\R^2}\partial_2u\cdot\partial_t bdx+2\frac{d}{dt}\int_{\R^2}\partial_2u\cdot bdx+C\left(\|\nabla b\|_{H^1}^2+\|\partial_t b\|_{H^1}^2+\|u_2\|_{L^2}^2\right)\nonumber\\
&\quad+ C\| (u,b)\|_{H^{2}}\left(\|\nabla b\|_{H^{1}}^2+\|\nabla u\|_{H^{1}}^2+\|\partial_tb\|_{H^1}^2\right),
\end{align}
where we have used
\begin{align*}
-\eta\int_{\R^2}\Delta b\cdot\partial_2u dx\leq C\| \nabla b\|_{H^{1}}^2+\frac{\alpha}{2}\|\partial_2 u\|_{L^{2}}^2.
\end{align*}

\vskip.1in
\par\textbf{Step II. Estimate of $\|\nabla^2\partial_2u\|_{L^2}$}

Applying $\nabla^2$ to \eqref{type-wave-1}$_2$ and dotting the resultant with $\nabla^2\partial_2u$ in $L^2$, we get
\begin{align*}
\alpha\|\nabla^2\partial_2u\|^2_{L^2}&=\gamma\frac{d}{dt}\int_{\R^2}\nabla^2\partial_2u\cdot\partial_t\nabla^2 bdx+\frac{d}{dt}\int_{\R^2}\nabla^2\partial_2u\cdot\nabla^2 bdx-\gamma\int_{\R^2}\partial_t\nabla^2\partial_2 u\cdot \partial_t\nabla^2bdx\nonumber\\
&\quad-\int_{\R^2}\partial_t\nabla^2\partial_2 u\cdot \nabla^2bdx+\int_{\R^2}\nabla^2(u\cdot\nabla b)\cdot\nabla^2\partial_2u dx\nonumber\\
&\quad-\eta\int_{\R^2}\nabla^2\Delta b\cdot\nabla^2\partial_2u dx-\int_{\R^2}\nabla^2(b\cdot\nabla u)\cdot\nabla^2\partial_2u dx\nonumber\\
&:=\gamma\frac{d}{dt}\int_{\R^2}\nabla^2\partial_2u\cdot\partial_t\nabla^2 bdx+\frac{d}{dt}\int_{\R^2}\nabla^2\partial_2u\cdot\nabla^2 bdx+\sum_{i=1}^{5}J_{i}.
\end{align*}

To estimate $J_{1}$, according to the velocity equation \eqref{type-wave-1}$_1$, we obtain
\begin{align*}
J_{1}
&=-\gamma\int_{\R^2}\nabla^2\partial_2(b\cdot\nabla b+\alpha\partial_2b-u\cdot\nabla u-\mu (0,u_2)^{T})\cdot \partial_t\nabla^2bdx\nonumber\\
&:=J_{11}+J_{12}+J_{13}+J_{14}.
\end{align*}

With the aid of Lemma \ref{2.1}, we obtain
\begin{align*}
J_{11}
&=-\gamma\int_{\R^2}\partial_t\nabla^2 b\cdot\partial_2((\nabla^2b\cdot\nabla)b+2(\nabla b\cdot\nabla)\nabla b+(b\cdot\nabla)\nabla^2 b)dx\nonumber\\
&\leq C\|\nabla^2 b\|_{L^{2}}^{\frac{1}{2}}\|\nabla^2\partial_1 b\|_{L^{2}}^{\frac{1}{2}}\|\nabla^2 b\|_{L^{2}}^{\frac{1}{2}}\|\nabla^2 \partial_2 b\|_{L^{2}}^{\frac{1}{2}}\|\partial_t\nabla^2 b\|_{L^{2}}\nonumber\\
&\quad+C(\| b\|_{L^{\infty}}+\|\nabla b\|_{L^{\infty}})\|\nabla^2 \partial_tb\|_{L^{2}}\|\nabla^3 b\|_{H^{1}}\nonumber\\
&\leq C\| b\|_{H^{3}}(\|\partial_t\nabla^2 b\|_{L^{2}}^2+\|\nabla^3 b\|_{H^{1}}^2)
\end{align*}
and similarly,
\begin{align*}
J_{13}&=-\gamma\int_{\R^2}\partial_t\nabla^2b\cdot\partial_2\big((\nabla^2u\cdot\nabla)u+2(\nabla u\cdot\nabla)\nabla u\big)dx+\gamma\int_{\R^2}\partial_t\nabla^2\partial_2 b\cdot(u\cdot\nabla)\nabla^2udx\nonumber\\
&\leq C\| u\|_{H^{3}}\left(\|\partial_t\nabla^2  b\|_{H^{1}}^2+\|\nabla^3 u\|_{L^{2}}^2\right).
\end{align*}

Clearly, we have
\begin{align*}
J_{12}+J_{14}
&\leq \gamma\alpha\|\partial_t\nabla^2b\|_{L^2}\|\nabla^4 b\|_{L^2}+\gamma\mu\|\partial_t\nabla^2b\|_{L^2}\|\nabla^3u_2\|_{L^2}.
\end{align*}

Therefore one has
\begin{align}\label{JJJ1}
J_{1}&\leq \frac{\gamma\alpha+\gamma\mu}{2}\|\partial_t\nabla^2 b\|_{L^2}^2+\frac{\gamma\alpha}{2}\|\nabla^4 b\|_{L^2}^2+\frac{\gamma\mu}{2}\|\nabla^3u_2\|_{L^2}^2\nonumber\\
&\quad+ C\| (u,b)\|_{H^{3}}\left(\|\partial_t\nabla^2 b\|_{H^1}^2+\|\nabla^4 b\|_{L^2}^2+\|\nabla^3u\|_{L^2}^2\right).
\end{align}

In view of similar argument to control $J_{1}$, it can be obtained that
\begin{align*}
J_{2}\leq&C\left(\|\nabla b\|_{H^3}^2+\|u_2\|_{H^3}^2\right)+ \| (u,b)\|_{H^{3}}\left(\|\nabla b\|_{H^{3}}^2+\|\nabla u\|_{H^{2}}^2\right).
\end{align*}

It follows from  Lemma \ref{2.1} that
\begin{align}\label{L3}
J_{3}
&=\int_{\R^2}((\nabla^2u\cdot\nabla) b+2(\nabla  u\cdot\nabla)\nabla b+(u\cdot\nabla)\nabla^2 b)\cdot\nabla^2\partial_2 udx\nonumber\\
&\leq C\|\nabla^2\partial_2 u\|_{L^{2}}\|\nabla^2u\|_{L^{2}}^{\frac{1}{2}}\|\nabla^2\partial_1 u\|_{L^{2}}^{\frac{1}{2}}\|\nabla b\|_{L^{2}}^{\frac{1}{2}}\|\nabla\partial_2 b\|_{L^{2}}^{\frac{1}{2}}+C\|u\|_{L^{\infty}}\|\nabla^3 b\|_{L^{2}}\| \nabla^2\partial_2u\|_{L^{2}}\nonumber\\
&\quad+C\|\nabla^2\partial_2 u\|_{L^{2}}\|\nabla^2b\|_{L^{2}}^{\frac{1}{2}}\|\nabla^2\partial_2 b\|_{L^{2}}^{\frac{1}{2}}\|\nabla u\|_{L^{2}}^{\frac{1}{2}}\|\nabla\partial_1 u\|_{L^{2}}^{\frac{1}{2}}\nonumber\\
&\leq C\| (u,b)\|_{H^{3}}\left(\|\nabla^3 b\|_{H^{1}}^2+\|\nabla^2u\|_{H^{1}}^2\right)
\end{align}
and similarly,
\begin{align}\label{L4}
J_{5}
\leq C\| (u,b)\|_{H^{3}}\left(\|\nabla^3 b\|_{H^{2}}^2+\|\nabla^2 u\|_{H^{1}}^2\right).
\end{align}
To sum up, we have
\begin{align}\label{A22}
\alpha\|\nabla^2\partial_2u\|^2_{L^2}&\leq2\gamma\frac{d}{dt}\int_{\R^2}\nabla^2\partial_2u\cdot\partial_t\nabla^2 bdx+2\frac{d}{dt}\int_{\R^2}\nabla^2\partial_2u\cdot\nabla^2bdx\nonumber\\
&\quad+C\left(\|\partial_t b\|_{H^2}^2+\|\nabla b\|_{H^3}^2+\|u_2\|_{H^3}^2\right)\nonumber\\
&\quad+ C\| (u,b)\|_{H^{3}}\left(\|\partial_t b\|_{H^{3}}^2+\|\nabla b\|_{H^{3}}^2+\|\nabla u\|_{H^{2}}^2\right),
\end{align}
where we have used
\begin{align*}
J_{4}&\leq C\| \nabla b\|_{H^{3}}^2+\frac{\alpha}{2}\|\nabla^2\partial_2u\|_{L^{2}}^2.
\end{align*}

\vskip.1in
Consequently, from \eqref{L2} and \eqref{A22}, we infer
\begin{align*}
\alpha\|\partial_2u\|^2_{H^2}&\leq2\frac{d}{dt}\int_{\R^2}\big(\gamma(\partial_2u\cdot\partial_t b+\nabla^2\partial_2u\cdot\partial_t\nabla^2 b)+\partial_2u\cdot b+\nabla^2\partial_2u\cdot\nabla^2b\big)dx\\
&\quad+C\left(\|\partial_t b\|_{H^3}^2+\|\nabla b\|_{H^3}^2+\|u_2\|_{H^3}^2\right)\\
&\quad+ C\| (u,b)\|_{H^{3}}\left(\|\partial_t b\|_{H^{3}}^2+\|\nabla b\|_{H^{3}}^2+\|\nabla u\|_{H^{2}}^2\right).
\end{align*}
Thus
\begin{align*}
&\alpha\int_{0}^{t}\|\partial_2u(\tau)\|^2_{H^2}d\tau\\
&\leq C\big(\|u_0\|_{H^3}^2+\|b_0\|_{H^2}^2+\|a_0\|_{H^2}^2+\|u\|_{H^3}^2+\|b\|_{H^2}^2+\|\partial_tb\|_{H^2}^2\big)\\
&\quad+C\int_{0}^{t}\left(\|\partial_{\tau} b(\tau)\|_{H^3}^2+\|\nabla b(\tau)\|_{H^3}^2+\|u_2(\tau)\|_{H^3}^2\right)d\tau\\
&\quad+C\sup_{0\leq \tau\leq t}\| (u,b)\|_{H^{3}}\int_{0}^{t}\left(\|\partial_{\tau} b(\tau)\|_{H^{3}}^2+\|\nabla b(\tau)\|_{H^{3}}^2+\| u_2(\tau)\|_{H^{3}}^2+\|\partial_2 u(\tau)\|_{H^{2}}^2\right)d\tau,
\end{align*}
where we have used $\|\partial_1 u\|_{H^2}=\|\nabla u_2\|_{H^2}$.

It follows that
$$\mathcal{E}_{02}(t)\leq C\left(\mathcal{E}_{01}(0)+\mathcal{E}_{01}(t)+\mathcal{E}_{01}^{\frac{3}{2}}(t)+\mathcal{E}_{02}^{\frac{3}{2}}(t)\right)$$
and we then  complete the proof of \eqref{in3.2}.
\end{proof}

\subsection{Completeness of the proof of Theorem \ref{Thm1}}
 With \eqref{in3.1} and \eqref{in3.2}, the proof of Theorem \ref{Thm1} will be proved by the bootstrapping argument.
\par
Multiplying \eqref{in3.2} by a suitable constant and then adding the resultant to \eqref{in3.1} yield
\begin{align}\label{bro}
\mathcal{E}_{0}(t)\leq&  C_1(\mathcal{E}_{0}(0)+\mathcal{E}_{0}^{\frac{3}{2}}(0))+C_2\mathcal{E}_{0}^{\frac{3}{2}}(t)
+C_3\mathcal{E}_{0}^{2}(t).
\end{align}
In order to apply the bootstrapping argument, we make the ansatz that
\begin{align*}
\mathcal{E}_0(t)\leq M,
\end{align*}
where
\begin{align*}
M=\min\left\{\frac{1}{(4C_2)^2},\,\,\frac{1}{4C_3} \right\}.
\end{align*}

It then follows from \eqref{intia} and \eqref{bro} that
\begin{align*}
\mathcal{E}_{0}(t)\leq& 2C_1\big(\mathcal{E}_{0}(0)+ \mathcal{E}_{0}^{\frac{3}{2}}(0)\big)\leq 2C_1\epsilon^2.
\end{align*}
If $\epsilon$ satisfies $$\epsilon^2\leq \min\left\{\frac{1}{64C_1C_2^2},\,\,\frac{1}{16C_1C_3} \right\},$$
then
\begin{align*}
 \mathcal{E}_{0}(t)\leq2C_1\epsilon^2\leq\frac{M}{2}.
\end{align*}
The bootstrapping argument then implies \eqref{stability} holds for all $t>0$. i.e.,
$$\mathcal{E}_{0}(t)\leq C\epsilon^2.$$ Then Theorem \ref{Thm1} is completed.
\hfill$\square$

\vskip.3in
\section{Proof of Theorem \ref{Thm2}}
\label{Dec}

 This section is devoted to proving Theorem \ref{Thm2}.  The proof takes advantage of anisotropic Sobolev inequalities and the stability result \eqref{stability}. For the sake of clarity, we divide this section into five subsections. The first subsection establishes the bound for negative Sobolev energy functional $\mathcal{E}_1(t)$. The subsequent three subsections establish the decay rates of $\|\nabla^{k} u(t)\|_{H^{3-k}}$, $\|\nabla^{k} b(t)\|_{H^{4-k}}$ and $\|\partial_t \nabla^{k} b(t)\|_{H^{3-k}}$ for $k=0,1,2$, respectively. The final subsection then extracts the decay rates for $\|\nabla^{3} u_2(t)\|_{L^{2}}$, $\|\nabla^{3} b_2(t)\|_{H^{1}}$ and  $\|\nabla^{3}\partial_t b_2(t)\|_{L^{2}}$.

\subsection{Estimates for $\mathcal{E}_1(t)$}

\par This subsection establishes the bound for negative Sobolev energy functional $\mathcal{E}_1(t)$ which will be used to control the low-order norm of ${E}_0(t)$. Now, we prove the following proposition:
\begin{Prop}\label{Prop11}
Assume that $(u,b)$ is a smooth solution to \eqref{type-wave-1} and
\begin{align*}
\Lambda^{-1} u_0,\Lambda^{-1} b_0,\Lambda^{-1}a_0
\in L^2(\R^2),
\end{align*}
then we have
\begin{align}\label{4.1}
&\mathcal{E}_1(t)\leq C.
\end{align}
\end{Prop}

 To obtain \eqref{4.1}, recalling \eqref{E-1}, we establish the desired upper bound for $\mathcal{E}_{11}(t)$ and $\mathcal{E}_{12}(t)$, respectively.
\begin{Lem}
Assume that $(u,b)$ is a smooth solution to \eqref{type-wave-1}, then there exist positive constants $\delta$ and $C$ such that
\begin{align}\label{FF}
&(1-\frac{C\epsilon^2}{4\delta}-C\epsilon^2)\|\Lambda^{-1} u\|^2_{L^2}+(1-C\epsilon^2)\|\Lambda^{-1} b\|^2_{L^2}+\frac{3\gamma^2}{2}\|\partial_t\Lambda^{-1}b\|^2_{L^2}
+2\gamma\eta\|b\|^2_{L^2}\nonumber\\
&+\int_0^t\left(3\mu\|\Lambda^{-1}u_2(\tau)\|^2_{L^2}+(3\eta-C\epsilon^2)\|b(\tau)\|^2_{L^2}+(3\gamma-C\epsilon^2)\|\partial_{\tau}\Lambda^{-1}b(\tau)\|^2_{L^2}\right)d\tau\nonumber\\
&\leq
C\left(\|\Lambda^{-1} u_0\|^2_{L^2}+\|\Lambda^{-1} b_0\|^2_{L^2}+2\gamma^2\|\Lambda^{-1}a_0\|^2_{L^2}
+2\gamma\eta\|b_0\|^2_{L^2}\right)\nonumber\\
&\quad+\left(12\gamma\alpha^2+6\delta+C\epsilon^2\right)\int_0^t\|\Lambda^{-1}\partial_2u(\tau)\|_{L^2}^2d\tau+C\epsilon^4.
\end{align}
\end{Lem}
\begin{proof}
Applying the operator $\Lambda^{-1}$ to \eqref{type-wave-1}$_1$ and \eqref{type-wave-1}$_2$, and multiplying the resulting equations by $\Lambda^{-1}u$ and $\Lambda^{-1}b$, respectively, we have
\begin{align}\label{4.2}
&\frac{1}{2}\frac{d}{dt}\big(\|\Lambda^{-1}u\|^2_{L^2}+\|\Lambda^{-1}b\|^2_{L^2}+2\gamma\int_{\R^2}\partial_t\Lambda^{-1}b\cdot\Lambda^{-1}bdx\big)+\mu\|\Lambda^{-1}u_2\|^2_{L^2}+\eta\|b\|^2_{L^2}-\gamma\|\partial_t\Lambda^{-1}b\|^2_{L^2}\nonumber\\
&=-\int_{\R^2}\Lambda^{-1}(u\cdot\nabla u)\cdot\Lambda^{-1}udx+\int_{\R^2}\Lambda^{-1}(b\cdot\nabla b)\cdot\Lambda^{-1}udx\nonumber\\
&\quad +\int_{\R^2}\Lambda^{-1}(b\cdot\nabla u)\cdot\Lambda^{-1}bdx-\int_{\R^2}\Lambda^{-1}(u\cdot \nabla b)\cdot\Lambda^{-1}bdx\nonumber\\
&:=K_1+K_2+K_3+K_4.
\end{align}
Then, applying the operator $\Lambda^{-1}$ to \eqref{type-wave-1}$_2$ and multiplying the resulting equation by $\partial_t\Lambda^{-1}b$ yield
\begin{align}\label{4.3}
&\frac{1}{2}\frac{d}{dt}(\gamma\|\partial_t\Lambda^{-1}b\|^2_{L^2}+\eta\|b\|^2_{L^2})+\|\partial_t\Lambda^{-1}b\|^2_{L^2}\nonumber\\
&=\int_{\R^2}\Lambda^{-1}(b\cdot\nabla u)\cdot\partial_t\Lambda^{-1}bdx-\int_{\R^2}\Lambda^{-1}(u\cdot\nabla b)\cdot\partial_t\Lambda^{-1}bdx+\alpha\int_{\R^2}\Lambda^{-1}\partial_2u\cdot\partial_t\Lambda^{-1}bdx\nonumber\\
&:=K_5+K_6+K_7.
\end{align}
By multiplying \eqref{4.3} by $2\gamma$ and adding it to \eqref{4.2}, we obtain
\begin{align}\label{4.4}
&\frac{1}{2}\frac{d}{dt}\left(\|\Lambda^{-1} u\|^2_{L^2}+\|\Lambda^{-1} b\|^2_{L^2}+2\gamma^2\|\partial_t\Lambda^{-1}b\|^2_{L^2}
+2\gamma\eta\|b\|^2_{L^2}+2\gamma\int_{\R^2}\partial_t\Lambda^{-1}b\cdot\Lambda^{-1}bdx
\right)\nonumber\\
&+\mu\|\Lambda^{-1}u_2\|^2_{L^2}+\eta\|b\|^2_{L^2}+\gamma\|\partial_t\Lambda^{-1}b\|^2_{L^2}\nonumber\\
&=\sum_{i=1}^{4}K_i+2\gamma\sum_{i=5}^{7}K_i.
\end{align}

To deal with $K_1$, we firstly note that
\begin{align*}
-\int_{\R^2}\Lambda^{-1}\partial_1(u_1 u_1)\cdot\Lambda^{-1}u_1dx
&=-\int_{\R^2}\Lambda^{-1}(u_1 u_1)\cdot\Lambda^{-1}\partial_2u_2dx\\
&=\int_{\R^2}\Lambda^{-1}\partial_2(u_1 u_1)\cdot\Lambda^{-1}u_2dx\\
&\leq
\|\Lambda^{-1}u_2\|_{L^{2}}\|u_1u_1\|_{L^{2}}\\
&\leq\frac{\mu}{8}\|\Lambda^{-1}u_2\|_{L^{2}}^2+C\|u_1\|_{L^{2}}^2\|u_1\|_{H^{2}}^2.
\end{align*}
Thus for any $\delta>0$, by H\"{o}lder's inequality and Young's inequality, one has
\begin{align*}
K_{1}&=-\sum_{i=1}^2\int_{\R^2}\Lambda^{-1}\partial_i(u_i u_1)\cdot\Lambda^{-1}u_1dx-\sum_{i=1}^2\int_{\R^2}\Lambda^{-1}\partial_i(u_i u_2)\cdot\Lambda^{-1}u_2dx\\
&\leq\frac{\mu}{4}\|\Lambda^{-1}u_2\|_{L^{2}}^2+C\|u_1\|_{L^{2}}^2\|u_1\|_{H^{2}}^2+\delta\|u_1\|_{L^{2}}^2+\frac{1}{4\delta}\|\Lambda^{-1}u_1\|_{L^{2}}^2\|u_2\|_{H^{2}}^2\\
&\quad+C\|u\|_{H^{2}}^2\|u_2\|_{L^{2}}^2,
\end{align*}
where we have used that the Riesz operator $\mathcal{R}_i=\partial_i(-\Delta)^{-\frac 12}=\partial_i\Lambda^{-1}$ is $L^2$ bounded.

Based on similar arguments to estimate $K_1$, we get
\begin{align*}
K_{2}&\leq\frac{\mu}{4}\|\Lambda^{-1}u_2\|_{L^{2}}^2+\frac{\eta}{6}\|b_2\|_{L^{2}}^2+C\|b_1\|_{H^2}^2\|\Lambda^{-1}u_1\|_{L^{2}}^2+C\|b\|_{H^{2}}^2\|b\|_{L^{2}}^2.
\end{align*}

By means of H\"{o}lder's inequality and Young's inequality, we obtain
\begin{align*}
K_{3}+K_{4}&=\int_{\R^2}\Lambda^{-1}\nabla\cdot(u\otimes b)\cdot\Lambda^{-1}bdx-\int_{\R^2}\Lambda^{-1}\nabla\cdot(b\otimes u)\cdot\Lambda^{-1}bdx\nonumber\\
&\leq \frac{\eta}{6}\|b\|_{L^{2}}^2+C\|u\|_{H^{2}}^2\|\Lambda^{-1}b\|_{L^{2}}^2,\nonumber\\
2\gamma(K_{5}+K_6)
&\leq
2\gamma(\|\Lambda^{-1}\nabla\cdot(b\otimes u)\|_{L^{2}}+\|\Lambda^{-1}\nabla\cdot(u\otimes b)\|_{L^{2}})\|\partial_t\Lambda^{-1}b\|_{L^{2}}\nonumber\\
&\leq
\frac{\eta}{6}\|b\|_{L^{2}}^2+C\|u\|_{H^{2}}^2\|\partial_t\Lambda^{-1}b\|_{L^{2}}^2,
\end{align*}
and
\begin{align*}
2\gamma K_{7}
&\leq
2\gamma\alpha\|\Lambda^{-1}\partial_2u\|_{L^{2}}\|\partial_t\Lambda^{-1}b\|_{L^{2}}\leq
2\gamma\alpha^2\|\Lambda^{-1}\partial_2u\|_{L^{2}}^2+\frac{\gamma}{2}\|\partial_t\Lambda^{-1}b\|_{L^{2}}^2.
\end{align*}

Putting all estimates of $K_1-K_7$ into \eqref{4.4}, integrating in time and using \eqref{stability}, we have
\begin{align*}
&(1-\frac{C\epsilon^2}{4\delta}-C\epsilon^2)\|\Lambda^{-1} u\|^2_{L^2}+(1-C\epsilon^2)\|\Lambda^{-1} b\|^2_{L^2}\nonumber\\
&+2\gamma^2\|\partial_t\Lambda^{-1}b\|^2_{L^2}
+2\gamma\eta\|b\|^2_{L^2}+2\gamma\int_{\R^2}\partial_t\Lambda^{-1}b\cdot\Lambda^{-1}bdx
\nonumber\\
&+2\int_0^t\left(\frac{\mu}{2}\|\Lambda^{-1}u_2(\tau)\|^2_{L^2}+(\frac{\eta}{2}-C\epsilon^2)\|b(\tau)\|^2_{L^2}+(\frac{\gamma}{2}-C\epsilon^2)\|\partial_{\tau}\Lambda^{-1}b(\tau)\|^2_{L^2}\right)d\tau\nonumber\\
&\leq
\|\Lambda^{-1} u_0\|^2_{L^2}+\|\Lambda^{-1} b_0\|^2_{L^2}+2\gamma^2\|\Lambda^{-1}a_0\|^2_{L^2}
+2\gamma\eta\|b_0\|^2_{L^2}+2\gamma\int_{\R^2}\Lambda^{-1}a_0\cdot\Lambda^{-1}b_0dx\nonumber\\
&\quad+\left(2\gamma\alpha^2+\delta+C\epsilon^2\right)\int_0^t\|\Lambda^{-1}\partial_2u(\tau)\|_{L^2}^2d\tau+C\epsilon^4,
\end{align*}
where we have used the fact $\|u_1\|_{L^2}=\|\Lambda^{-1}\partial_2u\|_{L^2}$.

With the help of H\"{o}lder's inequality and Young's inequality, we have
\begin{align*}
2\gamma\int_{\R^2}\partial_t\Lambda^{-1}b\cdot\Lambda^{-1}bdx
&\leq \frac{2}{3}\|\Lambda^{-1}b\|_{L^2}^2+\frac{3\gamma^2}{2}\|\partial_t\Lambda^{-1} b\|_{L^2}^2,
\end{align*}
and hence we can get the desired inequality \eqref{FF}.
\end{proof}

\begin{Lem}\label{LEM1}
Assume that $(u,b)$ is a smooth
 solution to \eqref{type-wave-1}, then there exist a positive constant $C$ such that
\begin{align}\label{FFF}
&-\gamma(\partial_t\Lambda^{-1}b,\Lambda^{-1}\partial_2u)-(\Lambda^{-1}b,\Lambda^{-1}\partial_2u)+\frac{\alpha}{2}\int_{0}^t\|\Lambda^{-1}\partial_2u(\tau)\|_{L^2}^2d\tau\nonumber\\
& \leq\int_{0}^t\Big((\frac{\gamma(\alpha+\mu)}{2}+C\epsilon^2)\|\partial_{\tau}\Lambda^{-1}b(\tau)\|_{L^2}^2+(\alpha+\frac{\mu+\eta}{2}+\frac{C\epsilon^2}{2\eta}+\frac{C\epsilon^2}{\alpha})\|b(\tau)\|_{L^2}^2
\nonumber\\
&\qquad \ \ \ \ +\frac{\mu}{2}\|\Lambda^{-1}u_2(\tau)\|_{L^2}^2\Big)d\tau+  C\epsilon^2,
\end{align}
where $(\cdot,\cdot)$ denotes the $L^2$-inner product.
\end{Lem}
\begin{proof}
Applying the operator $\Lambda^{-1}$ to \eqref{type-wave-1}$_2$ and multiplying the resulting equation by $\Lambda^{-1}\partial_2 u$ yield
 \begin{align}\label{F22}
&-\gamma\frac{d}{dt}(\partial_t\Lambda^{-1}b,\Lambda^{-1}\partial_2u)-\frac{d}{dt}(\Lambda^{-1}b,\Lambda^{-1}\partial_2u)+\alpha\|\Lambda^{-1}\partial_2u\|_{L^2}^2\nonumber\\
&=-\gamma(\partial_t\Lambda^{-1}b,\partial_t\Lambda^{-1}\partial_2u)-(\Lambda^{-1}b,\partial_t\Lambda^{-1}\partial_2u)-\eta(\Delta\Lambda^{-1} b,\Lambda^{-1}\partial_2u)\nonumber\\
&\quad+(\Lambda^{-1}(u\cdot\nabla b),\Lambda^{-1}\partial_2u)-(\Lambda^{-1}(b\cdot\nabla u),\Lambda^{-1}\partial_2u)\nonumber\\
&:=K_8+K_9+K_{10}+K_{11}+K_{12}.
\end{align}

\par For $K_8$, using the velocity equation \eqref{type-wave-1}$_1$, we have
\begin{align*}
K_8
&=-\gamma\int_{\R^2}\partial_t \Lambda^{-1}b\cdot\Lambda^{-1}\partial_2 (b\cdot\nabla b+\alpha\partial_2b-u\cdot\nabla u-\mu (0,u_2)^{T})dx\\
&\leq\frac{\gamma\alpha}{2}\|\partial_2b\|_{L^2}^2+\frac{\gamma(\alpha+\mu)}{2}\|\partial_t\Lambda^{-1}b\|_{L^2}^2+
\frac{\gamma\mu}{2}\|u_2\|_{L^2}^2\nonumber\\
&\quad+C\|\partial_t\Lambda^{-1} b\|_{L^{2}}^2(\| u\|_{H^{2}}^2+\| b\|_{H^{2}}^2)+C(\|\nabla u\|_{L^{2}}^2+\|\nabla b\|_{L^{2}}^2),
\end{align*}
where we have used  H\"{o}lder's inequality, Young's inequality and $H^2\hookrightarrow L^\infty$.

Similarly, it follows from H\"{o}lder's inequality, Young's inequality and $H^2\hookrightarrow L^\infty$ that
\begin{align*}
K_9
&=-\int_{\R^2} \Lambda^{-1}b\cdot \Lambda^{-1}\partial_2(b\cdot\nabla b+\alpha\partial_2b-u\cdot\nabla u-\mu (0,u_2)^{T})dx
\\&=\int_{\R^2} \Lambda^{-1}\partial_2 b\cdot\Lambda^{-1}(b\cdot\nabla b+\alpha\partial_2b-u\cdot\nabla u-\mu (0,u_2)^{T})dx\\
&\leq (\alpha+\frac{\mu+\eta}{2})\|b\|_{L^2}^2+
\frac{\mu}{2}\|\Lambda^{-1}u_2\|_{L^2}^2+\frac{1}{2\eta}\| b\|_{L^{2}}^2\| b\|_{H^{2}}^2\\
&\quad+\frac{3}{2\alpha}\| b\|_{L^{2}}^2\|u\|_{H^{2}}^2+\frac{\alpha}{6}(\|\Lambda^{-1}\partial_2 u\|_{L^{2}}^2+\| u_2\|_{L^{2}}^2),
\end{align*}
where we have used $\Lambda^{-1}(b\cdot\nabla b)=\sum\limits_{i=1}^2\Lambda^{-1}\partial_i(b_ib)$,  $\Lambda^{-1}(u\cdot\nabla u)=\sum\limits_{i=1}^2\Lambda^{-1}\partial_i(u_iu)$ and
\begin{align*}
 \|u\|_{L^{2}}^2 = \|u_1\|_{L^{2}}^2+ \|u_2\|_{L^{2}}^2=\|\Lambda^{-1}\partial_2 u\|_{L^{2}}^2+\| u_2\|_{L^{2}}^2.
\end{align*}

Moreover, there hold
\begin{align*}
&K_{10}
\leq\frac{3\eta^2}{2\alpha}\|\nabla b\|_{L^{2}}^2+\frac{\alpha}{6}\|\Lambda^{-1}\partial_2 u\|_{L^{2}}^2,
\end{align*}
and
\begin{align*}
K_{11}+K_{12}
\leq \frac{\alpha}{6}\|\Lambda^{-1}\partial_2 u\|_{L^{2}}^2+\frac{3}{2\alpha}\|u\|_{H^{2}}^2\| b\|_{L^{2}}^2.
\end{align*}

Finally, by integrating \eqref{F22} with respect to time, together with the estimates for $K_8-K_{12}$ and \eqref{stability}, we arrive at \eqref{FFF}.
\end{proof}

 Combining Lemma \ref{FF} and Lemma \ref{LEM1}, we now complete the proof of  Proposition \ref{Prop11}.

 \textit{Proof of Proposition} \ref{Prop11}. According to Lemma \ref{FF} and \ref{LEM1}, for a sufficiently small $\kappa_1$, a direct calculation of \eqref{FF}+$\kappa_1\cdot$\eqref{FFF} yields
 \begin{align*}
&(1-\frac{C\epsilon^2}{4\delta}-C\epsilon^2)\|\Lambda^{-1} u\|^2_{L^2}+(1-C\epsilon^2)\|\Lambda^{-1} b\|^2_{L^2}\nonumber\\
&+\frac{3\gamma^2}{2}\|\partial_t\Lambda^{-1}b\|^2_{L^2}
+2\gamma\eta\|b\|^2_{L^2}-\gamma\kappa_1(\partial_{t}\Lambda^{-1}b,\Lambda^{-1}\partial_2u)-\kappa_1(\Lambda^{-1}b,\Lambda^{-1}\partial_2u)\nonumber\\
&+\int_0^t\Big((3\mu-\frac{\kappa_1\mu}{2})\|\Lambda^{-1}u_2(\tau)\|^2_{L^2}
+\big(3\eta-C\epsilon^2-\kappa_1(\alpha+\frac{\mu+\eta}{2}+\frac{C\epsilon^2}{2\eta}+\frac{C\epsilon^2}{\alpha})\big)\|b(\tau)\|^2_{L^2}\nonumber\\
&\ \ \ \
+\big((3\gamma-C\epsilon^2)-\kappa_1(\frac{\gamma(\alpha+\mu)}{2}+C\epsilon^2)\big)\|\partial_{\tau}\Lambda^{-1}b(\tau)\|^2_{L^2}
\nonumber\\
&\ \ \ \ +(\frac{\alpha\kappa_1}{2}-(12\gamma\alpha^2+6\delta+C\epsilon^2))\|\Lambda^{-1}\partial_2u(\tau)\|_{L^2}^2\Big)d\tau\nonumber\\
&\leq
C\left(\|\Lambda^{-1} u_0\|^2_{L^2}+\|\Lambda^{-1} b_0\|^2_{L^2}+2\gamma^2\|\Lambda^{-1}a_0\|^2_{L^2}
+2\gamma\eta\|b_0\|^2_{L^2}\right)+C(\epsilon^4+\epsilon^2).
\end{align*}
Since $\kappa_1,\alpha,\epsilon$ and $\delta$ are mutually independent and can be taken arbitrarily small, we then obtain \eqref{4.1}.
\hfill$\square$

\subsection{The decay rates for $\| u(t)\|_{H^{3}},\|b(t)\|_{H^{4}}$ and $
		\|\partial_t b(t)\|_{H^{3}}$}
To prove the decay rates for $\| u(t)\|_{H^{3}},\|b(t)\|_{H^{4}}$ and $
		\|\partial_t b(t)\|_{H^{3}}$, the following proposition establishes the desired bound for ${E}_0(t)$.
\begin{Prop}\label{Prop1}
For some constant $C>0$, it holds that
\begin{align}\label{key2}
		E_0(t)\leq C.
\end{align}
\end{Prop}
 To prove \eqref{key2} we firstly show the following two lemmas. The first lemma focuses on bounding the time-weighted energy $(1+t)\|(u,b,\nabla b,\partial_tb)\|^2_{H^3}$ while the second lemma handles the inner product $(1+t)(\partial_tb,\partial_2u)_{H^2}$ to generate the time-weighted dissipation $(1+t)\|\partial_2u\|_{H^2}^2$.
 \begin{Lem}\label{key3}
Assume that $(u,b)$ is a smooth solution to \eqref{type-wave-1}, then we have
\begin{align}\label{inenergy11}
&\frac{1}{2}\frac{d}{dt}(1+t)\Big(\|u\|^2_{H^3}+\| b\|^2_{H^3}+2\gamma^2\|\partial_tb\|^2_{H^3}
+2\gamma\eta\|\nabla b\|^2_{H^3}+2\gamma(\partial_tb,b)_{H^3}\nonumber\\
&\qquad\qquad\quad+4\gamma\int_{\R^2}\big((b\cdot\nabla)\nabla^3 b\cdot\nabla^3 u+\gamma\nabla^3\partial_{t}b\cdot (u\cdot\nabla)\nabla^3 b+\alpha\nabla^3u\cdot\nabla^3\partial_2 b\big)dx\nonumber\\
&\qquad\qquad\quad +4\gamma\int_{\R^2}\big((b\cdot\nabla) b\cdot u+\gamma\partial_{t}b\cdot (u\cdot\nabla) b+\alpha u\cdot\partial_2 b\big)dx\Big)\nonumber\\
&+(1+t)\big((\frac{\mu}{3}-C(\epsilon+\epsilon^2))\|u_2\|^2_{H^3}+(\eta-2\gamma\alpha^2-\frac{3\gamma^2\alpha^2\mu}{2}-C(\epsilon+\epsilon^2))\|\nabla b\|^2_{H^3}\nonumber\\
&\qquad \ \ \ \ \ \ +(\gamma-C(\epsilon+\epsilon^2))\|\partial_{t}b\|^2_{H^3}-C(\epsilon+\epsilon^2)\|\partial_{2}u\|_{H^2}^2\big) \nonumber\\&\leq\frac{1}{2}\left(\|u\|^2_{H^3}+\| b\|^2_{H^3}+2\gamma^2\|\partial_tb\|^2_{H^3}
+2\gamma\eta\|\nabla b\|^2_{H^3}+2\gamma(\partial_t b,b)_{H^3}\right)+\frac{C}{\epsilon}\|u\|_{H^3}^2.
\end{align}
\end{Lem}

\begin{proof}
Just like the steps \eqref{3.1}-\eqref{mhd33} and \eqref{3.18}-\eqref{3.19}, direct calculation yields
\begin{align}\label{E1-1}
&\frac{1}{2}\frac{d}{dt}(1+t)\left(\|u\|^2_{H^3}+\| b\|^2_{H^3}+2\gamma^2\|\partial_tb\|^2_{H^3}
+2\gamma\eta\|\nabla b\|^2_{H^3}+2\gamma(\partial_tb,b)_{H^3}
\right)\nonumber\\
&\quad+(1+t)\left(\mu\|u_2\|^2_{H^3}+\eta\|\nabla b\|^2_{H^3}+\gamma\|\partial_tb\|^2_{H^3}\right)\nonumber\\
&=\frac{1}{2}\left(\|u\|^2_{H^3}+\| b\|^2_{H^3}+2\gamma^2\|\partial_tb\|^2_{H^3}
+2\gamma\eta\|\nabla b\|^2_{H^3}+2\gamma(\partial_t b,b)_{H^3}
\right)\nonumber\\
&\quad+2\gamma(1+t)((b\cdot\nabla) u,\partial_tb)_{H^3} -2\gamma(1+t)((u\cdot\nabla) b,\partial_t b)_{H^3}\nonumber\\
&\quad+2\gamma\alpha(1+t)(\partial_2u,\partial_tb)_{H^3}-(1+t)((u\cdot\nabla) u,u)_{H^3}\nonumber\\
&\quad+(1+t)((b\cdot\nabla) b,u)_{H^3} +(1+t)((b\cdot\nabla) u,b)_{H^3}-(1+t)((u\cdot \nabla) b,b)_{H^3}.
\end{align}

From \eqref{mhd333} and \eqref{mhd32}, we have
\begin{align}\label{E1-2}
&\frac{1}{2}\frac{d}{dt}(1+t)\Big(\|u\|^2_{H^3}+\| b\|^2_{H^3}+2\gamma^2\|\partial_tb\|^2_{H^3}
+2\gamma\eta\|\nabla b\|^2_{H^3}+2\gamma(\partial_tb,b)_{H^3}\Big)\nonumber\\
&\quad+(1+t)\big(\frac{\mu}{3}\|u_2\|^2_{H^3}+(\eta-2\gamma\alpha^2-\frac{3\gamma^2\alpha^2\mu}{2})\|\nabla b\|^2_{H^3}+\gamma\|\partial_{t}b\|^2_{H^3}\big)\nonumber\\
& \leq
\frac{1}{2}\left(\|u\|^2_{H^3}+\| b\|^2_{H^3}+2\gamma^2\|\partial_tb\|^2_{H^3}
+2\gamma\eta\|\nabla b\|^2_{H^3}+2\gamma(\partial_t b,b)_{H^3}
\right)\nonumber\\
&\quad+CV(t)(1+t)\left(\|\partial_{t} b(t)\|_{H^3}^2+\|\nabla b(t)\|_{H^3}^2+\|u_2(t)\|_{H^3}^2+\|\partial_{2}u(t)\|_{H^2}^2\right)\nonumber\\&\quad-2\gamma(1+t)\frac{d}{dt}\int_{\R^2}\big((b\cdot\nabla)\nabla^3 b\cdot\nabla^3 u+\gamma\nabla^3\partial_{t}b\cdot (u\cdot\nabla)\nabla^3 b+\alpha\nabla^3u\cdot\nabla^3\partial_2 b\big)dx\nonumber\\
&\quad-2\gamma(1+t)\frac{d}{dt}\int_{\R^2}\big((b\cdot\nabla) b\cdot u+\gamma\partial_{t}b\cdot (u\cdot\nabla) b+\alpha u\cdot\partial_2 b\big)dx,
\end{align}
where $V(t)$ is defined in \eqref{VT}.

For the third and fourth lines on the right-hand side of \eqref{E1-2}, direct computation yields
\begin{align}\label{E1-22}
&-2\gamma(1+t)\frac{d}{dt}\int_{\R^2}\big((b\cdot\nabla)\nabla^3 b\cdot\nabla^3 u+\gamma\partial_{t}\nabla^3b\cdot (u\cdot\nabla)\nabla^3 b+\alpha\nabla^3u\cdot\nabla^3\partial_2 b\big)dx\nonumber\\
&-2\gamma(1+t)\frac{d}{dt}\int_{\R^2}\big((b\cdot\nabla) b\cdot u+\gamma\partial_{t}b\cdot (u\cdot\nabla) b+\alpha u\cdot\partial_2 b\big)dx\nonumber\\
&=-2\gamma\frac{d}{dt}(1+t)\int_{\R^2}\big((b\cdot\nabla)\nabla^3 b\cdot\nabla^3 u+\gamma\partial_{t}\nabla^3b\cdot (u\cdot\nabla)\nabla^3 b+\alpha\nabla^3u\cdot\nabla^3\partial_2 b\big)dx\nonumber\\
&\quad+2\gamma\int_{\R^2}\big((b\cdot\nabla)\nabla^3 b\cdot\nabla^3 u+\gamma\partial_{t}\nabla^3b\cdot (u\cdot\nabla)\nabla^3 b+\alpha\nabla^3u\cdot\nabla^3\partial_2 b\big)dx\nonumber\\
&\quad-2\gamma\frac{d}{dt}(1+t)\int_{\R^2}\big((b\cdot\nabla) b\cdot u+\gamma\partial_{t}b\cdot (u\cdot\nabla) b+\alpha u\cdot\partial_2 b\big)dx\nonumber\\
&\quad+2\gamma\int_{\R^2}\big((b\cdot\nabla) b\cdot u+\gamma\partial_{t}b\cdot (u\cdot\nabla) b+\alpha u\cdot\partial_2 b\big)dx\nonumber\\
&\leq-2\gamma\frac{d}{dt}(1+t)\int_{\R^2}\big((b\cdot\nabla)\nabla^3 b\cdot\nabla^3 u+\gamma\partial_{t}\nabla^3b\cdot (u\cdot\nabla)\nabla^3 b+\alpha\nabla^3u\cdot\nabla^3\partial_2 b\big)dx\nonumber\\
&\quad-2\gamma\frac{d}{dt}(1+t)\int_{\R^2}\big((b\cdot\nabla) b\cdot u+\gamma\partial_{t}b\cdot (u\cdot\nabla) b+\alpha u\cdot\partial_2 b\big)dx\nonumber\\
&\quad+ C\epsilon(1+t)\left(\|\partial_{t} b\|_{H^3}^2+\|\nabla b\|_{H^3}^2+\|u_2\|_{H^3}^2+\|\partial_{2}u\|_{H^2}^2\right)+\frac{C}{\epsilon}\|u\|_{H^3}^2.
\end{align}
Here, we have used that
\begin{align*}
2\gamma\int_{\R^2}(b\cdot\nabla) b\cdot udx&\leq C\|b\|_{L^2}^{\frac{1}{2}}\|\partial_1b\|_{L^2}^{\frac{1}{2}}\|u\|_{L^2}^{\frac{1}{2}}\|\partial_2u\|_{L^2}^{\frac{1}{2}}\|\nabla b\|_{L^2}\nonumber\\
&\leq C\|(u,b)\|_{L^2}(\|\nabla b\|_{L^2}^{2}+\|\partial_2u\|_{L^2}^2)\nonumber\\
&\leq C\epsilon(1+t)(\|\nabla b\|_{L^2}^{2}+\|\partial_2u\|_{L^2}^2).
\end{align*}

Submitting \eqref{E1-2} and \eqref{E1-22} into \eqref{E1-1}, and using the bound $V(t)\leq C(\epsilon+\epsilon^2)$
derived from \eqref{stability}, we immediately find \eqref{inenergy11}.
\end{proof}

Next, we estimate the inner product $(1+t)(\partial_tb,\partial_2u)_{H^2}$ and prove the following lemma:
 \begin{Lem}\label{key4}
Assume that $(u,b)$ is a smooth
 solution to \eqref{type-wave-1}, then we have
\begin{align}\label{inenergy2}
&-2\gamma\frac{d}{dt}(1+t)(\partial_tb,\partial_2u)_{H^2}+(\alpha-C\epsilon)(1+t)\|\partial_2u\|_{H^2}^2\nonumber\\
&\leq(1+t)\Big((3\gamma(\mu+\alpha)+\frac{2}{\alpha}+C\epsilon)\|\partial_t b\|_{H^3}^2+(3\gamma\alpha+\frac{2\eta^2}{\alpha}+C\epsilon)\|\nabla b\|_{H^3}^2\nonumber\\
&\qquad\quad\quad\ \ +(3\gamma\mu+C\epsilon)\|u_2\|_{H^3}^2\Big)+\frac{\gamma^2}{\alpha}\|\partial_t b\|_{H^2}^2+\alpha\|\partial_2u\|_{H^2}^2.
\end{align}
\end{Lem}
 \begin{proof}
 Invoking the equation of  \eqref{type-wave-1}$_2$, we have
 \begin{align*}
&-\gamma\frac{d}{dt}(1+t)(\partial_tb,\partial_2u)_{H^2}+\alpha(1+t)\|\partial_2u\|_{H^2}^2\nonumber\\
&=-\gamma(\partial_tb,\partial_2u)_{H^2}-\gamma(1+t)(\partial_tb,\partial_t\partial_2u)_{H^2}+(1+t)(\partial_tb,\partial_2u)_{H^2}\nonumber\\
&\quad-\eta(1+t)(\Delta b,\partial_2u)_{H^2}+(1+t)(u\cdot\nabla b,\partial_2u)_{H^2}-(1+t)(b\cdot\nabla u,\partial_2u)_{H^2}\nonumber\\
&:=L_{1}+L_2+L_3+L_4+L_5+L_6.
\end{align*}

By H\"{o}lder's inequality and Young's inequality, we obtain
 \begin{align*}
&L_1\leq \frac{\gamma^2}{2\alpha}\|\partial_t b\|_{H^2}^2+\frac{\alpha}{2}\|\partial_2u\|_{H^2}^2,\\
&L_3\leq \frac{1}{\alpha}(1+t)\|\partial_tb\|_{H^2}^2+\frac{\alpha}{4}(1+t)\|\partial_2u\|_{H^2}^2,\\
&L_4\leq \frac{\eta^2}{\alpha}(1+t)\|\Delta b\|_{H^2}^2+\frac{\alpha}{4}(1+t)\|\partial_2u\|_{H^2}^2,
\\&L_5\leq (1+t)\| u\|_{H^2}\| \partial_2u\|_{H^2}\|\nabla b\|_{H^2},
\\&L_6\leq(1+t)\| \nabla u\|_{H^2}^2\| b\|_{H^2}.
\end{align*}

For $L_2$, from \eqref{J1} and \eqref{JJJ1}, it can be obtained that
\begin{align*}
L_2&=-\gamma(1+t)\int_{\R^2}\big(\partial_t\partial_2 u\cdot \partial_tb+\partial_t \nabla\partial_2u\cdot \partial_t\nabla b+\partial_t\nabla^2\partial_2 u\cdot\partial_t \nabla^2b\big)dx\\
&\leq(1+t)\bigg(\frac{3\gamma(\mu+\alpha)}{2}\|\partial_t b\|_{H^3}^2+\frac{3\gamma\alpha}{2}\|\nabla b\|_{H^3}^2+\frac{3\gamma\mu}{2}\|u_2\|_{H^3}^2\\
&\qquad\qquad\quad+ C\| (u,b)\|_{H^{3}}\left(\|\partial_t b\|_{H^{3}}^2+\|\nabla b\|_{H^{3}}^2+\|\nabla u\|_{H^{2}}^2\right)\bigg).
\end{align*}

Collecting all the above estimates $L_1-L_6$ and \eqref{stability}, we can complete the proof.
 \end{proof}

  Now, putting together the two lemmas above yields Proposition \ref{Prop1}.

 \textit{Proof of Proposition} \ref{Prop1}. Thanks to Lemma \ref{key3} and \ref{key4}, for a sufficiently small $\kappa_2$, the combination \eqref{inenergy11}+$\kappa_2\cdot$\eqref{inenergy2} yields
\begin{align*}
&\frac{1}{2}\frac{d}{dt}(1+t)\Big(\|u\|^2_{H^3}+\| b\|^2_{H^3}+2\gamma^2\|\partial_tb\|^2_{H^3}
+2\gamma\eta\|\nabla b\|^2_{H^3}+2\gamma(\partial_tb,b)_{H^3}-4\kappa_2\gamma(\partial_tb,\partial_2u)_{H^2}\nonumber\\
&\qquad\quad\qquad\quad+4\gamma\int_{\R^2}\big((b\cdot\nabla)\nabla^3 b\cdot\nabla^3 u+\gamma\nabla^3\partial_{t}b\cdot (u\cdot\nabla)\nabla^3 b+\alpha\nabla^3u\cdot\nabla^3\partial_2 b\big)dx\nonumber\\
&\qquad\quad\qquad\quad+4\gamma\int_{\R^2}\big((b\cdot\nabla) b\cdot u+\gamma\partial_{t}b\cdot (u\cdot\nabla) b+\alpha u\cdot\partial_2 b\big)dx\Big)\nonumber\\
&+(1+t)\big((\frac{\mu}{3}-C(\epsilon+\epsilon^2)-\kappa_2(3\gamma\mu+C\epsilon))\|u_2\|^2_{H^3}+(\kappa_2(\alpha-C\epsilon)-C\epsilon-C\epsilon^2)\|\partial_2u\|_{H^2}^2\nonumber\\
&\qquad\quad\quad\quad +(\eta-2\gamma\alpha^2-\frac{3\gamma^2\alpha^2\mu}{2}-C(\epsilon+\epsilon^2)-\kappa_2(3\gamma\alpha+\frac{2\eta^2}{\alpha}+C\epsilon))\|\nabla b\|^2_{H^3}\nonumber\\
&\qquad\quad\quad\quad+(\gamma-C(\epsilon+\epsilon^2)-\kappa_2(3\gamma(\mu+\alpha)+\frac{2}{\alpha}+C\epsilon))\|\partial_tb\|^2_{H^3}
\big)\nonumber\\&\leq C\left(\|u\|^2_{H^3}+\| b\|^2_{H^3}+2\gamma^2\|\partial_tb\|^2_{H^3}
+2\gamma\eta\|\nabla b\|^2_{H^3}+2\gamma(\partial_t b,b)_{H^3}
\right)+\frac{C}{\epsilon}\|u\|_{H^3}^2.
\end{align*}

After integrating over $[0,t]$ and using \eqref{stability} and \eqref{4.1}, we have
\begin{align*}
&(1+t)\Big(\|u\|^2_{H^3}+\| b\|^2_{H^3}+2\gamma^2\|\partial_tb\|^2_{H^3}
+2\gamma\eta\|\nabla b\|^2_{H^3}+2\gamma(\partial_tb,b)_{H^3}-4\kappa_2\gamma(\partial_tb,\partial_2u)_{H^2}\nonumber\\
&\qquad\qquad+4\gamma\int_{\R^2}\big((b\cdot\nabla)\nabla^3 b\cdot\nabla^3 u+\gamma\nabla^3\partial_{t}b\cdot (u\cdot\nabla)\nabla^3 b+\alpha\nabla^3u\cdot\nabla^3\partial_2 b\big)dx\nonumber\\
&\qquad\qquad+4\gamma\int_{\R^2}\big((b\cdot\nabla) b\cdot u+\gamma\partial_{t}b\cdot (u\cdot\nabla) b+\alpha u\cdot\partial_2 b\big)dx\Big)\nonumber\\
&+2\int_{0}^{t}(1+\tau)\big((\frac{\mu}{3}-C(\epsilon+\epsilon^2)-\kappa_2(3\gamma\mu+C\epsilon))\|u_2(\tau)\|^2_{H^3}\nonumber\\
&\ \ \ \ \ \ \ \ +(\eta-2\gamma\alpha^2-\frac{3\gamma^2\alpha^2\mu}{2}-C(\epsilon+\epsilon^2)-\kappa_2(3\gamma\alpha+\frac{2\eta^2}{\alpha}+C\epsilon))\|\nabla b(\tau)\|^2_{H^3}\nonumber\\
&\ \ \ \ \ \ \ \ +(\gamma-C(\epsilon+\epsilon^2)-\kappa_2(3\gamma(\mu+\alpha)+\frac{2}{\alpha}+C\epsilon))\|\partial_{\tau}b(\tau)\|^2_{H^3}\nonumber\\
&\ \ \ \ \ \ \ \ +(\kappa_2(\alpha-C\epsilon)-C\epsilon-C\epsilon^2)\|\partial_2u(\tau)\|_{H^2}^2\big)d\tau\nonumber\\
&\leq C\int_{0}^{t}\left(\|u(\tau)\|^2_{H^3}+\| b(\tau)\|^2_{H^3}+2\gamma^2\|\partial_{\tau}b(\tau)\|^2_{H^3}
+2\gamma\eta\|\nabla b(\tau)\|^2_{H^3}+\frac{C}{\epsilon}\|u(\tau)\|_{H^3}^2\right)d\tau+C.
\end{align*}

Noting that
\begin{align*}
&2\gamma(\partial_tb,b)_{H^3}
\leq \frac{2}{3}\|b\|_{H^3}^2+\frac{3\gamma^2}{2}\|\partial_t b\|_{H^3}^2,\\
&4\gamma\alpha\int_{\R^2}u\cdot\partial_2bdx
\leq \frac{2}{3}\|u\|_{L^2}^2+6\gamma^2\alpha^2\|\nabla b\|_{L^2}^2,\\
&4\gamma\alpha\int_{\R^2}\nabla^3u\cdot\nabla^3\partial_2bdx
\leq \frac{2}{3}\|\nabla^3u\|_{L^2}^2+6\gamma^2\alpha^2\|\nabla^4 b\|_{L^2}^2,\\
&-4\kappa_2\gamma(\partial_tb,\partial_2u)_{H^2}\leq 2\gamma\kappa_2\|\partial_tb\|_{H^2}^2+2\gamma\kappa_2\|u\|_{H^3}^2,
\end{align*}
and
 \begin{align*}
&4\gamma\int_{\R^2}\big((b\cdot\nabla) b\cdot u+\gamma\partial_{t}b\cdot (u\cdot\nabla) b\big)dx\nonumber\\&\quad+4\gamma\int_{\R^2}\big((b\cdot\nabla)\nabla^3 b\cdot\nabla^3 u+\gamma\partial_{t}\nabla^3b\cdot (u\cdot\nabla)\nabla^3 b\big)dx\nonumber\\
&\leq C\epsilon(\|\nabla b\|_{H^3}^2+\|\nabla u\|_{H^2}^2+\|\partial_t b\|_{H^3}^2)
\end{align*}
where we have used \eqref{stability}.
 Therefore, if $\alpha,\kappa_2$ and $\epsilon$ are sufficiently small, we can get \eqref{key2}. \hfill$\square$

\subsection{The decay rates for $\|\nabla^2 u(t)\|_{H^{1}},\|\nabla^2 b(t)\|_{H^{2}}$ and $
		\|\partial_t\nabla^2 b(t)\|_{H^{1}}$}
This subsection is devoted to proving $a$ $priori$ estimates for $E_2(t)$ and then the decay rates for $\|\nabla^2 u(t)\|_{H^{1}},\|\nabla^2 b(t)\|_{H^{2}}$ and $\|\partial_t\nabla^2 b(t)\|_{H^{1}}$.
\begin{Prop}\label{Prop2}
For some constant $C>0$, it holds that
\begin{align}\label{as3}
		E_2(t)
		\leq C.
\end{align}
\end{Prop}
The proof of Proposition \ref{Prop2} is long, so we divide it into two lemmas.
\begin{Lem}\label{keyy}
Assume that $(u,b)$ is a smooth solution to \eqref{type-wave-1}, then we have
\begin{align}\label{boundE}
&\frac{1}{2}\frac{d}{dt}(1+t)^3\Big(\|\nabla^2 u\|^2_{H^1}+\|\nabla^2 b\|^2_{H^1}+2\gamma^2\|\partial_t\nabla^2b\|^2_{H^1}
+2\gamma\eta\|\nabla^{3}b\|^2_{H^1}+2\gamma(\partial_t\nabla^2b,\nabla^2b)_{H^1}
\nonumber\\
&\qquad\qquad\quad+4\gamma\sum_{m=2}^3\int_{\mathbb{R}^2}
		\big((b\cdot\nabla)\nabla^m b\cdot\nabla^m u
		+\alpha\nabla^m u\cdot\nabla^m\partial_2 b
		+\gamma\partial_t\nabla^m b\cdot (u\cdot\nabla)\nabla^m b\big)dx\nonumber\\
&+(\frac{\mu}{3}-C\epsilon-C\epsilon^2)(1+t)^3\|\nabla^{2}u_2\|^2_{H^1}+(\eta-2\gamma\alpha^2-\frac{3\gamma^2\alpha^2\mu}{2}-C\epsilon-C\epsilon^2)(1+t)^3\|\nabla^{3}b\|^2_{H^1}\nonumber\\
&+(\gamma-C\epsilon-C\epsilon^2)(1+t)^3\|\partial_t\nabla^2b\|^2_{H^1}- C(\epsilon+\epsilon^2)(1+t)^3\|\nabla^2\partial_{2}u\|_{L^2}^2\nonumber\\
&\leq C(1+t)^3(\|\nabla^2 u\|_{L^2}^2+\|\nabla^2 b\|_{L^2}^2+\|\partial_t\nabla^2 b\|_{L^{2}}^2)(\|b\|_{H^3}^2+\| u_2\|_{H^3}^2+\| \partial_2u\|_{H^2}^2+\|\partial_t b\|_{H^3}^2)\nonumber\\
&\quad+\frac{C}{4\epsilon}(1+t)(\|u_2\|_{H^3}^2+\|\partial_2 u\|_{H^2}^2+\|\nabla b\|_{H^3}^2+\|\partial_t b\|_{H^3}^2).
\end{align}
\end{Lem}

\begin{proof}
Applying the operator $\nabla^2$ to \eqref{type-wave-1}$_1$ and \eqref{type-wave-1}$_2$, and then taking the $H^1$-inner product of the resulting equations with $\nabla^2 u$ and $\nabla^2 b$, we have
\begin{align}\label{4.10}
&\frac{1}{2}\frac{d}{dt}(\|\nabla^2u\|^2_{H^1}+\|\nabla^2b\|^2_{H^1}+2\gamma(\partial_t\nabla^2b,\nabla^2b)_{H^1})+\mu\|\nabla^{2}u_2\|^2_{H^1}+\eta\|\nabla^{3}b\|^2_{H^1}-\gamma\|\partial_t\nabla^2b\|^2_{H^1}\nonumber\\
&=-\left(\nabla^2(u\cdot\nabla u),\nabla^2u\right)_{H^1}+\left(\nabla^2(b\cdot\nabla b),\nabla^2u\right)_{H^1}\nonumber\\
&\quad+\left(\nabla^2(b\cdot\nabla u),\nabla^2b\right)_{H^1}-\left(\nabla^2(u\cdot\nabla b),\nabla^2b\right)_{H^1}.
\end{align}

 We apply the operator $\nabla^2$ to \eqref{type-wave-1}$_2$, and take the $H^1$-inner product of the resulting equations with $\partial_t\nabla^2b$, then we obtain
\begin{align}\label{4.11}
&\frac{1}{2}\frac{d}{dt}(\gamma\|\partial_t\nabla^2b\|^2_{H^1}+\eta\|\nabla^{3}b\|^2_{H^1})+\|\partial_t\nabla^2b\|^2_{H^1}\nonumber\\
&=(\nabla^2(b\cdot\nabla u),\partial_t\nabla^2 b)_{H^1} +\alpha(\nabla^2\partial_2u\cdot\partial_t\nabla^2b)_{H^1}-(\nabla^2(u\cdot\nabla b)\cdot\partial_t\nabla^2 b)_{H^1}.
\end{align}

Multiplying \eqref{4.11} by $2\gamma$ and adding it to \eqref{4.10}, then multiplying the resultant by the time weight $(1+t)^3$ leads to
\begin{align*}
&\frac{1}{2}\frac{d}{dt}(1+t)^3\left(\|\nabla^2 u\|^2_{H^1}+\|\nabla^2 b\|^2_{H^1}+2\gamma^2\|\partial_t\nabla^2b\|^2_{H^1}
+2\gamma\eta\|\nabla^{3}b\|^2_{H^1}+2\gamma(\partial_t\nabla^2b,\nabla^2b)_{H^1}
\right)\nonumber\\
&+\mu(1+t)^3\|\nabla^{2}u_2\|^2_{H^1}+\eta(1+t)^3\|\nabla^{3}b\|^2_{H^1}+\gamma(1+t)^3\|\partial_t\nabla^2b\|^2_{H^1}\nonumber\\
&=\frac{3}{2}(1+t)^2\left(\|\nabla^2 u\|^2_{H^1}+\|\nabla^2 b\|^2_{H^1}+2\gamma^2\|\partial_t\nabla^2b\|^2_{H^1}
+2\gamma\eta\|\nabla^{3}b\|^2_{H^1}+2\gamma(\partial_t\nabla^2b,\nabla^2b)_{H^1}
\right)\nonumber\\
&\quad+2\gamma(1+t)^3(\nabla^2(b\cdot\nabla u),\partial_t\nabla^2 b)_{H^1} +2\gamma\alpha(1+t)^3(\nabla^2\partial_2u,\partial_t\nabla^2b)_{H^1}\nonumber\\
&\quad-2\gamma(1+t)^3(\nabla^2(u\cdot\nabla b),\partial_t\nabla^2 b)_{H^1}-(1+t)^3(\nabla^2(u\cdot\nabla u),\nabla^2u)_{H^1}\nonumber\\
&\quad+(1+t)^3(\nabla^2(b\cdot\nabla b),\nabla^2u)_{H^1} +(1+t)^3(\nabla^2(b\cdot\nabla u),\nabla^2b)_{H^1}-(1+t)^3(\nabla^2(u\cdot \nabla b),\nabla^2b)_{H^1}\nonumber\\
&:=(1+t)^2M_1+(1+t)^3\sum_{i=2}^8M_i.
\end{align*}

By the Gagliardo-Nirenberg interpolation inequality, we have
\begin{align*}
\|\nabla^2 u\|_{L^2}\leq C\| \nabla u\|_{L^2}^{\frac{1}{2}}\|\nabla^{3}u\|_{L^2}^{\frac{1}{2}},
\end{align*}
and then we get
\begin{align*}
(1+t)^2\|\nabla^2 u\|^2_{H^1}
&=(1+t)^2(\|\nabla^2 u\|^2_{L^2}+\|\nabla^3 u\|^2_{L^2})\nonumber\\
&\leq C\left((1+t)\|\nabla u\|_{H^2}^2\right)^{\frac{1}{2}}\left((1+t)^3\|\nabla^3 u\|_{L^2}^2\right)^{\frac{1}{2}}\nonumber\\&\leq C\epsilon(1+t)^3\|\nabla^3 u\|_{L^2}^2+\frac{C}{4\epsilon}(1+t)\|\nabla u\|_{H^2}^2.
\end{align*}
The other terms in $M_{1}$ can be estimated similarly and then there holds
\begin{align}\label{m1-0}
(1+t)^2M_{1}
&\leq C\epsilon(1+t)^3(\|\nabla^2 u_2\|_{H^1}^2+\|\nabla^2\partial_2 u\|_{L^2}^2+\|\nabla^3 b\|_{H^1}^2+\|\nabla^2\partial_t b\|_{H^1}^2)\nonumber\\&\quad+\frac{C}{4\epsilon}(1+t)(\|\nabla u\|_{H^3}^2+\|\partial_2 u\|_{H^2}^2+\|\nabla b\|_{H^3}^2+\|\partial_t b\|_{H^3}^2).
\end{align}

For $M_{2}$, we use Leibniz's formula to decompose it into the following three terms,
\begin{align}\label{GKKK}
M_2
&=2\gamma\Big(\int_{\R^2}\big((\nabla^2b\cdot\nabla) u+2(\nabla b\cdot\nabla)\nabla u\big)\cdot\partial_t\nabla^2 bdx\nonumber\\
&\quad\quad\quad+\int_{\R^2}\big((\nabla^3b\cdot\nabla) u+3(\nabla^2b\cdot\nabla)\nabla u+3(\nabla b\cdot \nabla)\nabla^2 u\big)\cdot\partial_t\nabla^3 bdx\Big)\nonumber\\
&\quad+2\gamma\int_{\R^2}(b\cdot\nabla) \nabla^3 u\cdot\partial_t\nabla^3 bdx+2\gamma\int_{\R^2}(b\cdot\nabla) \nabla^2 u\cdot\partial_t\nabla^2 bdx\nonumber\\
&:= M_{21}+M_{22}+M_{23}.
\end{align}

By Lemma \ref{2.1}, Young's inequality  and \eqref{stability}, we deduce
\begin{align}\label{H-1}
M_{21}&\leq
C\|\nabla^2b\|_{L^{2}}^{\frac{1}{2}}\|\nabla^2\partial_2b\|_{L^{2}}^{\frac{1}{2}}\|\nabla u\|_{L^{2}}^{\frac{1}{2}}\|\nabla\partial_1 u\|_{L^{2}}^{\frac{1}{2}}\|\partial_t\nabla^2 b\|_{L^2}+C\|\nabla b\|_{L^{\infty}}\|\nabla^2u\|_{L^{2}}\|\partial_t\nabla^2 b\|_{L^2}\nonumber\\
&\quad+C\|\nabla^3b\|_{L^{2}}^{\frac{1}{2}}\|\nabla^3\partial_2b\|_{L^{2}}^{\frac{1}{2}}\|\nabla u\|_{L^{2}}^{\frac{1}{2}}\|\nabla\partial_1 u\|_{L^{2}}^{\frac{1}{2}}\|\partial_t\nabla^3 b\|_{L^2}\nonumber\\
&\quad+
\|\nabla^2 b\|_{L^{2}}^{\frac{1}{2}}\|\nabla^2\partial_2b\|_{L^{2}}^{\frac{1}{2}}\|\nabla^2 u\|_{L^{2}}^{\frac{1}{2}}\|\nabla^2\partial_1 u\|_{L^{2}}^{\frac{1}{2}}\|\partial_t\nabla^3 b\|_{L^2}+C\|\nabla b\|_{L^{\infty}}\|\nabla^3u\|_{L^{2}}\|\partial_t\nabla^3 b\|_{L^2}\nonumber\\
&\leq
C(\|u\|_{H^{2}}+\|b\|_{H^{3}})(\|\nabla^2 u_2\|_{H^{1}}^2+\|\nabla^2\partial_2 u\|_{L^{2}}^2+\|\partial_t\nabla^2 b\|_{H^1}^2+\|\nabla^3 b\|_{H^1}^2)\nonumber\\
&\quad +C\epsilon\|\partial_t\nabla^2 b\|_{L^2}^2+C\|\nabla^2 u\|_{H^{1}}^2\|\nabla b\|_{H^{2}}^2\nonumber\\
&\leq
C\epsilon(\|\nabla^2 u_2\|_{H^{1}}^2+\|\nabla^2\partial_2 u\|_{L^{2}}^2+\|\partial_t\nabla^2 b\|_{H^1}^2+\|\nabla^3 b\|_{H^1}^2)+C\|\nabla^2 u\|_{H^{1}}^2\|\nabla b\|_{H^{2}}^2,
\end{align}
where we have used $\|\nabla^m\partial_1 u\|_{L^{2}}=\|\nabla^m \nabla u_2\|_{L^{2}}$ for $m=1,2$ and $\|\nabla^3 u\|_{L^{2}}=\|\nabla^2\partial_1 u\|_{L^{2}}+\|\nabla^2\partial_2 u\|_{L^{2}}$.

For $M_{22}$, thanks to a similar argument to prove Proposition \ref{A1b} and \eqref{stability}, we have
\begin{align}\label{GGGG}
		M_{22}&\leq
		-2\gamma\frac{d}{dt}\int_{\mathbb{R}^2}
		\big((b\cdot\nabla)\nabla^3 b\cdot\nabla^3 u
		+\alpha\nabla^3 u\cdot\nabla^3\partial_2 b
		+\gamma\partial_t\nabla^3 b\cdot (u\cdot\nabla)\nabla^3 b\big)dx
		\notag\\ &\quad+C\left(\epsilon+\epsilon^2\right)\left(\|\partial_t\nabla^3b\|_{L^2}^2
+\|\nabla^2\partial_1u\|_{L^2}^2+\|\nabla^2\partial_2u\|_{L^2}^2+\|\nabla^3 b\|_{H^1}^2\right)\notag\\
		&\quad+\frac{2\mu}{3}\|\nabla^3 u_2\|_{L^2}^2	+\Bigl(2\gamma\alpha^2+\frac{3\gamma^2\alpha^2\mu}{2}\Bigr)
		\|\nabla^4 b\|_{L^2}^2-2\gamma\alpha\int_{\R^2}\nabla^3\partial_2u\cdot\partial_t\nabla^3 b  dx.
	\end{align}

To deal with  $M_{23}$ we divide into three steps to proceed. We will note that  $M_{23}$ generates many nonlinear terms and involves more significant adjustments relative to proof of Proposition 2.1.

 {\it {Step 1.}} In this step, we  use integration by parts and \eqref{type-wave-1}$_1$ to replace $\partial_tu$, and then get
\begin{align}\label{K}
M_{23}&=-2\gamma\int_{\R^2}(b\cdot\nabla)\partial_t\nabla^2 b\cdot\nabla^2 udx\nonumber\\
&=-2\gamma\frac{d}{dt}\int_{\R^2}(b\cdot\nabla)\nabla^2 b\cdot\nabla^2 udx+2\gamma\int_{\R^2}(\partial_t b\cdot\nabla)\nabla^2 b\cdot\nabla^2 udx\nonumber\\
&\quad+2\gamma\int_{\R^2}( b\cdot\nabla)\nabla^2 b\cdot\nabla^2\partial_t udx\nonumber\\
&=-2\gamma\frac{d}{dt}\int_{\R^2}(b\cdot\nabla)\nabla^2 b\cdot\nabla^2 udx+2\gamma\int_{\R^2}(\partial_t b\cdot\nabla)\nabla^2 b\cdot\nabla^2 udx-2\gamma\int_{\R^2}(b\cdot\nabla)\nabla^2 b\cdot\nabla^3pdx
\nonumber\\
&\quad+2\gamma\int_{\R^2}(b\cdot\nabla)\nabla^2 b\cdot\nabla^2(-\mu (0,u_2)^{T} +\alpha\partial_2 b)dx+2\gamma\int_{\R^2}(b\cdot\nabla)\nabla^2 b\cdot\nabla^2( (b \cdot \nabla) b)dx
\nonumber\\
&\quad-2\gamma\int_{\R^2}(b\cdot\nabla)\nabla^2 b\cdot\nabla^2( (u \cdot \nabla) u)dx.
\end{align}
By Young's inequality, we have
\begin{align}\label{K-1-1}
2\gamma\int_{\R^2}(\partial_tb\cdot\nabla)\nabla^2 b\cdot \nabla^2 u&\leq
C\|\partial_t b\|_{L^{\infty}}\|\nabla^3 b\|_{L^2}\|\nabla^2 u\|_{L^2}\nonumber\\
&\leq C\epsilon\|\nabla^3b\|_{L^2}^2+C\|\nabla^2 u\|_{L^2}^2\|\partial_tb\|_{H^2}^2.
\end{align}
By an argument similar to that in \eqref{SP}-\eqref{SP1} and invoking \eqref{stability}, we obtain
\begin{align*}
&-2\gamma\int_{\R^2}(b\cdot\nabla)\nabla^2 b\cdot\nabla^3 pdx
\nonumber\\&\leq
C\|b\|_{L^\infty}\|\nabla^3 b\|_{L^2}\|\nabla^3 p\|_{L^2}\nonumber\\
&\leq C\|b\|_{H^3}\|\nabla^3 b\|_{L^2}\big(\|\nabla\big(\nabla\cdot(u\cdot\nabla u)\big)\|_{L^2}+\|\nabla\big(\nabla\cdot(b\cdot\nabla b)\big)\|_{L^2}+\|\nabla\partial_2u_2\|_{L^2}\big)\nonumber\\
&\leq C\|b\|_{H^3}\|\nabla^3 b\|_{L^2}\Big(\|\nabla u\|_{L^{2}}^{\frac{1}{2}}\|\nabla\partial_{1} u\|_{L^{2}}^{\frac{1}{2}}
\|\nabla^2u\|_{L^{2}}^{\frac{1}{2}}\|\nabla^2\partial_{2}u\|_{L^{2}}^{\frac{1}{2}}\nonumber\\
&\qquad \ \ \ \ \ \ \ \ \ \ \ \ \ \ \ \ \ \ \ \  +\|\nabla b\|_{L^{2}}^{\frac{1}{2}}\|\nabla\partial_{1}b\|_{L^{2}}^{\frac{1}{2}}
\|\nabla^2b\|_{L^{2}}^{\frac{1}{2}}\|\nabla^2\partial_{2}b\|_{L^{2}}^{\frac{1}{2}}+\|\nabla^2u_2\|_{L^2}\Big)\nonumber\\
&\leq C\|b\|_{H^3}\|\nabla^3 b\|_{L^2}\Big(\|u\|_{H^{3}}(\|\nabla^2u_2\|_{L^{2}}+\|\nabla^2\partial_2u\|_{L^{2}})+\|\nabla^2b\|_{L^{2}}\|\nabla b\|_{H^{2}}+\|\nabla^2 u_2\|_{L^{2}}\Big)\nonumber\\
&\leq C\left(\epsilon+\epsilon^2\right)\left(\|\nabla^2\partial_2 u\|_{L^{2}}^2+\|\nabla^2 u_2\|_{L^2}^2+\|\nabla^3 b\|_{L^2}^2\right)+C\|\nabla^2 b\|_{L^2}^2\|\nabla b\|_{H^2}^2.
\end{align*}

Moreover, in view of \eqref{stability}, we have
\begin{align}\label{K-2}
&2\gamma\int_{\R^2}(b\cdot\nabla) \nabla^2 b\cdot\nabla^2(-\mu (0,u_2)^{T}  +\alpha\partial_2 b)dx
\nonumber\\&\leq
C\|b\|_{L^\infty}\|\nabla^3 b\|_{L^2}\|\nabla^2 u_2\|_{L^2}+C\|b\|_{L^\infty}\|\nabla^3 b\|_{L^2}^2\nonumber\\
&\leq \epsilon\left(\|\nabla^3 b\|_{L^2}^2+\|\nabla^2 u_2\|_{L^2}^2\right)
\end{align}
and
\begin{align}\label{K-21}
&2\gamma\int_{\R^2}(b\cdot\nabla) \nabla^2 b\cdot\nabla^2((b \cdot \nabla) b)dx\nonumber\\
&=2\gamma\int_{\R^2}(b\cdot\nabla)\nabla^2 b\cdot((\nabla^2b \cdot \nabla) b+2(\nabla b \cdot \nabla)\nabla b+(b \cdot \nabla)\nabla^2 b)dx\nonumber\\
&
\leq
C\|b\|_{L^\infty}\|\nabla^3 b\|_{L^2}\|\nabla^2 b\|_{L^2}\|\nabla b\|_{H^2}+C\|b\|_{L^\infty}^2\|\nabla^3 b\|_{L^2}^2\nonumber\\
&\leq C\epsilon^2\|\nabla^3 b\|_{L^2}^2+C\|\nabla^2 b\|_{L^2}^2\|\nabla b\|_{H^2}^2.
\end{align}
By integration by parts, we have
\begin{align}\label{K-3}
&-2\gamma\int_{\R^2}(b\cdot\nabla)\nabla^2 b\cdot\nabla^2( (u \cdot \nabla) u)dx
\nonumber\\
&=-2\gamma\int_{\R^2}(b\cdot\nabla)\nabla^2 b\cdot((\nabla^2u \cdot \nabla) u+2(\nabla u \cdot \nabla)\nabla u+(u \cdot \nabla)\nabla^2 u)dx\nonumber\\
&=-2\gamma\int_{\R^2}(b\cdot\nabla)\nabla^2 b\cdot((\nabla^2u \cdot \nabla) u+2(\nabla u \cdot \nabla)\nabla u)dx+2\gamma\int_{\R^2}((u\cdot\nabla) b\cdot\nabla)\nabla^2b \cdot \nabla^2 udx\nonumber\\
&\quad-2\gamma\int_{\R^2}((b\cdot\nabla)u\cdot\nabla)\nabla^2 b\cdot\nabla^2udx-2\gamma\int_{\R^2}(b\cdot\nabla)\nabla^2 u\cdot(u \cdot \nabla)\nabla^2 bdx\nonumber\\
&\leq
C\|b\|_{L^\infty}\|\nabla^3 b\|_{L^2}\|\nabla^2 u\|_{L^2}\|\nabla u\|_{L^{\infty}}+
C\|u\|_{L^\infty}\|\nabla^3 b\|_{L^2}\|\nabla^2 u\|_{L^2}\|\nabla b\|_{L^{\infty}}\nonumber\\
&\quad-2\gamma\int_{\R^2}(b\cdot\nabla)\nabla^2 u\cdot(u \cdot \nabla)\nabla^2 bdx\nonumber\\
&\leq C\epsilon^2\|\nabla^3b\|_{L^2}^2+C\|\nabla^2 u\|_{L^2}^2\left(\|\nabla u\|_{H^2}^2+\|\nabla b\|_{H^2}^2\right)-2\gamma\int_{\R^2}(b\cdot\nabla)\nabla^2 u\cdot(u \cdot \nabla)\nabla^2 bdx.
\end{align}

Inserting \eqref{K-1-1}-\eqref{K-3} into \eqref{K} implies
\begin{align}\label{K-11}
M_{23}
&\leq -2\gamma\frac{d}{dt}\int_{\R^2}(b\cdot\nabla)\nabla^2 b\cdot\nabla^2 udx-2\gamma\int_{\R^2}(b\cdot\nabla)\nabla^2 u\cdot(u \cdot \nabla)\nabla^2 bdx\nonumber\\
&\quad+C(\epsilon+\epsilon^2)(\|\nabla^3b\|_{L^2}^2+\|\nabla^2\partial_2 u\|_{L^2}^2+\|\nabla^2 u_2\|_{H^1}^2)\nonumber\\
&\quad+C(\|\nabla^2 u\|_{{L^2}}^2+\|\nabla^2b\|_{{L^2}}^2)\left(\|\nabla u\|_{H^2}^2+\|\nabla b\|_{H^2}^2+\|\partial_tb\|_{H^2}^2\right).
\end{align}

 {\it {Step 2.}} In this step, we will deal with the second term  on the right-hand side of \eqref{K-11}. Here, we use \eqref{type-wave-1}$_2$ to replace $(b\cdot \nabla) u$, and integration by parts and then have
\begin{align}\label{K-6}
&-2\gamma\int_{\R^2}(b\cdot \nabla)\nabla^2 u\cdot (u\cdot\nabla)\nabla^2 bdx\nonumber\\
&=-2\gamma\int_{\R^2}\nabla^2 ((b\cdot \nabla) u)\cdot( u\cdot\nabla)\nabla^2 bdx+2\gamma\int_{\R^2}(\nabla^2 b\cdot \nabla) u\cdot (u\cdot\nabla)\nabla^2 bdx\nonumber\\
&\quad+4\gamma\int_{\R^2}(\nabla b\cdot \nabla)\nabla u\cdot (u\cdot\nabla)\nabla^2 bdx\nonumber\\
&=-2\gamma^2\frac{d}{dt}\int_{\R^2}\partial_{t}\nabla^2 b\cdot (u\cdot\nabla)\nabla^2 bdx+2\gamma^2\int_{\R^2}\partial_{t}\nabla^2 b\cdot (\partial_{t}u\cdot\nabla)\nabla^2 bdx\nonumber\\
&\quad-2\gamma\int_{\R^2}\partial_{t}\nabla^2b\cdot (u\cdot\nabla )\nabla^2 bdx+2\gamma\eta\int_{\R^2}\nabla^2\Delta b\cdot (u\cdot\nabla)\nabla^2 bdx\nonumber\\
&\quad+2\gamma\alpha\int_{\R^2}(\partial_2u\cdot\nabla)\nabla^2 u\cdot\nabla^2 bdx+2\gamma\alpha\int_{\R^2}(u\cdot\nabla)\nabla^2 u\cdot\nabla^2\partial_2 bdx\nonumber\\
&\quad-2\gamma\int_{\R^2}\nabla^2((u\cdot\nabla) b)\cdot (u\cdot\nabla)\nabla^2 bdx+2\gamma\int_{\R^2}(\nabla^2 b\cdot \nabla) u\cdot (u\cdot\nabla)\nabla^2 bdx\nonumber\\
&\quad+4\gamma\int_{\R^2}(\nabla b\cdot \nabla)\nabla u\cdot (u\cdot\nabla)\nabla^2 bdx.
\end{align}

By making use of \eqref{stability} and \eqref{A8-8} we then get
\begin{align}\label{K-7}
&2\gamma^2\int_{\R^2}\partial_{t}\nabla^2 b\cdot (\partial_{t}u\cdot\nabla)\nabla^2 bdx-2\gamma\int_{\R^2}\partial_{t}\nabla^2b\cdot (u\cdot\nabla )\nabla^2 bdx\nonumber\\
&\quad+2\gamma\eta\int_{\R^2}\nabla^2\Delta b\cdot (u\cdot\nabla)\nabla^2 bdx\nonumber\\
&\leq C(\|\partial_{t} u\|_{L^{\infty}}+\| u\|_{L^{\infty}})(\|\partial_t\nabla^2 b\|_{L^{2}}^2+\|\nabla^3b\|_{H^{1}}^2)\nonumber\\
&\leq  C(\epsilon+\epsilon^2)\left(\|\nabla^3 b\|_{H^1}^2+\|\partial_t\nabla^2 b\|_{L^{2}}^2\right).
\end{align}

Moreover, there holds
\begin{align}\label{K-9}
2\gamma\alpha\int_{\R^2}(\partial_2u\cdot\nabla)\nabla^2 u\cdot\nabla^2 bdx
 &\leq
C\|\nabla^3 u\|_{L^2}\|\nabla^2 b\|_{L^2}^{\frac{1}{2}}\|\nabla^2\partial_2 b\|_{L^2}^{\frac{1}{2}}\|\partial_2 u\|_{L^2}^{\frac{1}{2}}\|\partial_1\partial_2 u\|_{L^2}^{\frac{1}{2}}\nonumber\\
&\leq C\epsilon\left(\|\nabla^3 b\|_{L^2}^2+\|\nabla^2\partial_2 u\|_{L^2}^2+\|\nabla^2 u_2\|_{H^1}^2\right).
\end{align}

Like \eqref{K-3}, one can  derive that
\begin{align}\label{K-8}
&-2\gamma\int_{\R^2}\nabla^2((u\cdot\nabla) b)\cdot (u\cdot\nabla)\nabla^2 bdx+2\gamma\int_{\R^2}(\nabla^2 b\cdot \nabla) u\cdot (u\cdot\nabla)\nabla^2 bdx\nonumber\\
&\quad+4\gamma\int_{\R^2}(\nabla b\cdot \nabla)\nabla u\cdot (u\cdot\nabla)\nabla^2 bdx\nonumber\\
&=-2\gamma\int_{\R^2}(\nabla^2 u\cdot\nabla) b\cdot (u\cdot\nabla)\nabla^2 bdx-4\gamma\int_{\R^2}(\nabla u\cdot\nabla) \nabla b\cdot (u\cdot\nabla)\nabla^2 bdx\nonumber\\&\quad-2\gamma\int_{\R^2}( u\cdot\nabla) \nabla^2 b\cdot (u\cdot\nabla)\nabla^2 bdx+2\gamma\int_{\R^2}(\nabla^2 b\cdot \nabla) u\cdot (u\cdot\nabla)\nabla^2 bdx\nonumber\\
&\quad+4\gamma\int_{\R^2}(\nabla b\cdot \nabla)\nabla u\cdot (u\cdot\nabla)\nabla^2 bdx\nonumber\\
&\leq C\epsilon^2\|\nabla^3b\|_{L^2}^2+C(\|\nabla^2u\|_{L^2}^2+\|\nabla^2 b\|_{L^2}^2)(\|\nabla u\|_{H^2}^2+\|\nabla b\|_{H^2}^2).
\end{align}

Consequently, by submitting \eqref{K-7}-\eqref{K-8} into \eqref{K-6} one has
\begin{align}\label{K-1}
&-2\gamma\int_{\R^2}(b\cdot\nabla)\nabla^2 u\cdot(u \cdot \nabla)\nabla^2 bdx\nonumber\\
&\leq -2\gamma^2\frac{d}{dt}\int_{\R^2}\partial_{t}\nabla^2 b\cdot (u\cdot\nabla)\nabla^2 bdx+2\gamma\alpha\int_{\R^2}(u\cdot\nabla)\nabla^2 u\cdot\nabla^2\partial_2 bdx \nonumber\\
&\quad+C(\epsilon+\epsilon^2)(\|\nabla^3b\|_{H^1}^2+\|\nabla^2\partial_2 u\|_{L^2}^2+\|\nabla^2 u_2\|_{H^1}^2+\|\partial_t\nabla^2 b\|_{L^{2}}^2)\nonumber\\
&\quad+C(\|\nabla^2u\|_{L^2}^2+\|\nabla^2 b\|_{L^2}^2)\left(\|\nabla u\|_{H^2}^2+\|\nabla b\|_{H^2}^2\right).
\end{align}

{\it {Step 3.}} In this step, to deal with the second term  on the right-hand side of \eqref{K-1}, we substitute $(u\cdot \nabla) u$ by using equation \eqref{type-wave-1}$_1$, utilize integration by parts and then get
\begin{align}\label{Part III}
&2\gamma\alpha\int_{\R^2}(u\cdot\nabla)\nabla^2 u\cdot\nabla^2\partial_2 bdx\nonumber\\
&=2\gamma\alpha\int_{\R^2}\nabla^2 ((u\cdot\nabla)u)\cdot\nabla^2\partial_2 bdx-2\gamma\alpha\int_{\R^2}(\nabla^2 u\cdot\nabla)u\cdot\nabla^2\partial_2 bdx\nonumber\\
&\quad-4\gamma\alpha\int_{\R^2}(\nabla u\cdot\nabla)\nabla u\cdot\nabla^2\partial_2 bdx\nonumber\\
&=-2\gamma\alpha\frac{d}{dt}\int_{\R^2}\nabla^2 u\cdot\nabla^2\partial_2 bdx-2\gamma\alpha\int_{\R^2}\nabla^2\partial_2u\cdot\partial_t\nabla^2 bdx+2\gamma\alpha\int_{\R^2}\nabla^2((b\cdot\nabla )b)\cdot \nabla^2\partial_2 bdx\nonumber\\
&\quad-2\gamma\alpha\mu\int_{\R^2}\nabla^2(0,u_2)^{T}\cdot\nabla^2\partial_2 bdx+2\gamma\alpha^2\|\nabla^2 \partial_2b\|_{L^2}^2-2\gamma\alpha\int_{\R^2}(\nabla^2 u\cdot\nabla)u\cdot\nabla^2\partial_2 bdx\nonumber\\
&\quad-4\gamma\alpha\int_{\R^2}(\nabla u\cdot\nabla)\nabla u\cdot\nabla^2\partial_2 bdx.
\end{align}

Applying Young's inequality and \eqref{stability}, we infer that
\begin{align}\label{Part III1}
&2\gamma\alpha\int_{\R^2}\nabla^2((b\cdot\nabla )b)\cdot\nabla^2\partial_2 bdx\nonumber\\&
=2\gamma\alpha\int_{\R^2}((\nabla^2 b\cdot\nabla) b+2(\nabla b\cdot\nabla)\nabla b+( b\cdot\nabla)\nabla^2 b)\cdot\nabla^2\partial_2 bdx\nonumber\\
&\leq C\|\nabla^3 b\|_{L^2}\|\nabla b\|_{L^{\infty}}\|\nabla^2 b\|_{L^2}+C\|b\|_{L^{\infty}}\|\nabla^3 b\|_{L^{2}}^2\nonumber\\&\leq C\epsilon\|\nabla^3 b\|_{L^{2}}^2+C\|\nabla^2 b\|_{L^{2}}^2\|\nabla b\|_{H^{2}}^2,
\end{align}
and
 \begin{align}\label{Part III3}
-2\gamma\mu\alpha\int_{\R^2}\nabla^2 (0,u_2)^{T}\cdot\nabla^2\partial_2 bdx
\leq \frac{2\mu}{3}\| \nabla^2 u_2\|_{L^2}^2+\frac{3\gamma^2\alpha^2\mu}{2}\|\nabla^3 b\|_{L^2}^2.
\end{align}

Due to  the anisotropic inequality \eqref{AN}, we have
\begin{align}\label{Part III2}
&-2\gamma\alpha\int_{\R^2}(\nabla^2 u\cdot\nabla)u\cdot\nabla^2\partial_2 bdx-4\gamma\alpha\int_{\R^2}(\nabla u\cdot\nabla)\nabla u\cdot\nabla^2\partial_2 bdx\nonumber\\
&\leq C\|\nabla^3 b\|_{L^2}\|\nabla u\|_{L^2}^{\frac{1}{2}}\|\nabla\partial_1 u\|_{L^2}^{\frac{1}{2}}\|\nabla^2 u\|_{L^2}^{\frac{1}{2}}\|\nabla^2\partial_2 u\|_{L^2}^{\frac{1}{2}}\nonumber\\ &\leq C\|u\|_{H^{2}}(\|\nabla^3 b\|_{L^{2}}^2+\|\nabla^2\partial_2 u\|_{L^{2}}^2+\|\nabla\partial_1 u\|_{L^{2}}^2).
\end{align}

Thus, it follows from \eqref{Part III}-\eqref{Part III2}  and \eqref{stability} that
\begin{align*}
&2\gamma\alpha\int_{\R^2}(u\cdot\nabla)\nabla^2 u\cdot\nabla^2\partial_2 bdx\nonumber\\
&\leq-2\gamma\alpha\frac{d}{dt}\int_{\R^2}\nabla^2 u\cdot\nabla^2\partial_2 bdx-2\gamma\alpha\int_{\R^2}\nabla^2\partial_2u\cdot\partial_t\nabla^2 bdx\nonumber\\
&\quad+ \frac{2\mu}{3}\| \nabla^2 u_2\|_{L^2}^2+(\frac{3\gamma^2\alpha^2\mu}{2}+2\gamma\alpha^2)\|\nabla^3 b\|_{L^2}^2 \nonumber\\
&\quad+C\epsilon(\|\nabla^3b\|_{H^1}^2+\|\nabla^2 \partial_2u\|_{L^2}^2+\|\nabla^2 u_2\|_{H^1}^2)+C\|\nabla^2 b\|_{L^2}^2\|\nabla b\|_{H^2}^2.
\end{align*}
From {\it Step 1-Step 3 }, we have
\begin{align}\label{KKKK}
M_{23}&\leq
-2\gamma\frac{d}{dt}\int_{\R^2}((b\cdot\nabla)\nabla^2 b\cdot\nabla^2 udx+\alpha\nabla^2 u\cdot\nabla^2\partial_2 b+\gamma\partial_{t}\nabla^2 b\cdot (u\cdot\nabla)\nabla^2 b)dx\nonumber\\
&\quad+\frac{2\mu}{3}\| \nabla^2 u_2\|_{L^2}^2+(\frac{3\gamma^2\alpha^2\mu}{2}+2\gamma\alpha^2)\|\nabla^3 b\|_{L^2}^2-2\gamma\alpha\int_{\R^2}\nabla^2\partial_2u\cdot\partial_t\nabla^2 bdx\nonumber\\
&\quad+C(\|\nabla^2u\|_{L^2}^2+\|\nabla^2 b\|_{L^2}^2)\left(\|\nabla u\|_{H^2}^2+\|\nabla b\|_{H^2}^2+\|\partial_tb\|_{H^3}^2\right)\nonumber\\
&\quad+C\left(\epsilon+\epsilon^2\right)\left(\|\nabla^2 u_2\|_{H^{1}}^2+\|\nabla^2\partial_2 u\|_{L^{2}}^2+\|\partial_t\nabla^2 b\|_{H^1}^2+\|\nabla^3 b\|_{H^1}^2\right).
\end{align}

Based on \eqref{GKKK}-\eqref{GGGG} and \eqref{KKKK},  we have
\begin{align}\label{KKKKKK1}
&(1+t)^3(M_{2}+M_{3})\nonumber\\&\leq
-2\gamma(1+t)^3\frac{d}{dt}\int_{\R^2}\big((b\cdot\nabla)\nabla^2 b\cdot\nabla^2 u+\alpha\nabla^2 u\cdot\nabla^2\partial_2 b+\gamma\partial_{t}\nabla^2 b\cdot (u\cdot\nabla)\nabla^2 b\big)dx\nonumber\\
&\quad-2\gamma(1+t)^3\frac{d}{dt}\int_{\R^2}\big((b\cdot\nabla)\nabla^3 b\cdot\nabla^3 u+\alpha\nabla^3 u\cdot\nabla^3\partial_2 b+\gamma\partial_{t}\nabla^3 b\cdot (u\cdot\nabla)\nabla^3 b\big)dx\nonumber\\
&\quad+\frac{2\mu}{3}(1+t)^3\| \nabla^2 u_2\|_{H^1}^2+(\frac{3\gamma^2\alpha^2\mu}{2}+2\gamma\alpha^2)(1+t)^3\|\nabla^3 b\|_{H^1}^2\nonumber\\
&\quad+C(1+t)^3(\|\nabla^2u\|_{H^1}^2+\|\nabla^2 b\|_{H^1}^2+\|\partial_t\nabla^2 b\|_{H^1}^2)\left(\|\nabla u\|_{H^2}^2+\|\nabla b\|_{H^2}^2+\|\partial_tb\|_{H^2}^2\right)\nonumber\\
&\quad+C\left(\epsilon+\epsilon^2\right)(1+t)^3\left(\|\nabla^2 u_2\|_{H^{1}}^2+\|\nabla^2\partial_2 u\|_{L^{2}}^2+\|\partial_t\nabla^2 b\|_{H^1}^2+\|\nabla^3 b\|_{H^1}^2\right).
\end{align}

Next, we deal with the first two terms on the right side of \eqref{KKKKKK1}, a direct calculation yields
\begin{align}\label{T1}
&-2\gamma(1+t)^3\frac{d}{dt}\int_{\R^2}\big((b\cdot\nabla)\nabla^2 b\cdot\nabla^2 u+\alpha\nabla^2 u\cdot\nabla^2\partial_2b+\gamma\partial_{t}\nabla^2 b\cdot (u\cdot\nabla)\nabla^2 b\big)dx\nonumber\\
&=-2\gamma\frac{d}{dt}(1+t)^3\int_{\R^2}\big((b\cdot\nabla)\nabla^2 b\cdot\nabla^2 u+\alpha\nabla^2 u\cdot\nabla^2\partial_2 b+\gamma\partial_{t}\nabla^2 b\cdot (u\cdot\nabla)\nabla^2 b\big)dx\nonumber\\
 &\quad+6\gamma(1+t)^2\int_{\R^2}\big((b\cdot\nabla)\nabla^2 b\cdot\nabla^2 u+\alpha\nabla^2 u\cdot\nabla^2\partial_2b+\gamma\partial_{t}\nabla^2 b\cdot (u\cdot\nabla)\nabla^2 b\big)dx
\end{align}
and
\begin{align}\label{T2}
&-2\gamma(1+t)^3\frac{d}{dt}\int_{\R^2}\big((b\cdot\nabla)\nabla^3 b\cdot\nabla^3 u+\alpha\nabla^3 u\cdot\nabla^3\partial_2 b+\gamma\partial_{t}\nabla^3 b\cdot (u\cdot\nabla)\nabla^3 b\big)dx\nonumber\\
&=-2\gamma\frac{d}{dt}(1+t)^3\int_{\R^2}\big((b\cdot\nabla)\nabla^3 b\cdot\nabla^3 u+\alpha\nabla^3 u\cdot\nabla^3\partial_2 b+\gamma\partial_{t}\nabla^3 b\cdot (u\cdot\nabla)\nabla^3 b\big)dx\nonumber\\
 &\quad+6\gamma(1+t)^2\int_{\R^2}\big((b\cdot\nabla)\nabla^3 b\cdot\nabla^3 u+\alpha\nabla^3 u\cdot\nabla^3\partial_2b+\gamma\partial_{t}\nabla^3 b\cdot (u\cdot\nabla)\nabla^3 b\big)dx.
\end{align}
Noticing that by using Young's inequality, Sobolev's inequality and \eqref{stability}, we can obtain
\begin{align}\label{T3}
&6\gamma(1+t)^2\int_{\R^2}\big((b\cdot\nabla)\nabla^2 b\cdot\nabla^2 u+\alpha\nabla^2 u\cdot\nabla^2\partial_2 b+\gamma\partial_{t}\nabla^2 b\cdot (u\cdot\nabla)\nabla^2 b\big)dx\nonumber\\
&\leq6\gamma(1+t)^2\|b\|_{L^\infty}\|\nabla^3 b\|_{L^2}\|\nabla^2 u\|_{L^2} +6\gamma\alpha\left((1+t)\|\nabla^2 u\|_{L^2}^2\right)^{\frac{1}{2}}\left((1+t)^3\|\nabla^3 b\|_{L^2}^2\right)^{\frac{1}{2}}\nonumber\\
&\quad+6\gamma^2(1+t)^2\|u\|_{L^\infty}\|\nabla^3 b\|_{L^2}\|\partial_t\nabla^2 b\|_{L^2} \nonumber\\
&\leq C\epsilon(1+t)^3(\|\nabla^3 b\|_{L^2}^2+\|\partial_t\nabla^2 b\|_{L^2}^2) +\frac{C}{4\epsilon}(1+t)\|\nabla^2u\|_{L^2}^2+\frac{C}{4\epsilon}(1+t)^3\|\nabla^2 u\|_{L^2}^2\|b\|_{H^2}^2
\end{align}
and
\begin{align}\label{T4}
&6\gamma(1+t)^2\int_{\R^2}\big((b\cdot\nabla)\nabla^3 b\cdot\nabla^3 u+\alpha\nabla^3 u\cdot\nabla^3\partial_2 b+\gamma\partial_{t}\nabla^3 b\cdot (u\cdot\nabla)\nabla^3 b\big)dx\nonumber\\
&\leq6\gamma(1+t)^2\|b\|_{L^\infty}\|\nabla^3 b\|_{H^1}\|\nabla^3 u\|_{L^2} +6\gamma\alpha\left((1+t)\|\nabla^3 u\|_{L^2}^2\right)^{\frac{1}{2}}\left((1+t)^3\|\nabla^3 b\|_{H^1}^2\right)^{\frac{1}{2}}\nonumber\\
&\quad+6\gamma^2(1+t)^2\|u\|_{L^\infty}\|\nabla^3 b\|_{H^1}\|\partial_t\nabla^2 b\|_{H^1} \nonumber\\
&\leq C\epsilon(1+t)^3(\|\nabla^3 u\|_{L^2}^2+\|\nabla^3 b\|_{H^1}^2+\|\partial_t\nabla^2 b\|_{H^1}^2) +\frac{C}{4\epsilon}(1+t)\|\nabla^3u\|_{L^2}^2.
\end{align}

Putting \eqref{T1}-\eqref{T4} into \eqref{KKKKKK1} immediately implies
\begin{align*}
&(1+t)^3(M_{2}+M_{3})\nonumber\\&\leq
-2\gamma\frac{d}{dt}(1+t)^3\sum_{m=2}^3\int_{\mathbb{R}^2}
		\big((b\cdot\nabla)\nabla^m b\cdot\nabla^m u
		+\alpha\nabla^m u\cdot\nabla^m\partial_2 b
		+\gamma\nabla^m\partial_t b\cdot (u\cdot\nabla)\nabla^m b\big)dx\nonumber\\
&\quad+\frac{2\mu}{3}(1+t)^3\| \nabla^2 u_2\|_{H^1}^2+(\frac{3\gamma^2\alpha^2\mu}{2}+2\gamma\alpha^2)(1+t)^3\|\nabla^3 b\|_{H^1}^2\nonumber\\
&\quad+C\left(\epsilon+\epsilon^2\right)(1+t)^3\left(\|\nabla^2 u_2\|_{H^{1}}^2+\|\nabla^2\partial_2 u\|_{L^{2}}^2+\|\partial_t\nabla^2 b\|_{H^1}^2+\|\nabla^3 b\|_{H^1}^2\right)+\frac{C}{4\epsilon}(1+t)\|\nabla u\|_{H^2}^2\nonumber\\&\quad+C(1+t)^3(\|\nabla^2u\|_{H^1}^2+\|\nabla^2 b\|_{H^1}^2+\|\nabla^2\partial_t b\|_{H^1}^2)\left(\|\nabla u\|_{H^2}^2+\|b\|_{H^3}^2+\|\partial_tb\|_{H^2}^2\right).
\end{align*}

For $M_{4}$, it follows similarly from $M_{21}$ that
\begin{align*}
M_{4}&\leq
C\epsilon(\|\nabla^2 u_2\|_{H^{1}}^2+\|\nabla^2\partial_2 u\|_{L^{2}}^2+\|\partial_t\nabla^2 b\|_{H^1}^2+\|\nabla^3 b\|_{H^1}^2)+C\|\nabla^2 u\|_{H^{1}}^2\|\nabla b\|_{H^{2}}^2.
\end{align*}

For $M_5$, it follows from Lemma \ref{2.1} and Young's inequality that
\begin{align*}
M_5&=-\int_{\R^2}\big((\nabla^2u\cdot\nabla) u+2(\nabla u\cdot\nabla)\nabla u\big)\cdot\nabla^2udx\nonumber\\&\quad-\int_{\R^2}\big((\nabla^3u\cdot\nabla) u+3(\nabla^2u\cdot\nabla)\nabla u+3(\nabla u\cdot\nabla)\nabla^2 u\big)\cdot\nabla^3udx
\nonumber\\&\leq C\|\nabla^2u\|_{L^2}^{\frac{1}{2}}\|\nabla^2\partial_{2}u\|_{L^2}^{\frac{1}{2}}\|\nabla u\|_{L^2}^{\frac{1}{2}}\|\nabla\partial_{1}u\|_{L^2}^{\frac{1}{2}}\|\nabla^2u\|_{L^2}\nonumber\\&\quad+C\|\nabla^2u\|_{L^2}^{\frac{1}{2}}\|\nabla^2\partial_{1}u\|_{L^2}^{\frac{1}{2}}
\|\nabla^2 u\|_{L^2}^{\frac{1}{2}}\|\nabla^2\partial_{2} u\|_{L^2}^{\frac{1}{2}}\|\nabla^3u\|_{L^2}+C\|\nabla u\|_{L^{\infty}}\|\nabla^3u\|_{L^2}^2\nonumber\\&
\leq C\epsilon(\|\nabla^2 u_2\|_{H^{1}}^2+\|\nabla^2\partial_2 u\|_{L^{2}}^2)+C\|\nabla^2 u\|_{H^{1}}^2\|\nabla u\|_{H^{2}}^2.
\end{align*}

By integration by parts and $\nabla\cdot b=0$, we find
\begin{align*}
M_{6}+M_{7}=M_{61}+M_{62},
\end{align*}
where \begin{align*}
M_{61}&:=
\int_{\R^2}(\nabla^2b\cdot\nabla) b\cdot\nabla^2udx+
2\int_{\R^2}(\nabla b\cdot\nabla)\nabla b\cdot\nabla^2udx
\nonumber\\
&\quad+\int_{\R^2}(\nabla^2b\cdot\nabla) u\cdot\nabla^2bdx+
2\int_{\R^2}(\nabla b\cdot\nabla)\nabla u\cdot\nabla^2bdx\\
&\leq
C\|\nabla^2 b\|_{L^2}^{\frac{1}{2}}\|\nabla^2\partial_1b\|_{L^2}^{\frac{1}{2}}\|\nabla^2 u\|_{L^2}^{\frac{1}{2}}
\|\nabla^2\partial_2u\|_{L^2}^{\frac{1}{2}}\|\nabla b\|_{L^2}\nonumber\\
&\quad+
C\|\nabla^2 b\|_{L^2}^{\frac{1}{2}}\|\nabla^2\partial_1b\|_{L^2}^{\frac{1}{2}}\|\nabla^2 b\|_{L^2}^{\frac{1}{2}}
\|\nabla^2\partial_2b\|_{L^2}^{\frac{1}{2}}\|\nabla u\|_{L^2}\nonumber\\
&\leq
C\epsilon(\|\nabla^3b\|_{H^1}^2+\|\nabla^2\partial_2u\|_{L^2}^2)+C(\|\nabla^2 u\|_{L^2}^2+\|\nabla^2 b\|_{L^2}^2)(\|\nabla u\|_{L^2}^2+\|\nabla b\|_{L^2}^2)
\end{align*}
and
\begin{align*}
M_{62}&:=
3\int_{\R^2}(\nabla b\cdot\nabla)\nabla^2b\cdot\nabla^3udx+
3\int_{\R^2}(\nabla^2b\cdot\nabla)\nabla b\cdot\nabla^3udx
+\int_{\R^2}(\nabla^3b\cdot\nabla) b\cdot\nabla^3udx\nonumber\\
&\quad+
3\int_{\R^2}(\nabla b\cdot\nabla)\nabla^2u\cdot\nabla^3bdx+
3\int_{\R^2}(\nabla^2b\cdot\nabla)\nabla u\cdot\nabla^3bdx
+\int_{\R^2}(\nabla^3b\cdot\nabla) u\cdot\nabla^3bdx\nonumber\\
&\leq C\|\nabla b\|_{L^\infty}\|\nabla^3b\|_{L^2}\|\nabla^3u\|_{L^2}+
C\|\nabla^2 b\|_{L^2}^{\frac{1}{2}}\|\nabla^2\partial_1b\|_{L^2}^{\frac{1}{2}}\|\nabla^2b\|_{L^2}^{\frac{1}{2}}
\|\nabla^2\partial_2b\|_{L^2}^{\frac{1}{2}}\|\nabla^3u\|_{L^2}\nonumber\\
&\quad+C\|\nabla^2 b\|_{L^2}^{\frac{1}{2}}\|\nabla^2\partial_1b\|_{L^2}^{\frac{1}{2}}\|\nabla^2u\|_{L^2}^{\frac{1}{2}}
\|\nabla^2\partial_2u\|_{L^2}^{\frac{1}{2}}\|\nabla^3b\|_{L^2}+C\|\nabla u\|_{L^\infty}\|\nabla^3b\|_{L^2}^2\nonumber\\
&\leq
C\|(u,b)\|_{H^3}(\|\nabla^3b\|_{H^1}^2+\|\nabla^2u_2\|_{H^1}^2+\|\partial_2\nabla^2u\|_{L^2}^2).
\end{align*}
According to \eqref{stability}, we then have
\begin{align*}
M_{6}+M_7
&\leq C\epsilon(\|\nabla^3b\|_{H^1}^2+\|\nabla^2u_2\|_{H^1}^2+\|\nabla^2\partial_2u\|_{L^2}^2)\nonumber\\
&\quad+C(\|\nabla^2 u\|_{L^2}^2+\|\nabla^2 b\|_{L^2}^2)(\|\nabla b\|_{L^2}^2+\|\nabla u\|_{L^2}^2).
\end{align*}

Similarly, there has
\begin{align}\label{H-4}
M_8&\leq C\epsilon(\|\nabla^3b\|_{H^1}^2+\|\nabla^2u_2\|_{H^1}^2+\|\nabla^2\partial_2u\|_{L^2}^2)\nonumber\\
&\quad+C(\|\nabla^2 u\|_{L^2}^2+\|\nabla^2 b\|_{L^2}^2)(\|\nabla b\|_{L^2}^2+\|\nabla u\|_{L^2}^2).
\end{align}

Collecting all the estimates of $M_1$ in \eqref{m1-0} through $M_8$ in \eqref{H-4}, we conclude that \eqref{boundE} holds.
\end{proof}

The following lemma focuses on $$(1+t)^3\|\nabla^2\partial_2u\|_{L^2}^2.$$
\begin{Lem}\label{keyy1}
Assume that $(u,b)$ is a smooth solution to \eqref{type-wave-1}, then we have
\begin{align}\label{keyE}
&-\gamma\frac{d}{dt}(1+t)^3(\partial_t\nabla^2b,\nabla^2\partial_2u)+(\frac{\alpha}{2}-C\epsilon)(1+t)^3\|\nabla^2\partial_2u\|_{L^2}^2\nonumber\\
& \leq(1+t)^3\Big((\frac{27\gamma^2}{2\alpha}+\frac{\gamma(\mu+\alpha)}{2}+\frac{3}{2\alpha}+C\epsilon)\|\partial_t \nabla^2 b\|_{L^2}^2+ (\frac{3\eta^2}{2\alpha}+\frac{\gamma\alpha}{2}+C\epsilon)\|\nabla^3 b\|_{H^1}^2\nonumber\\
&\quad\qquad\qquad+(\frac{\gamma\mu}{2}+C\epsilon)\|\nabla^2u_2\|_{H^1}^2+\frac{1}{\epsilon}\|\nabla^2 u\|_{H^{1}}^2\|\nabla b\|_{H^{1}}^2\Big),
\end{align}
\end{Lem}
 \begin{proof}
 Invoking the equation of  \eqref{type-wave-1}$_2$ and \eqref{stability}, we have
 \begin{align*}\label{}
&-\gamma\frac{d}{dt}(1+t)^3(\partial_t\nabla^2b,\nabla^2\partial_2u)+\alpha(1+t)^3\|\nabla^2\partial_2u\|_{L^2}^2\nonumber\\
&=-3\gamma(1+t)^2(\partial_t\nabla^2b,\nabla^2\partial_2u)-\gamma(1+t)^3(\partial_t\nabla^2b,\partial_t\nabla^2\partial_2u)+(1+t)^3(\partial_t\nabla^2b,\nabla^2\partial_2u)\nonumber\\
&\quad-\eta(1+t)^3(\Delta\nabla^2 b,\nabla^2\partial_2u)+(1+t)^3(\nabla^2(u\cdot\nabla b),\partial_2\nabla^2u)-(1+t)^3(\nabla^2(b\cdot\nabla u),\nabla^2\partial_2u)
\nonumber\\
&=-3\gamma(1+t)^2(\partial_t\nabla^2b,\nabla^2\partial_2u)+(1+t)^3J_1+(1+t)^3(\partial_t\nabla^2b,\nabla^2\partial_2u)\nonumber\\
&\quad-\eta(1+t)^3(\Delta\nabla^2 b,\nabla^2\partial_2u)+(1+t)^3(J_3+J_5)
\nonumber\\
& \leq(1+t)^3\Big(\frac{27\gamma^2}{2\alpha}\|\partial_t \nabla^2 b\|_{L^2}^2+\frac{\alpha}{6}\|\nabla^2 \partial_2u\|_{L^2}^2+\frac{3}{2\alpha}\|\partial_t\nabla^2b\|_{H^1}^2+\frac{\alpha}{6}\|\nabla^2\partial_2u\|_{L^2}^2,\\
&\quad\qquad\qquad+ \frac{3\eta^2}{2\alpha}\|\nabla^3 b\|_{H^1}^2+\frac{\alpha}{6}\|\nabla^2\partial_2u\|_{L^2}^2+\frac{\gamma}{2}\left((\mu+\alpha) \|\partial_t\nabla^2b\|_{H^1}^2+\alpha \|\nabla^3 b\|_{H^1}^2+\mu \|\nabla^2u_2\|_{H^1}^2\right)\nonumber\\
&\quad\qquad\qquad+C\epsilon\left(\|\partial_t\nabla^2 b\|_{H^{1}}^2+\|\nabla^3 b\|_{H^{1}}^2+\|\nabla^2\partial_2 u\|_{L^{2}}^2+\|\nabla^2u_2\|_{H^{1}}^2\right)+\frac{1}{\epsilon}\|\nabla^2 u\|_{H^{1}}^2\|\nabla b\|_{H^{1}}^2\Big)\nonumber\\
& \leq(1+t)^3\Big((\frac{27\gamma^2}{2\alpha}+\frac{\gamma(\mu+\alpha)}{2}+\frac{3}{2\alpha}+C\epsilon)\|\partial_t \nabla^2 b\|_{L^2}^2+ (\frac{3\eta^2}{2\alpha}+\frac{\gamma\alpha}{2}+C\epsilon)\|\nabla^3 b\|_{H^1}^2\\
&\quad\qquad\qquad+(\frac{\alpha}{2}+C\epsilon)\|\nabla^2\partial_2 u\|_{L^2}^2+(\frac{\gamma\mu}{2}+C\epsilon)\|\nabla^2u_2\|_{H^1}^2+\frac{1}{\epsilon}\|\nabla^2 u\|_{H^{1}}^2\|\nabla b\|_{H^{1}}^2\Big),
\end{align*}
where we have used the estimates of  $J_1$  given in \eqref{JJJ1}, and
\begin{align*}
J_{3}+J_{5}
&\leq C\|\nabla^2 \partial_2u\|_{L^{2}}\|\nabla^2u\|_{L^{2}}^{\frac{1}{2}}\|\nabla^2\partial_1 u\|_{L^{2}}^{\frac{1}{2}}\|\nabla b\|_{L^{2}}^{\frac{1}{2}}\|\nabla\partial_2 b\|_{L^{2}}^{\frac{1}{2}}+C\|u\|_{L^{\infty}}\|\nabla^3 b\|_{L^{2}}\|\nabla^2\partial_2 u\|_{L^{2}}\nonumber\\
&\quad+C\|\nabla^2 \partial_2u\|_{L^{2}}\|\nabla^2b\|_{L^{2}}^{\frac{1}{2}}\|\nabla^2\partial_2 b\|_{L^{2}}^{\frac{1}{2}}\|\nabla u\|_{L^{2}}^{\frac{1}{2}}\|\nabla\partial_1 u\|_{L^{2}}^{\frac{1}{2}}
+C\|b\|_{L^{\infty}}\|\nabla^3 u\|_{L^{2}}^2\nonumber\\
&\leq C\epsilon\|\nabla^2 \partial_2u\|_{L^{2}}^2+\frac{1}{\epsilon}\|\nabla^2 u\|_{H^{1}}^2\|\nabla b\|_{H^{1}}^2+\|(u,b)\|_{H^{2}}\left(\|\nabla^3 b\|_{L^{2}}^2+\|\nabla^3 u\|_{L^{2}}^2\right)\nonumber\\
&\quad+C(\|\nabla u\|_{L^{2}}+\|\nabla ^2b\|_{L^{2}})\left(\|\nabla^3 b\|_{L^{2}}^2+\|\nabla^2\partial_2 u\|_{L^{2}}^2+\|\nabla\partial_1u\|_{L^{2}}^2\right)
\nonumber\\
&\leq C\epsilon\left(\|\nabla^3b\|_{L^{2}}^2+\|\nabla^2\partial_2 u\|_{L^{2}}^2+\|\nabla^2u_2\|_{H^{1}}^2\right)+\frac{1}{\epsilon}\|\nabla^2 u\|_{H^{1}}^2\|\nabla b\|_{H^{1}}^2.
\end{align*}

Then we get the desired estimate \eqref{keyE}.
 \end{proof}

 Now we complete the proof of   Proposition \ref{Prop2} based on the above two lemmas.

 \textit{Proof of Proposition} \ref{Prop2}. With Lemma \ref{keyy} and \ref{keyy1} at disposal, performing \eqref{boundE}+$\kappa_3\cdot$\eqref{keyE} yields
\begin{align*}
&\frac{1}{2}\frac{d}{dt}(1+t)^3\bigg(\|\nabla^2 u\|^2_{H^1}+\|\nabla^2 b\|^2_{H^1}+2\gamma^2\|\partial_t\nabla^2b\|^2_{H^1}
+2\gamma\eta\|\nabla^{3}b\|^2_{H^1}\nonumber\\
&\qquad\qquad+2\gamma(\partial_t\nabla^2b,\nabla^2b)_{H^1}-2\gamma\kappa_3(\partial_t\nabla^2b,\nabla^2\partial_2u)\nonumber\\
&\qquad\qquad+4\gamma\sum_{m=2}^3\int_{\mathbb{R}^2}
		\big((b\cdot\nabla)\nabla^m b\cdot\nabla^m u
		+\alpha\nabla^m u\cdot\nabla^m\partial_2 b
		+\gamma\partial_t\nabla^m b\cdot (u\cdot\nabla)\nabla^m b\big)dx\bigg)\nonumber\\       &+(1+t)^3\Big((\frac{\mu}{3}-C\epsilon-C\epsilon^2-\kappa_3(\frac{\gamma\mu}{2}+C\epsilon))\|\nabla^{2}u_2\|^2_{H^1}\nonumber\\
&\qquad\ \ \ \ \ \ \ \ +\big(\eta-2\gamma\alpha^2-\frac{3\gamma^2\alpha^2\mu}{2}-C\epsilon-C\epsilon^2-\kappa_3(\frac{3\eta^2}{2\alpha}+\frac{\gamma\alpha}{2}+C\epsilon)\big)\|\nabla^{3}b\|^2_{H^1}\nonumber\\
&\qquad\ \ \ \ \ \ \ \ +\big(\gamma-C\epsilon-C\epsilon^2-\kappa_3(\frac{27\gamma^2}{2\alpha}+\frac{\gamma(\mu+\alpha)}{2}+\frac{3}{2\alpha}+C\epsilon)\big)\|\partial_t\nabla^2b\|^2_{H^1}\nonumber\\
&\qquad\ \ \ \ \ \ \ \ +(\kappa_3(\frac{\alpha}{2}-C\epsilon)- C(\epsilon+\epsilon^2))\|\nabla^2\partial_2u\|_{L^2}^2\Big)\nonumber\\
&\leq C(1+t)^3(\|\nabla^2 u\|_{H^1}^2+\|\nabla^2 b\|_{H^1}^2+\|\partial_t\nabla^2 b\|_{H^{1}}^2)(\| b\|_{H^3}^2+\| u_2\|_{H^3}^2+\| \partial_2u\|_{H^2}^2+\|\partial_t b\|_{H^3}^2)\nonumber\\
&\quad+\frac{\kappa_3}{\epsilon}\|\nabla^2 u\|_{H^{1}}^2\|\nabla b\|_{H^{1}}^2+\frac{C}{\epsilon}(1+t)(\|u_2\|_{H^3}^2+\|\partial_2 u\|_{H^2}^2+\|\nabla b\|_{H^3}^2+\|\partial_t b\|_{H^3}^2).
\end{align*}
Here $\kappa_3$ is a  parameter, and $\alpha, \kappa_3$, $\epsilon$ are sufficiently small.

Noting that
\begin{align*}
2\gamma(\partial_t\nabla^2 b,\nabla^2 b)_{H^1}
&\leq \frac{2}{3}\|\nabla^2 b\|_{H^1}^2+\frac{3\gamma^2}{2}\|\partial_t\nabla^2 b\|_{H^1}^2,\\
4\gamma\alpha\int_{\R^2}\nabla^2u\cdot\nabla^2\partial_2bdx
&\leq \frac{2}{3}\|\nabla^2 u\|_{L^2}+6\gamma^2\alpha^2\|\nabla^3 b\|_{L^2}^2,\\
4\gamma\alpha\int_{\R^2}\nabla^3u\cdot\nabla^3\partial_2bdx
&\leq \frac{2}{3}\|\nabla^3 u\|_{L^2}+6\gamma^2\alpha^2\|\nabla^3 b\|_{L^2}^2,\\
-2\kappa_3\gamma(\partial_t\nabla^2 b,\partial_2\nabla^2 u)&\leq \kappa_3\gamma\big(\|\partial_t\nabla^2 b\|_{L^2}^2+\|\nabla^3 u\|_{H^1}^2\big),
\end{align*}
and
 \begin{align*}
&4\gamma\int_{\R^2}\left((b\cdot\nabla)\nabla^2 b\cdot\nabla^2 u+\gamma\partial_{t}\nabla^2b\cdot (u\cdot\nabla)\nabla^2 b\right)dx\nonumber\\
&+4\gamma\int_{\R^2}\left((b\cdot\nabla)\nabla^3 b\cdot\nabla^3 u+\gamma\partial_{t}\nabla^3b\cdot (u\cdot\nabla)\nabla^3b\right)dx\nonumber\\
&\leq C\epsilon(\|\nabla^2 u\|_{H^1}^2+\|\nabla^2 b\|_{H^2}^2+\|\partial_t\nabla^2 b\|_{H^1}^2).
\end{align*}

And in view of \eqref{stability}, \eqref{4.1}, \eqref{key2} and Gronwall's inequality, we then arrive at
\begin{align*}
&(1+t)^3\Big(\|\nabla^2 u\|^2_{H^1}+\|\nabla^2 b\|^2_{H^1}+2\gamma^2\|\partial_t\nabla^2 b\|^2_{H^1}
+2\gamma\eta\|\nabla^3 b\|^2_{H^1}\nonumber\\
&\qquad\quad+2\gamma(\partial_t\nabla^2 b,\nabla^2 b)_{H^1}-4\kappa_3\gamma(\partial_t\nabla^2 b,\partial_2\nabla^2 u)\nonumber\\
&\qquad\quad+4\gamma\sum_{m=2}^3\int_{\R^2}\big((b\cdot\nabla)\nabla^m b\cdot\nabla^m u+\gamma\partial_{t}\nabla^mb\cdot (u\cdot\nabla)\nabla^m b+\alpha\nabla^mu\cdot\nabla^m\partial_2 b\big)dx\Big)\nonumber\\
&+2\int_{0}^{t}(1+\tau)^3\big((\frac{\mu}{3}-C(\epsilon+\epsilon^2)-\kappa_3(\frac{\gamma\mu}{2}+C\epsilon))\|\nabla^2 u_2(\tau)\|^2_{H^1}\nonumber\\
&\ \ \ \ \ \ \ \ +(\eta-2\gamma\alpha^2-\frac{3\gamma^2\alpha^2\mu}{2}-C(\epsilon+\epsilon^2)-\kappa_3(\frac{3\eta^2}{2\alpha}+\frac{\gamma\alpha}{2}+C\epsilon)\|\nabla ^3b(\tau)\|^2_{H^1}\nonumber\\
&\ \ \ \ \ \ \ \ +(\gamma-C(\epsilon+\epsilon^2)-\kappa_3(\frac{27\gamma^2}{2\alpha}+\frac{\gamma(\mu+\alpha)}{2}+\frac{3}{2\alpha}+C\epsilon)\|\partial_{\tau}\nabla^2 b(\tau)\|^2_{H^1}\nonumber\\
&
\ \ \ \ \ \ \ \ +(\kappa_3(\frac{\alpha}{2}-C\epsilon)-C\epsilon-C\epsilon^2)\|\nabla \partial_2u(\tau)\|_{H^1}^2\big)d\tau\nonumber\\
&\leq C\int_{0}^{t}(1+\tau)\left(\| u_2(\tau)\|^2_{H^3}+\|\partial_2 u(\tau)\|^2_{H^3}+\|\partial_{\tau} b(\tau)\|^2_{H^3}
+\|\nabla b(\tau)\|^2_{H^3}\right)d\tau+C.
\end{align*}
Then we can get \eqref{as3}.
\hfill$\square$
\subsection{The decay rates for $\|\nabla u(t)\|_{H^{2}},\|\nabla b(t)\|_{H^{3}}$ and $
		\|\partial_t\nabla b(t)\|_{H^{2}}$}

By the Gagliardo-Nirenberg interpolation inequality, this subsection proves the upper bound for $E_1(t)$, then the decay rates for $\|\nabla u(t)\|_{H^{2}},\|\nabla b(t)\|_{H^{3}}$ and $\|\partial_t\nabla b(t)\|_{H^{2}}$ then follow directly.
\begin{Prop}\label{key51}
For a constant $C>0$, it holds that
\begin{align}\label{key511}
		E_1(t)
		\leq  CE_0^{\frac{1}{2}}(t)E_2^{\frac{1}{2}}(t).
\end{align}
\end{Prop}

\begin{proof}
By the Gagliardo-Nirenberg interpolation inequality, we have
\begin{align*}
\|\nabla f\|_{L^2}\leq C\| f\|_{L^2}^{\frac{1}{2}}\|\nabla^{2}f\|_{L^2}^{\frac{1}{2}},
\end{align*}
which implies
\begin{align*}
&(1+t)^2\|\nabla u\|_{H^2}^2\leq C\left((1+t)\| u\|_{H^3}^2\right)^{\frac{1}{2}}\left((1+t)^3\|\nabla^{2}u\|_{H^1}^2\right)^{\frac{1}{2}},\nonumber\\
&(1+t)^2\|\nabla b\|_{H^3}^2\leq C\left((1+t)\| b\|_{H^4}^2\right)^{\frac{1}{2}}\left((1+t)^3\|\nabla^{2}b\|_{H^2}^2\right)^{\frac{1}{2}},
\end{align*}
and
\begin{align*}
\int_0^t(1+\tau)^2\|\partial_2\nabla u(\tau)\|_{H^1}^2d\tau
\leq C\left(\int_0^t(1+\tau)\|\partial_2 u(\tau)\|_{H^2}^2d\tau\right)^{\frac{1}{2}}\left(\int_0^t(1+\tau)^3\|\partial_2\nabla u(\tau)\|_{H^1}^2d\tau\right)^{\frac{1}{2}}.
\end{align*}

Finally, \eqref{key511} follows from the above inequalities immediately.
\end{proof}

\subsection{The decay rates for $\|\nabla^3 u_2(t)\|_{L^{2}},\|\nabla^3 b_2(t)\|_{H^{1}}$ and $
		\|\partial_t\nabla^3 b_2(t)\|_{L^{2}}$}
\medskip
In the last subsection, we establish $a$ $priori$ estimate for $E_3(t)$ as follows and then get directly the desired decay rates.
\begin{Prop}\label{key5.1}
Assume that $(u,b)$ is a smooth solution to \eqref{type-wave-1}, then we have
\begin{align*}
	E_3(t)\leq C,
\end{align*}
\end{Prop}
\begin{proof}
Proposition \ref{key5.1} follows by an argument similar to that used for Proposition  \ref{Prop2}, the details are omitted.
\end{proof}

\medskip

\subsection*{Acknowledgements.}
Chen was partially supported by the National Natural Science Foundation of China (No. 12501294) and Start-up funds for doctoral research of Anhui Normal University (No. 762350). Fei was partially supported by the National Science Foundation of China (Grant No.12271004, No. 12471222) and the Natural Science Foundation of Anhui Province of China (Grant No. 2308085J10). Lin was partially supported by the National Natural Science Foundation of China (NNSFC) (Grant No. 12571246) and the Natural Science Foundation of Sichuan Province (Grant No. 2023NSFSC0056). Wu was partially supported by the
National Science Foundation of the United States (Grant No. DMS 2309748).
\subsection*{Data availability}
No data was used for the research described in the article.

\end{document}